\documentclass[11pt]{article}
\usepackage{epsfig}
\usepackage{amssymb,amsmath,amsthm,amscd}
\usepackage{latexsym}
\usepackage{amsmath}

\usepackage{amsmath}
\numberwithin{equation}{section}

\newtheorem{thm}{Theorem}[section]
\newtheorem{prop}[thm]{Proposition}

\newtheorem{lem}[thm]{Lemma}

\theoremstyle{definition}
\newtheorem{defn}[thm]{Definition}

\newcommand{\be}{\begin{equation}}
\newcommand{\ee}{\end{equation}}

\newcommand{\R}{\mathbb{R}}
\newcommand{\N}{\mathbb{N}}
\newcommand{\E}{\mathbb{E}}

\begin{document}

\baselineskip=1.2\baselineskip

\pagestyle{plain}
\title{Pullback measure attractors for non-autonomous
stochastic FitzHugh-Nagumo system  with distribution dependence on unbounded domains
 \footnote{This work was supported
 by NSFC (12371178) and Natural Science Foundation of Sichuan province under grant 2023NSFSC1342. }}

\author{{Ruiyan Hu,  Dingshi Li, Tianhao Zeng\footnote{Corresponding authors:
zengtianhao123@my.swjtu.edu.cn (T. Zeng).}}
 \\{ \small\textsl{ School of Mathematics, Southwest Jiaotong University, Chengdu, 610031, P. R. China}}}

\date{}
\maketitle

{ \bf Abstract}  \,This paper is primarily focused on the asymptotic dynamics of a non-autonomous stochastic FitzHugh-Nagumo system with distribution dependence, specifically on unbounded domains  $\mathbb{R}^{n}$. Initially, we establish the well-posedness of solutions for the FitzHugh-Nagumo system with distribution dependence by utilizing the Banach fixed-point theorem. Subsequently, we demonstrate the existence and uniqueness of pullback measure attractors for this system through the application of splitting techniques, tail-end estimates and Vitali's theorem.

{\bf Keywords} FitzHugh-Nagumo system;  distribution dependence; pullback measure attractor.

 {\bf MSC 2010.} Primary  37L55; Secondary 34F05, 37L30, 60H10

\section{Introduction}
\setcounter{equation}{0}
In this paper, we study the asymptotic behaviors of the following non-autonomous stochastic Fitzhugh-Nagumo system with distribution dependence
 driven by nonlinear noise on  $\mathbb{R}^{n}$:
\begin{equation}\label{a1}
\left\{ \begin{array}{l}
  du\left( t \right)-\Delta u\left( t \right)dt+\lambda u\left( t \right)dt+\alpha v\left( t \right)dt+{{f}}\left( t,x,u,{{\mathcal{L}}_{u\left( t \right)}} \right)dt
  ={{G}_{1}}\left( t,x,u,{{\mathcal{L}}_{u\left( t \right)}} \right)dt \\
  +\sum\limits_{k=1}^{\infty }{\left( {{\theta }_{1,k}}\left( t,x \right)
  +w\left( x \right){{\sigma }_{k}}\left( t,u,{{\mathcal{L}}_{u\left( t \right)}} \right) \right)}d{{W}_{k}}\left( t \right),\quad x\in {{\mathbb{R}}^{n}},\ t>\tau , \\
  dv\left( t \right)+\gamma v\left( t \right)dt-\beta u\left( t \right)dt= {{G}_{2}}\left( t,x \right)dt \\
  +\sum\limits_{k=1}^{\infty }{\left( {{\theta }_{2,k}}\left( t,x \right)
  +{{\delta }_{k}}v\left( t \right) \right)}d{{W}_{k}}\left( t \right),\quad x\in {{\mathbb{R}}^{n}},\ t>\tau ,
\end{array} \right.
\end{equation}
with initial condition
\begin{align}\label{a2}
	u\left( \tau ,x \right)={{\xi }_{1}},\ v\left( \tau ,x \right)={{\xi }_{2}},\quad x\in {{\mathbb{R}}^{n}},
\end{align}
where $\tau \in \mathbb{R}$,  $\lambda,\alpha, \beta$ and $\gamma$ are positive constants and satisfy $\gamma>\lambda$, $f$ is a nonlinear function with arbitrary growth rate,  ${{\mathcal{L}}_{u\left( t \right)}}$ is the distribution of $u\left( t \right)$, $G_{1}$ is a Lipschitz function, ${{G}_{2}}\in {{L}^{\infty }}\left( \mathbb{R},{{H}^{1}}\left( {{\mathbb{R}}^{n}} \right)\cap {{L}^{4}}\left( {{\mathbb{R}}^{n}} \right) \right)$, ${{\theta }_{i,k}}:\mathbb{R}\to {{L}^{2}}\left( {{\mathbb{R}}^{n}} \right)$ for $i=1,2$, are given, $w\in {{H}^{1}}\left( \mathbb{R}^{n} \right)\cap {{W}^{1,\infty }}\left( {\mathbb{R}}^{n} \right)$, ${{\sigma }_{k}}$ is a nonlinear diffusion term for each $\ k\in \mathbb{N}$, ${{\delta }_{k}},\ k\in \mathbb{N}$, are nonnegative constants  and ${{\left\{ {{W}_{k}} \right\}}_{k\in \mathbb{N}}}$ is a sequence of independent standard real-valued Winner processes on a complete filtered probability space $\left( \Omega ,\mathcal{F},{{\left\{ {{\mathcal{F}}_{t}} \right\}}_{t\in \mathbb{R}}},\mathbb{P} \right)$, respectively.

Distribution-dependent stochastic differential equations (SDEs), alternatively known as McKean-Vlasov SDEs (MVSDEs) or mean-field SDEs, are equations that depend not only on the state of their solutions but also on the distribution of those solutions. This dependency causes the Markov operators associated with them to lose their semigroup properties. Recent years have seen a surge of interest in these equations, following pioneering works such as \cite{KMP,MHP}. The existence and uniqueness of solutions have been firmly established in studies including \cite{XFY,JGW,HSM,HXR,HSL,RZE,WDL}. Additionally, singular coefficients have been explored in \cite{BHM, HWF,WDS,RZW}, while large deviation principles have been introduced in \cite{CWW,LSZ,WDG,XZL}. It is noteworthy that Shi et al. \cite{Shi} have investigated the existence of pullback measure attractors for distribution-dependent stochastic reaction-diffusion equations.

 The Fitzhugh-Nagumo system, as a mathematical model, is dedicated to describing the transmission of signals within the nervous system, as documented in \cite{JNS,RFI}. The limiting behavior of solutions to the Fitzhugh-Nagumo system, particularly when defined on unbounded domains, has been extensively studied in numerous works, including \cite{AAB,EVV,AGY,AGB,BT,BW,WLB,RW,WAG,WZ}. For bounded domains, similar investigations have been conducted in \cite{LLW,LLZ,YL}. However, the literature has yet to explore the Fitzhugh-Nagumo system with distribution dependence. In this paper, we delve into the dynamics of the Fitzhugh-Nagumo system with distribution dependence, specifically as it is defined on $\mathbb{R}^{n}$.

 For simplicity, in the following the norm of $L^2(\R^n)$ is denoted as $\|\cdot\|$ and
 we write ${{\xi }_{0}}=\left( {{\xi }_{1}},{{\xi }_{2}} \right)$,
  ${{\mathbb{L}}^{2}}\left( {{\mathbb{R}}^{n}} \right)={{L}^{2}}\left( {{\mathbb{R}}^{n}} \right)\times {{L}^{2}}\left( {{\mathbb{R}}^{n}} \right)$
  and ${{\mathbb{H}}^{1}}\left( {{\mathbb{R}}^{n}} \right)={{H}^{1}}\left( {{\mathbb{R}}^{n}} \right)\times {{H}^{1}}\left( {{\mathbb{R}}^{n}} \right)$.
 To describe the main results of this paper, we denote by $\mathcal{P}\left( {{\mathbb{L}}^{2}}\left( {{\mathbb{R}}^{n}} \right) \right)$
 the space of probability measures on $\left( {{\mathbb{L}}^{2}}\left( {{\mathbb{R}}^{n}} \right),\mathcal{B}\left( {{\mathbb{L}}^{2}}\left( {{\mathbb{R}}^{n}} \right) \right) \right)$), where  $\mathcal{B}\left( {{\mathbb{L}}^{2}}\left( {{\mathbb{R}}^{n}} \right) \right)$ is the Borel $\sigma$-algebra of ${{\mathbb{L}}^{2}}\left( {{\mathbb{R}}^{n}} \right)$.
 The weak topology of $\mathcal{P}\left( {{\mathbb{L}}^{2}}\left( {{\mathbb{R}}^{n}} \right) \right)$ is metrizable, and the corresponding metric is denoted by ${{d}_{\mathcal{P}\left( {{\mathbb{L}}^{2}}\left( {{\mathbb{R}}^{n}} \right) \right)}}$.  Set
\[
{{\mathcal{P}}_{4}}\left( {\mathbb{L}^{2}}\left( {\mathbb{R}}^{n} \right) \right)=\left\{ \mu \in \mathcal{P}\left( {\mathbb{L}^{2}}\left( {\mathbb{R}}^{n} \right) \right):\int_{{\mathbb{L}^{2}}\left( {\mathbb{R}}^{n} \right)}{\left\| \xi  \right\|_{{\mathbb{L}^{2}}\left( {\mathbb{R}}^{n} \right)}^{4}d\mu \left( \xi  \right)}<\infty  \right\}.
\]
Then $\left( {{\mathcal{P}}_{4}}\left( {\mathbb{L}^{2}}\left( {\mathbb{R}}^{n} \right) \right),{{d}_{\mathcal{P}\left( {\mathbb{L}^{2}}\left( {\mathbb{R}}^{n} \right) \right)}} \right)$ is a metric space. Given $r > 0$, denote by
\[
{{B}_{{{\mathcal{P}}_{4}}\left( {\mathbb{L}^{2}}\left( {\mathbb{R}}^{n} \right) \right)}}\left( r \right)=\left\{ \mu \in {{\mathcal{P}}_{4}}\left( {\mathbb{L}^{2}}\left( {\mathbb{R}}^{n} \right) \right):\int_{{\mathbb{L}^{2}}\left({\mathbb{R}}^{n} \right)}{\left\| \xi  \right\|_{{\mathbb{L}^{2}}\left( {\mathbb{R}}^{n} \right)}^{4}d\mu \left( \xi  \right)}\le {{r}^{4}} \right\}.
\]

Given $\tau \le t$ and $\mu \in {{\mathcal{P}}_{4}}\left( {\mathbb{L}^{2}}\left( {\mathbb{R}}^{n} \right) \right)$, let $P_{\tau ,t}^{*}\mu $ be the law of the solution of \eqref{a1}-\eqref{a2} with initial law $\mu $  at initial time $\tau $. If $\phi :{\mathbb{L}^{2}}\left( {\mathbb{R}}^{n} \right)\to \mathbb{R}$  is a bounded Borel function, then we write
\[
{{p}_{\tau ,t}}\phi \left( {{u}_{\tau }} \right)=\mathbb{E}\left( \phi \left( k\left( t,\tau ,{{\xi }_{0}} \right) \right) \right),\quad\forall {{\xi }_{0}}\in {{\mathbb{L}}^{2}}\left( {{\mathbb{R}}^{n}} \right),
\]
where $k\left( t,\tau ,{{\xi }_{0}} \right)=(u\left( t,\tau ,{{\xi }_{0}} \right),v\left( t,\tau ,{{\xi }_{0}} \right))$
is the solution of \eqref{a1}-\eqref{a2} with initial value $\xi_{0}$  at initial time $\tau $.
Note that for the McKean-Vlasov stochastic  equation as  the  equation of \eqref{a1}, $P_{\tau ,t}^{*}$
is not the dual of ${{P}_{\tau ,t}}$ (see \cite{WDL}) in the sense that
\begin{align}\label{a3}
\int_{{\mathbb{L}^{2}}\left( {\mathbb{R}}^{n} \right)}{{{P}_{\tau ,t}}\phi \left( \xi  \right)d\mu \left( \xi  \right)}\ne \int_{{\mathbb{L}^{2}}\left( {\mathbb{R}}^{n} \right)}{\phi \left( \xi  \right)dP_{\tau ,t}^{*}\left( \xi  \right)},
\end{align}
where $\mu \in {{\mathcal{P}}_{4}}\left( {\mathbb{L}^{2}}\left( {\mathbb{R}}^{n} \right) \right)$ and $\phi :{\mathbb{L}^{2}}\left( {\mathbb{R}}^{n} \right)\to \mathbb{R}$ is a bounded Borel function.

The main goal of this paper is to prove that the system \eqref{a1}-\eqref{a2} admits a unique pullback  measure attractors
$\mathcal{A}=\{\mathcal{A}\left( \tau  \right):\tau \in \mathbb{R}\}\in {{\mathcal{D}}}$ in ${{\mathbb L}^{2}}\left( {{\mathbb{R}}^{n}} \right)$. The first difficulty lies in the fact that Sobolev embeddings are not compact as a result of the Fitzhugh-Nagumo system is partially dissipative systems. To overcome this obstacle, we need to decompose  $v$ into two functions,  one part is uniformly asymptotically null in ${{L}^{2}}\left( {{\mathbb{R}}^{n}} \right)$ and the other is regular in the sense that it belongs to ${{H}^{1}}\left( {{\mathbb{R}}^{n}} \right)$ as $t\rightarrow\infty$. The second obstacle lies in the fact that we can't directly obtain that ${{\left\{ P_{\tau ,t}^{*} \right\}}_{\tau \le t}}$ is a continuous non-autonomous dynamical system on $\left( {{\mathcal{P}}_{4}}\left( {\mathbb{L}^{2}}\left( {{\mathbb{R}}^{n}} \right) \right),{{d}_{\mathcal{P}\left( {\mathbb{L}^{2}}\left( {{\mathbb{R}}^{n}} \right) \right)}} \right)$ as a result of  ${P_{\tau ,t}^*}$ is no longer the dual of ${P_{\tau ,t}}$ as demonstrated by \eqref{a3}. To solve this problem, we will take the advantage of the regularity of ${{\mathcal{P}}_{4}}\left( {\mathbb{L}^{2}}\left( {{\mathbb{R}}^{n}} \right) \right)$ and apply the Vitali's theorem to prove the continuity of ${{\left\{ P_{\tau ,t}^{*} \right\}}_{\tau \le t}}$ on the subspace $\left( {{B}_{{{\mathcal{P}}_{4}}\left( {\mathbb{L}^{2}}\left( {{\mathbb{R}}^{n}} \right) \right)}}\left( r \right),{{d}_{\mathcal{P}\left( {\mathbb{L}^{2}}\left( {{\mathbb{R}}^{n}} \right) \right)}} \right)$
instead of the entire space $\left( {{\mathcal{P}}_{4}}\left( {\mathbb{L}^{2}}\left( {{\mathbb{R}}^{n}} \right) \right),{{d}_{\mathcal{P}\left( {\mathbb{L}^{2}}\left( {{\mathbb{R}}^{n}} \right) \right)}} \right)$.

The organization of this paper is as follows. In Section 2, we recall  the abstract theory of
  pullback measure attractors. In Section 3, we give some assumptions  which will
  be used in subsequent estimates. Section 4 is devoted to obtain the well-posedness of \eqref{a1}-\eqref{a2}.
  Section 5 is devoted to obtain the necessary estimates about the solutions of  \eqref{a1}-\eqref{a2}.
  In Section 6, we establish  the existence and uniqueness of pullback measure attractors for the system \eqref{a1}-\eqref{a2}.

\section{Preliminaries}
In this section, in order to discuss issues in an abstract framework, we  recall
some abstract theory of pullback measure attractors \cite{LW24, Shi}.

In the following, let $X$ be a separable Banach space with norm $\|\cdot\|_{X}$. Define $C_{b}(X)$
as the space of bounded continuous functions $\varphi:X\rightarrow\mathbb{R}$ endowed with the norm
\begin{equation*}
\|\varphi\|_{\infty}=\sup_{x\in X}|\varphi(x)|.
\end{equation*}
Let $L_{b}(X)$ denote the space of bounded Lipschitz functions on $X$ equipped with the norm
\begin{equation*}
\|\varphi\|_{L}=\|\varphi\|_{\infty}+\mbox{Lip}(\varphi),
\end{equation*}
where
\begin{equation*}
\mbox{Lip}(\varphi):= \mathop{\sup}\limits_{\scriptstyle x_{1},x_{2}\in X \hfill \atop
\scriptstyle x_{1}\neq x_{2}
\hfill} \frac{|\varphi(x_{1})-\varphi(x_{2})|}{\|x_{1}-x_{2}\|_{X}}<\infty,
\end{equation*}
for $f\in C_{b}(X)$. Denote by $\mathcal{P}(X)$ the set of probability measure on $(X,\mathcal{B}(X))$, where $\mathcal{B}(X)$ is the Borel $\sigma$-algebra of $X$. Given  $\varphi\in C_{b}(X)$ and $\mu\in\mathcal{P}(X)$, we write
\begin{equation*}
(\varphi,\mu)=\int_{X}\varphi(x)\mu(dx).
\end{equation*}

We say that a sequence $\{\mu_{n}\}_{n=1}^{\infty}\subseteq\mathcal{P}(X)$ is weakly convergent to $\mu\in\mathcal{P}(X)$ if and only if for every $\varphi\in C_{b}(X)$,
\begin{equation*}
\lim\limits_{n\rightarrow\infty}(\varphi,\mu_{n})=(\varphi,\mu).
\end{equation*}
The weak topology of  $\mathcal{P}(X)$ is metrizable with metric given by
\begin{equation*}
d_{\mathcal{P}(X)}(\mu_{1},\mu_{2})
= \mathop {\sup }\limits_{\scriptstyle \varphi \in L_b \left( X \right) \hfill \atop
                                           \scriptstyle \left\| \varphi \right\|_L  \le
                                            1 \hfill}|(\varphi,\mu_{1})-(\varphi,\mu_{2})|,
                                            ~~~\forall \mu_{1},\mu_{2}\in\mathcal{P}(X).
\end{equation*}
Note that the $(\mathcal{P}(X),d)$ is a Polish space. Denote $(\mathcal{P}_{p}(X),\mathbb{W}_{p})$ be the Polish space such that  for all $p\geq1$,
\begin{equation*}
\mathcal{P}_{p}(X)=\left\{\mu\in\mathcal{P}(X):\int_{X}\|x\|^{p}_{X}\mu(dx)<\infty\right\},
\end{equation*}
and
\[
{{\mathbb{W}}_{p}}\left( \mu ,\nu  \right)=\underset{\pi \in \prod \left( \mu ,\nu  \right)}{\mathop{\inf }}\,{{\left( \int_{X\times X}{\left\| x-y \right\|_{X}^{p}\pi \left( dx,dy \right)} \right)}^{\frac{1}{p}}},\quad \forall \mu ,\nu \in {{\mathcal{P}}_{p}}\left( X \right),
\]
 where $\prod \left( \mu ,\nu  \right)$ is the set of all couplings of $\mu$ and $\nu$, the metric ${\mathbb{W}}_{p} $ is called the Wasserstein distance.
  Given $r>0$, define
\begin{equation*}
B_{\mathcal{P}_{p}(X)}(r)=\left\{\mu\in\mathcal{P}_{p}(X):\left(\int_{X}\|x\|_{X}^{p}\mu(dx)\right)^{\frac{1}{p}}\leq r\right\}.
\end{equation*}
A subset $\mathcal E\subseteq {{\mathcal{P}}_{p}}\left( X \right)$ is bounded if there is $r>0$
such that $\mathcal E\in {{B}_{{{\mathcal{P}}_{p}}\left( X \right)}}\left( r \right)$. If $\mathcal E$ is bounded in
${{{\mathcal{P}}_{p}}\left( X \right)}$, then we set
 \[
{{\left\| \mathcal E \right\|}_{{{\mathcal{P}}_{p}}\left( X \right)}}=\underset{\mu \in \mathcal E}{\mathop{\sup }}\,{{\left( \int_{X}{\left\| x \right\|_{X}^{p}\mu \left( dx \right)} \right)}^{\frac{1}{p}}}.
\]
Since $\left( {{\mathcal{P}}_{p}}\left( X \right),{{\mathbb{W}}_{p}} \right)$ is a Polish space, but $\left( {{\mathcal{P}}_{p}}\left( X \right),{{d}_{{{\mathcal{P}}_{p}}\left( X \right)}} \right)$ is not complete. Note that for any $r>0$, ${{B}_{{{\mathcal{P}}_{p}}\left( X \right)}}\left( r \right)$ is a closed subset of ${{\mathcal{P}}_{p}}\left( X \right)$ with respect to the metric ${{d}_{{{\mathcal{P}}_{p}}\left( X \right)}}$, we know that the space $\left( {{B}_{{{\mathcal{P}}_{p}}\left( X \right)}}\left( r \right),{{d}_{{{\mathcal{P}}_{p}}\left( X \right)}} \right)$ is complete for every $r>0$.
\begin{defn}\label{defn1}
A family $S=\{S(t,\tau):t\in\mathbb{R}^{+},\tau\in\mathbb{R}\}$ of mappings from $\mathcal{P}_{p}(X)$ to $\mathcal{P}_{p}(X)$ is called a continuous non-autonomous dynamical system on $\mathcal{P}_{p}(X)$, if for all $\tau\in\mathbb{R}$ and $t,s\in\mathbb{R}^{+}$, the following conditions are satisfied

$(a)$ $S(0,\tau)=I_{\mathcal{P}_{p}(X)}$, where $I_{\mathcal{P}_{p}(X)}$ is the identity operator on $\mathcal{P}_{p}(X)$;

$(b)$ $S(t+s,\tau)=S(t,s+\tau)\circ S(s,\tau)$;

$(c)$ $S(t,\tau):\mathcal{P}_{p}(X)\rightarrow\mathcal{P}_{p}(X)$ is continuous.
\end{defn}
\begin{defn}\label{defn2}
A set $D\subseteq\mathcal{P}_{p}(X)$ is called a bounded subset if there is $r>0$ such that $D\subseteq B_{\mathcal{P}_{p}(X)}(r)$.
\end{defn}
In the sequel, we denote by $\mathcal{D}$ a collection of some families of nonempty subsets of $\mathcal{P}_{p}(X)$ parametrized by $\tau\in\mathbb{R}$, that is,
\begin{equation*}
\mathcal{D}=\left\{D=\{D(\tau)\subseteq\mathcal{P}_{p}(X):D(\tau)\neq\emptyset,\tau\in\mathbb{R}\}:D~\mbox{satisfies~some~conditions}\right\}.
\end{equation*}
\begin{defn}\label{defn3}
A collection $\mathcal{D}$ of some families of nonempty subsets of $\mathcal{P}_{p}(X)$ is said to be neighborhood-closed if for each $D=\{D(\tau):\tau\in\mathbb{R}\}\in\mathcal{D}$, there exists a positive number $\epsilon$ depending on $D$ such that the family
\begin{equation*}
\left\{B(\tau):B(\tau)~\mbox{is~a~nonempty~subset~of}~\mathcal{N}_{\epsilon}(D(\tau)),~\forall\tau\in\mathbb{R}\right\},
\end{equation*}
also belongs to $\mathcal{D}$.
\end{defn}
Note that the neighborhood closedness of $\mathcal{D}$ implies for each $D\in\mathcal{D}$,
\begin{equation}\label{LR3}
\widetilde{D}=\left\{\widetilde{D}(\tau):\emptyset\neq\widetilde{D}(\tau)\subseteq D(\tau),\tau\in\mathbb{R}\right\}\in\mathcal{D}.
\end{equation}
A collection $\mathcal{D}$ satisfying \eqref{LR3} is said to be inclusion-closed in the literature.
\begin{defn}\label{defn4}
A family $K=\{K(\tau):\tau\in\mathbb{R}\}\in\mathcal{D}$ is called a $\mathcal{D}$-pullback absorbing set for $S$ if for each $\tau\in\mathbb{R}$ and every $D\in\mathcal{D}$, there exists $T=T(\tau,D)>0$ such that
\begin{equation*}
S(t,\tau-t)D(\tau-t)\subseteq K(\tau),~~~\mbox{for~all}~t\geq T.
\end{equation*}
\end{defn}
\begin{defn}\label{defn5}
The non-autonomous dynamical system $S$ is said to be $\mathcal{D}$-pullback asymptotically compact in $\mathcal{P}_{p}(X)$ if for each $\tau\in\mathbb{R}$, $\left\{S(t_{n},\tau-t_{n})\mu_{n}\right\}_{n=1}^{\infty}$ has a convergent subsequence in $\mathcal{P}_{p}(X)$ whenever $t_{n}\rightarrow+\infty$ and $\mu_{n}\in D(\tau-t_{n})$ with $D\in\mathcal{D}$.
\end{defn}
\begin{defn}\label{defn6}
A family $\mathcal{A}=\{\mathcal{A}(\tau):\tau\in\mathbb{R}\}\in\mathcal{D}$ is called a $\mathcal{D}$-pullback measure attractor for $S$ if the following conditions are satisfied,

$(i)$ $\mathcal{A}(\tau)$ is compact in $\mathcal{P}_{p}(X)$ for each $\tau\in\mathbb{R}$;

$(ii)$ $\mathcal{A}$ is invariant, that is, $S(t,\tau)\mathcal{A}(\tau)=\mathcal{A}(t+\tau)$, for all $\tau\in\mathbb{R}$ and $t\in\mathbb{R}^{+}$;

$(iii)$ $\mathcal{A}$ attracts every set in $\mathcal{D}$, that is, for each $D=\{D(\tau):\tau\in\mathbb{R}\}\in\mathcal{D}$,
\begin{equation*}
\lim\limits_{t\rightarrow\infty}d(S(t,\tau-t)D(\tau-t),\mathcal{A}(\tau))=0.
\end{equation*}
\end{defn}
\begin{defn}\label{defn7}
A mapping $\psi:\mathbb{R}\times\mathbb{R}\rightarrow\mathcal{P}_{p}(X)$ is called a complete orbit of $S$ if for every $s\in\mathbb{R}$, $t\in\mathbb{R}^{+}$ and $\tau\in\mathbb{R}$, the following holds
\begin{equation}\label{4}
S(t,s+\tau)\psi(s,\tau)=\psi(t+s,\tau).
\end{equation}
In addition, if there exists $D=\{D(\tau):\tau\in\mathbb{R}\}\in\mathcal{D}$ such that $\psi(t,\tau)$ belongs to $D(\tau+t)$ for every $t\in\mathbb{R}$ and $\tau\in\mathbb{R}$, then $\psi$ is called a $\mathcal{D}$-complete orbit of $S$.
\end{defn}
\begin{defn}\label{defn8}
A mapping $\xi:\mathbb{R}\rightarrow\mathcal{P}_{p}(X)$ is called a complete solution of $S$ if for every $t\in\mathbb{R}^{+}$ and $\tau\in\mathbb{R}$, the following holds
\begin{equation*}
S(t,\tau)\xi(\tau)=\xi(t+\tau).
\end{equation*}
In addition, if there exists $D=\{D(\tau):\tau\in\mathbb{R}\}\in\mathcal{D}$ such that $\xi(\tau)$ belongs to $D(\tau)$ for every $\tau\in\mathbb{R}$, then $\xi$ is called a $\mathcal{D}$-complete solution of $S$.
\end{defn}
\begin{defn}\label{defn9}
For each $D=\{D(\tau):\tau\in\mathbb{R}\}\in\mathcal{D}$ and $\tau\in\mathbb{R}$, the pullback $\omega$-limit set of $D$ at $\tau$ is defined by
\begin{equation*}
\omega(D,\tau):=\bigcap\limits_{s\geq0}\overline{\bigcup\limits_{t\geq s}S(t,\tau-t)D(\tau-t)},
\end{equation*}
that is,
\begin{equation*}
\omega(D,\tau)=\left\{\nu\in\mathcal{P}_{p}(X):\mbox{there~exists}~t_{n}\rightarrow\infty,~\mu_{n}\in D(\tau-t_{n})~\mbox{such~that}~ \nu=\lim\limits_{n\rightarrow\infty}S(t_{n},\tau-t_{n})\mu_{n}\right\}.
\end{equation*}
\end{defn}
Based on above notation, we give the following main criterion about the existence
and uniqueness of $\mathcal{D}$-pullback measure attractor \cite{WBC}.
\begin{prop}\label{P1}
Let $\mathcal{D}$ be a neighborhood-closed collection of families of subsets of $\mathcal{P}_{p}(X)$ and $S$ be a continuous non-autonomous dynamical system on $\mathcal{P}_{p}(X)$. Then $S$ has a unique $\mathcal{D}$-pullback measure attractor $\mathcal{A}$ in $\mathcal{P}_{p}(X)$ if and only if $S$ has a closed $\mathcal{D}$-pullback absorbing set $K\in\mathcal{D}$ and $S$ is $\mathcal{D}$-pullback asymptotically compact in $\mathcal{P}_{p}(X)$.
The $\mathcal D$-pullback measure attractor $\mathcal A$    is given by, for each $\tau\in \R$,
\begin{align*}
\begin{split}
\mathcal A\left( \tau \right)  = \omega(K,\tau)
 =\{\psi(0,\tau):\,\psi\,\text{is a}\,\,  \mathcal D \text{-complete orbit of}\,\,S\}\\
 =\{\xi(\tau):\,\xi\,\text{is a}\,\,  \mathcal D \text{-complete solution  of}\,\,S\}.
\end{split}
\end{align*}
\end{prop}

\section{Abstract Formulation of Stochastic Equations}
\setcounter{equation}{0}
In this section, we present some assumptions that will be used in subsequent estimates. Throughout this paper, we set $\delta_{0}$ for the Dirac probability measure at $0$. \\
 \textbf{(H1).\quad Assumption on nonlinear term $f$}.  Suppose $f:\mathbb{R}\times{\mathbb{R}}^{n}\times\mathbb{R}\times\mathcal{P}_{2}({{L}^{2}}\left( {{\mathbb{R}}^{n}} \right))\rightarrow\mathbb{R}$ is continuous and differentiable with respect to the second and third arguments, that is for all $t,u,u_{1},u_{2}\in \mathbb{R}$, $x\in {\mathbb{R}}^{n}$ and $\mu,\mu_{1},\mu_{2}\in {{\mathcal{P}}_{2}}\left( {{L}^{2}}\left( {{\mathbb{R}}^{n}} \right) \right)$,
 \begin{equation}\label{c1}
 	f\left( t,x,0,{{\delta }_{0}} \right)=0,
 \end{equation}
 \begin{equation}\label{c2}
	f\left( t,x,u,\mu  \right)u\ge {{\alpha }_{1}}{{\left\vert u \right\vert}^{p}}-{{\phi }_{1}}\left( t,x \right)\left( 1+{{\left\vert u \right\vert}^{2}} \right)-{{\psi }_{1}}\left( x \right)\mu \left( {{\left\| \cdot  \right\|}^{2}} \right),
\end{equation}
\begin{equation}\label{c3}
	\begin{split}
			 \left\vert f\left( t,x,{{u}_{1}},{{\mu }_{1}} \right)-f\left( t,x,{{u}_{2}},{{\mu }_{2}} \right) \right\vert\le& {{\alpha }_{2}}\left( {{\phi }_{2}}\left( t,x \right)+{{\left\vert {{u}_{1}} \right\vert}^{p-2}}+{{\left\vert {{u}_{2}} \right\vert}^{p-2}} \right)\left\vert {{u}_{1}}-{{u}_{2}} \right\vert \\
		& +{{\phi }_{3}}\left( t,x \right)\mathbb{W}\left( {{\mu }_{1}},{{\mu }_{2}} \right),
	\end{split}
\end{equation}
 \begin{equation}\label{c4}
 \frac{\partial f}{\partial u}\left( t,x,u,\mu  \right)\ge -{{\phi }_{4}}\left( t,x \right),
 \end{equation}
 \begin{equation}\label{c5}
 \left\vert \frac{\partial f}{\partial x}\left( t,x,u,\mu  \right) \right\vert\le {{\phi }_{5}}\left( t,x \right)\left( 1+\left\vert u \right\vert+\sqrt{\mu \left( {{\left\| \cdot  \right\|}^{2}} \right)} \right),
 \end{equation}
 where $p\ge 2$, ${{\alpha }_{1}}>0$, ${{\alpha }_{2}}>0$, ${{\psi }_{1}}\in {{L}^{1}}\left( {\mathbb{R}}^{n} \right)\cap {{L}^{\infty }}\left( {\mathbb{R}}^{n} \right)$ and ${{\phi }_{i}}\in {{L}^{\infty }}\left( {{\mathbb{R}}},{{L}^{1}}\left( {\mathbb{R}}^{n} \right)\cap {{L}^{\infty }}\left( {\mathbb{R}}^{n} \right) \right)$ for $i=1,2,3,4,5$.

 By \eqref{c1}-\eqref{c3} we can get that for all  $t,u\in \mathbb{R}$, $x\in {\mathbb{R}}^{n}$ and $\mu \in {{\mathcal{P}}_{2}}\left( {{L}^{2}}\left( {{\mathbb{R}}^{n}} \right) \right)$,
 \begin{equation}\label{c6}
 	\left\vert f\left( t,x,u,\mu  \right) \right\vert\le {{\alpha }_{3}}{{\left\vert u \right\vert}^{p-1}}+{{\phi }_{6}}\left( t,x \right)\left( 1+\sqrt{\mu \left( {{\left\| \cdot  \right\|}^{2}} \right)} \right),
 \end{equation}
 where ${{\alpha }_{3}}>0$ and ${{\phi }_{6}}\in {{L}^{\infty }}\left( {{\mathbb{R}}},{{L}^{1}}\left( {\mathbb{R}}^{n} \right)\cap {{L}^{\infty }}\left( {\mathbb{R}}^{n} \right) \right)$.\\
 \textbf{(H2).\quad Assumption on nonlinear terms $G_{1}$}. Suppose $G_{1}:\mathbb{R}\times{\mathbb{R}}^{n}\times\mathbb{R}\times\mathcal{P}_{2}({{L}^{2}}\left( {{\mathbb{R}}^{n}} \right))\rightarrow\mathbb{R}$ is continuous and differentiable with respect to the second and third arguments, that is for all $t,u,u_{1},u_{2}\in \mathbb{R}$, $x\in {\mathbb{R}}^{n}$ and $\mu,\mu_{1},\mu_{2}\in {{\mathcal{P}}_{2}}\left( {{L}^{2}}\left( {{\mathbb{R}}^{n}} \right) \right)$,
\begin{equation}\label{c7}
\left\vert {{G}_{1}}\left( t,x,u,\mu  \right) \right\vert\le {{\phi }_{g}}\left( t,x \right)+{{\phi }_{7}}\left( t,x \right)\left\vert u \right\vert+{{\psi }_{g}}\sqrt{\mu \left( {{\left\| \cdot  \right\|}^{2}} \right)},
\end{equation}
\begin{equation}\label{c8}
\left\vert {{G}_{1}}\left( t,x,{{u}_{1}},{{\mu }_{1}} \right)-{{G}_{1}}\left( t,x,{{u}_{2}},{{\mu }_{2}} \right) \right\vert\le {{\phi }_{7}}\left( t,x \right)\left( \left\vert {{u}_{1}}-{{u}_{2}} \right\vert+{{\mathbb{W}}_{2}}\left( {{\mu }_{1}},{{\mu }_{2}} \right) \right),
\end{equation}
\begin{equation}\label{c9}
	\left\vert \frac{\partial {{G}_{1}}}{\partial x}\left( t,x,u,\mu  \right) \right\vert\le {{\phi }_{8}}\left( t,x \right)+{{\phi }_{7}}\left( t,x \right)\left( \left\vert u \right\vert+\sqrt{\mu \left( {{\left\| \cdot  \right\|}^{2}} \right)} \right),
\end{equation}
where ${{\phi }_{7}}\in {{L}^{\infty }}\left( {{\mathbb{R}}},{{L}^{1}}\left( {\mathbb{R}}^{n} \right)\cap {{L}^{\infty }}\left( {\mathbb{R}}^{n} \right) \right)$, ${{\psi }_{g}}\in {{L}^{1}}\left( {\mathbb{R}}^{n} \right)\cap {{L}^{\infty }}\left({\mathbb{R}}^{n} \right)$ and ${{\phi }_{g}}$, ${{\phi }_{8}}\in L_{loc}^{2}\left( \mathbb{R},{{L}^{2}}\left( {\mathbb{R}}^{n} \right) \right)$.

By (3.8) we can get that  for all  $t,u\in \mathbb{R}$, $x\in {\mathbb{R}}^{n}$ and $\mu \in {{\mathcal{P}}_{2}}\left( {{L}^{2}}\left( {{\mathbb{R}}^{n}} \right) \right)$,
\begin{equation}\label{c10}
 \left\vert \frac{\partial {{G}_{1}}}{\partial u}\left( t,x,u,\mu  \right) \right\vert\le {{\phi }_{7}}\left( t,x \right).
\end{equation} \\
 \textbf{(H3).\quad Assumption on diffusion term $\theta_{i}$}. Suppose the function ${{\theta }_{i}}=\left\{ {{\theta }_{i,k}} \right\}_{k=1}^{\infty }:\mathbb{R}\to {{L}^{2}}\left( {\mathbb{R}}^{n},{{l}^{2}} \right),\,i=1,2$, are continuous and satisfy
\begin{equation}\label{c11}
\sum\limits_{k=1}^{\infty }{\int_{t}^{t+1}{{{\left\| \nabla {{\theta }_{i,k}}\left( s \right) \right\|}^{2}}}ds}<\infty ,\quad \forall t\in \mathbb{R},\,i=1,2.
\end{equation}\\
\textbf{(H4).\quad Assumption on diffusion term $\sigma_{k}$}. Suppose for every $k\in \mathbb{N}$, $\sigma_{k}$ is continuous and ${{\sigma _{k}}\left( {t,u,\mu } \right)}$ is differentiable in $u$ and Lipschitz continuous in both $u$ and $\mu$ uniformly for $t\in \mathbb{R}$, that is for all  $t,u,u_{1},u_{2}\in \mathbb{R}$ and $\mu ,\mu_{1},\mu_{2}\in {{\mathcal{P}}_{2}}\left( {{L}^{2}}\left( {{\mathbb{R}}^{n}} \right) \right)$,
\begin{equation}\label{c12}
\left\vert {{\sigma }_{k}}\left( t,u,\mu  \right) \right\vert\le {{\beta }_{1,k}}\left( 1+\sqrt{\mu \left( {{\left\| \cdot  \right\|}^{2}} \right)} \right)+{{\gamma }_{1,k}}\left\vert u \right\vert,
\end{equation}
and
\begin{equation}\label{c13}
\left\vert {{\sigma }_{k}}\left( t,{{u}_{1}},{{\mu }_{1}} \right)-{{\sigma }_{k}}\left( t,{{u}_{2}},{{\mu }_{2}} \right) \right\vert\le {{L}_{\sigma ,k}}\left( \left\vert {{u}_{1}}-{{u}_{2}} \right\vert+{{\mathbb{W}}_{2}}\left( {{\mu }_{1}},{{\mu }_{2}} \right) \right),
\end{equation}
where ${{\beta }_{1}}=\left\{ {{\beta }_{1,k}} \right\}_{k=1}^{\infty }$ and ${{\gamma }_{1}}
=\left\{ {{\gamma }_{1,k}} \right\}_{k=1}^{\infty }$ are nonnegative sequences
with $\sum\limits_{k=1}^{\infty }{\left( \beta _{1,k}^{2}+\gamma _{1,k}^{2} \right)}<\infty $
and ${{L}_{\sigma }}=\left\{ {{L}_{\sigma ,k}} \right\}_{k=1}^{\infty }$
is a sequence of nonnegative numbers satisfy $\sum\limits_{k=1}^{\infty }{L_{\sigma ,k}^{2}}<\infty $.

 From \eqref{c13} we can get that for all $t\in \mathbb{R}$, $u\in {{L}^{2}}\left( {{\mathbb{R}}^{n}} \right)$ and $\mu\in {{\mathcal{P}}_{2}}\left( {{L}^{2}}\left( {{\mathbb{R}}^{n}} \right) \right)$,
 \begin{align}\label{c14}
 	\begin{split}
 \left\vert \frac{\partial {{\sigma }_{k}}}{\partial u}\left( t,u,\mu  \right) \right\vert\le {{L}_{{{\sigma }},k}}.
	\end{split}
\end{align}
Denote by $l^{2}$ the space of square summable sequences of real numbers. For all $t\in \mathbb{R}$, $u,v\in {{L}^{2}}\left( {{\mathbb{R}}^{n}} \right)$ and $\mu\in {{\mathcal{P}}_{2}}\left( {{L}^{2}}\left( {{\mathbb{R}}^{n}} \right) \right)$, define  maps $\sigma(t,u,\mu),\delta(t,v):l^{2}\rightarrow {{L}^{2}}\left( {{\mathbb{R}}^{n}} \right)$ by
\begin{equation}\label{c15}
		{{\sigma }}\left( t,u,\mu  \right)\left( \eta  \right)\left( x \right)=\sum\limits_{k=1}^{\infty }{\left( {{\theta }_{1,k}}\left( t,x \right)+w\left( x \right){{\sigma }_{k}}\left( t,u\left( x \right),\mu  \right) \right){{\eta }_{k}}},\quad \forall \eta =\left\{ {{\eta }_{k}} \right\}_{k=1}^{\infty }\in {{l}^{2}},\,x\in {{\mathbb{R}}^{n}},
\end{equation}
and
\begin{equation}\label{c+}
{{\delta }}\left( t,v  \right)(\eta)\left( x \right)=\sum\limits_{k=1}^{\infty }{\left( {{\theta }_{2,k}}\left( t,x \right)+{{\delta}_{k}}v\left( t \right) \right){{\eta }_{k}}},\quad \forall \eta =\left\{ {{\eta }_{k}} \right\}_{k=1}^{\infty }\in {{l}^{2}},\, x\in {{\mathbb{R}}^{n}}.
\end{equation}			
 Let $L_{2}(l^{2},{{L}^{2}}\left( {{\mathbb{R}}^{n}} \right))$ be the space of Hilbert-Schmidt operators from $l^{2}$ to ${{L}^{2}}\left( {{\mathbb{R}}^{n}} \right)$ with norm ${{\left\| \cdot  \right\|}_{{{L}_{2}}\left( {{l}^{2}},{{L}^{2}}\left( {{\mathbb{R}}^{n}} \right) \right)}}$.  It follows from \eqref{c11}-\eqref{c12}, \eqref{c15}-\eqref{c+} we can get
 \begin{equation}\label{c16}
 	 \left\| {{\sigma }}\left( t,u,\mu  \right) \right\|_{{{L}_{2}}\left( {{l}^{2}},{{L}^{2}}\left( {{\mathbb{R}}^{n}} \right) \right)}^{2}\le2\left\| {{\theta }_{1}}\left( t \right) \right\|_{{{L}^{2}}\left( {{\mathbb{R}}^{n}},l^{2} \right)}^{2}+8{{\left\| w \right\|}^{2}}\left\| {{\beta }_{1}} \right\|_{{{l}^{2}}}^{2}\left( 1+\mu \left( {{\left\| \cdot  \right\|}^{2}} \right) \right)+4\left\| w \right\|_{{{L}^{\infty }}\left( {{\mathbb{R}}^{n}} \right)}^{2}\left\| {{\gamma }_{1}} \right\|_{{{l}^{2}}}^{2}{{\left\| u \right\|}^{2}}
\end{equation}
and
\begin{equation}\label{c17}
\left\| {{\delta }}\left( t,v \right) \right\|_{{{L}_{2}}\left( {{l}^{2}},{{L}^{2}}\left( {{\mathbb{R}}^{n}} \right) \right)}^{2}\le 2\left\| {{\theta }_{2}}\left( t \right) \right\|_{{{L}^{2}}\left( {{\mathbb{R}}^{n}},l^{2} \right)}^{2}+2\|\delta\|^2_{l^2}{{\left\| v \right\|}^{2}},
\end{equation}
where $\delta =\left\{ {{\delta }_{k}} \right\}_{k=1}^{\infty }$ is a sequence of  nonnegative  numbers satisfy $2\sum\limits_{k=1}^{\infty }{{{\left\vert {{\delta }_{k}} \right\vert}^{2}}}<\gamma <\infty$.
 Besides, by \eqref{c13} and \eqref{c17} we can get that
 for all  $t,u_{1},u_{2},v_{1},v_{2}\in \mathbb{R}$ and $\mu ,\mu_{1},\mu_{2}\in {{\mathcal{P}}_{2}}\left( {{L}^{2}}\left( {{\mathbb{R}}^{n}} \right) \right)$,
 \begin{align}\label{c00}
 	\begin{split}
 		 & \left\| {{\sigma }}\left( t,{{u}_{1}},{{\mu }_{1}} \right)-{{\sigma }}\left( t,{{u}_{2}},{{\mu }_{2}} \right) \right\|_{{{L}_{2}}\left( {{l}^{2}},{{L}^{2}}\left( {{\mathbb{R}}^{n}} \right) \right)}^{2} \\
 	\le	&  2\left\| {{L}_{{\sigma }}} \right\|_{{{l}^{2}}}^{2}\left( \left\| w \right\|_{{{L}^{\infty }}\left( {{\mathbb{R}}^{n}} \right)}^{2}{{\left\| {{u}_{1}}-{{u}_{2}} \right\|}^{2}}+{{\left\| w \right\|}^{2}}\mathbb{W}_{2}^{2}\left( {{\mu }_{1}},{{\mu }_{2}} \right) \right),
 	\end{split}
 \end{align}
 and
 \begin{align}\label{c01}
 	\begin{split}
 \left\| \delta \left( t,{{v}_{1}} \right)-\delta \left( t,{{v}_{2}} \right) \right\|_{{{L}_{2}}\left( {{l}^{2}},{{L}^{2}}\left( {{\mathbb{R}}^{n}} \right) \right)}^{2}\le \left\| \delta \right\|_{{{l}^{2}}}^{2}\left( {{\left\| {{v}_{1}}-{{v}_{2}} \right\|}^{2}} \right).
 	\end{split}
\end{align}
  We now reformulate problem \eqref{a1}-\eqref{a2} as follows
 \begin{equation}\label{c18}
 	\left\{ \begin{array}{l}
 		 du(t)-\Delta u(t)dt+\lambda u(t)dt+\alpha v(t)dt+{{f}}\left( t,x,u,{{\mathcal{L}}_{u\left( t \right)}} \right)dt={{G}_{1}}\left( t,x,u,{{\mathcal{L}}_{u\left( t \right)}} \right)dt \\
 		 +{{\sigma }}\left( t,u,{{\mathcal{L}}_{u\left( t \right)}} \right)d{{W}}\left( t \right),\quad  t>\tau , \\
 		 dv\left( t \right)+\gamma  v\left( t \right)dt-\beta  u\left( t \right)dt= {{G}_{2}}\left( t,x \right)dt
 		 +{\delta}\left( t,v \right)dW\left( t \right),\quad t>\tau,
 	\end{array} \right.
 \end{equation}
 with initial condition
 \begin{align}\label{c19}
 	u\left( \tau \right)={{\xi }_{1}},\ v\left( \tau \right)={{\xi }_{2}}.
 \end{align}
For a Banach space $X$ and $\tau \in \mathbb{R}$, we use  $L_{{{\mathcal{F}}_{\tau }}}^{2}\left( X \right)$ to denote the space of all ${{\mathcal{F}}_{\tau }}$ -measurable, $X$-valued random variables $\varphi$ with $\mathbb{E}\left\| \varphi  \right\|_{X}^{2}<\infty$, where $\mathbb{E}$ means the mathematical expectation. In the sequel, we assume $c_{i}, i\in\mathbb{N}$,  are positive constants depend on $\tau$, but not on $D_{1}$  and  we also assume that the coefficient $\lambda$ is sufficiently large such that there exists a sufficiently small number $\eta\in(0,1)$,
\begin{align}\label{c21}
	\begin{split}
		 2\lambda -5\eta >&24{{\left\| w \right\|}^{2}}\left\| {{\beta }_{1}} \right\|_{{{l}^{2}}}^{2}+12\left\| w \right\|_{{{L}^{\infty }}\left( {{\mathbb{R}}^{n}} \right)}^{2}\left\| {{\gamma }_{1}} \right\|_{{{l}^{2}}}^{2}+2{{\left\| {{\phi }_{1}} \right\|}_{{{L}^{\infty }}\left( \mathbb{R},{{L}^{\infty }}\left( {{\mathbb{R}}^{n}} \right) \right)}}+2{{\left\| {{\psi }_{1}} \right\|}_{{{L}^{1}}\left( {{\mathbb{R}}^{n}} \right)}} \\
		& +2{{\left\| {{\phi }_{7}} \right\|}_{{{L}^{\infty }}\left( \mathbb{R},{{L}^{1}}\left( {{\mathbb{R}}^{n}} \right)\cap {{L}^{\infty }}\left( {{\mathbb{R}}^{n}} \right) \right)}}+{{\left\| {{\psi }_{g}} \right\|}_{{{L}^{1}}\left( {{\mathbb{R}}^{n}} \right)\cap {{L}^{\infty }}\left( {{\mathbb{R}}^{n}} \right)}}+6\|\delta\|^2_{l^2}.
		\end{split}
	\end{align}
 Given a subset $E$ of ${{\mathcal{P}}_{2}}\left( \mathbb L^2(\R^n) \right)$, which is denoted by
\[
{{\left\| E \right\|}_{{{\mathcal{P}}_{2}}\left( \mathbb L^2(\R^n) \right)}}=\inf \left\{ r>0:\underset{\mu \in E}{\mathop{\sup }}\,{{\left( \int_{\mathbb L^2(\R^n)}{\left\| z \right\|_{\mathbb L^2(\R^n)}^{2}\mu \left( dz \right)} \right)}^{\frac{1}{2}}}\le r \right\},
\]
with the convention that $\inf \emptyset =\infty$. If $E$ is a bounded subset of ${{\mathcal{P}}_{2}}\left( \mathbb L^2(\R^n) \right)$, then ${{\left\| E \right\|}_{{{\mathcal{P}}_{2}}\left( \mathbb L^2(\R^n) \right)}}<\infty $. Denote ${{D}_{1}}=\left\{ D\left( \tau  \right):\tau \in \mathbb{R}, D\left( \tau  \right)\ \text{is a bounded nonempty subset of}\ {{\mathcal{P}}_{2}}\left( \mathbb L^2(\R^n) \right) \right\}$ and ${{D}_{2}}=\left\{ D\left( \tau  \right):\tau \in \mathbb{R}, D\left( \tau  \right)\ \text{is a bounded nonempty subset of}\ {{\mathcal{P}}_{4}}\left( \mathbb L^2(\R^n) \right) \right\}$.
Let ${{\mathcal{D}}_{0}}$ be the collection of all such familities ${{D}_{1}}$ which further satisfy
\begin{align}\label{c22}
	\begin{split}
		\underset{\tau \to -\infty }{\mathop{\lim }}\,{{e}^{\eta \tau }}\left\| {{D}_{1}}\left( \tau  \right) \right\|_{{{\mathcal{P}}_{2}}\left( \mathbb L^2(\R^n)\right)}^{2}=0,
	\end{split}
\end{align}
where $\eta $ is the same number as in \eqref{c21}. Similarly denote by $\mathcal{D}$ the collection of all such families ${{D}_{2}}$ which further satisfy
\begin{align}\label{c23}
	\begin{split}
		\underset{\tau \to -\infty }{\mathop{\lim }}\,{{e}^{2\eta \tau }}\left\| {{D}_{2}}\left( \tau  \right) \right\|_{{{\mathcal{P}}_{4}}\left( \mathbb L^2(\R^n) \right)}^{4}=0.
	\end{split}
\end{align}
It is evident $\mathcal{D}\in {{\mathcal{D}}_{0}}$ .
We also assume the following condition in order to deriving uniform estimates of solutions
\begin{equation}\label{c24}
		\int_{-\infty }^{\tau }{{{e}^{\eta s}}\left( \left\| {{\phi }_{g}}\left( s \right)  \right\|_{{{L}^{2}}\left( {\mathbb{R}}^{n} \right)}^{2}+\left\| {{\theta }_{1}}\left( s \right)   \right\|_{{{L}^{2}}\left( {\mathbb{R}}^{n},{{l}^{2}} \right)}^{2}+\left\| {{\theta }_{2}}\left( s \right)  \right\|_{{{L}^{2}}\left( {\mathbb{R}}^{n},{{l}^{2}} \right)}^{2} \right)}ds<\infty ,\quad \forall \tau \in \mathbb{R},
\end{equation}
and
\begin{equation}\label{c25}
		\int_{-\infty }^{\tau }{{{e}^{\eta s}}\left( \left\| {{\phi }_{g}}\left( s \right) \right\|_{{{L}^{2}}\left( {{\mathbb{R}}^{n}} \right)}^{4}+\left\| {{\theta }_{1}}\left( s \right) \right\|_{{{L}^{2}}\left( {{\mathbb{R}}^{n}},{{l}^{2}} \right)}^{4}+\left\| {{\theta }_{2}}\left( s \right) \right\|_{{{L}^{2}}\left( {{\mathbb{R}}^{n}},{{l}^{2}} \right)}^{4} \right)}ds<\infty  ,\quad \forall \tau \in \mathbb{R}.
\end{equation}

\section{Well-posedness of Stochastic Equations}
\setcounter{equation}{0}
In this section, we give definition of  the solutions of the system \eqref{c18}-\eqref{c19} and establish the existence and uniqueness of solutions by Banach fixed-point theorem.

Firstly, we show the definition of solutions of system \eqref{c18}-\eqref{c19}.
\begin{defn}\label{defn m1}
	For every $\tau \in \mathbb{R}$ and  ${{\xi }_{0}}=\left( {{\xi }_{1}},{{\xi }_{2}} \right)\in L_{{{\mathcal{F}}_{\tau}}}^{2}\left( \Omega ,{{\mathbb{L}}^{2}}\left( {{\mathbb{R}}^{n}} \right) \right)$, a continuous stochastic process ${{k}}=\left( {{u}},{{v}} \right)$ is called a solution of system \eqref{c18}-\eqref{c19} if for every $T>0$,
	\[
		k\in {{L}^{2}}\left( \Omega ,C[\tau ,\tau +T],{{\mathbb{L}}^{2}}\left( {{\mathbb{R}}^{n}} \right) \right)\quad \text{with}
	\]
\[
u\in {{L}^{2}}\left( \Omega ,{{L}^{2}}\left( \tau ,\tau +T;{{H}^{1}}\left( {{\mathbb{R}}^{n}} \right) \right) \right)\cap {{L}^{p}}\left( \Omega ,{{L}^{p}}\left( \tau ,\tau +T;{{L}^{p}}\left( {{\mathbb{R}}^{n}} \right) \right) \right),
\]
such that for all $t\ge \tau $ and $\zeta \in {{H}^{1}}\left( {{\mathbb{R}}^{n}} \right)\cap {{L}^{p}}\left( {{\mathbb{R}}^{n}} \right)$,
\begin{align*}
	& \left( u\left( t \right),\zeta  \right)-\int_{\tau }^{t}{\left( \Delta u(s),\Delta \zeta  \right)}ds+\int_{\tau }^{t}{\left( \lambda u(s)+\alpha v\left( s \right),\zeta  \right)}ds  +\int_{\tau }^{t}{\int_{{{\mathbb{R}}^{n}}}{f\left( s,x,u\left( s \right),{{\mathcal{L}}_{u\left( s \right)}} \right)}}\zeta \left( x \right)dxds \\
	 =&\left( {{\xi }_{1}},\zeta  \right)+\int_{\tau }^{t}{\int_{{{\mathbb{R}}^{n}}}{{{G}_{1}}\left( s,x,u,{{\mathcal{L}}_{u\left( s \right)}} \right)\zeta \left( x \right)}}dxds+\int_{\tau }^{t}{\left( \sigma \left( s,u,{{\mathcal{L}}_{u\left( s \right)}} \right),\zeta  \right)}dW\left( s \right),
\end{align*}
and
\begin{align*}
	\left( v\left( t \right),\zeta  \right)+\int_{\tau }^{t}{\left( \gamma v\left( s \right)-\beta u\left( s \right),\zeta  \right)ds}=\left( {{\xi }_{2}},\zeta  \right)+\int_{\tau }^{t}{\left( {{G}_{2}}\left( s,x \right),\zeta  \right)ds}+\int_{\tau }^{t}{\left( \delta \left( s,v\left( s \right) \right),\zeta  \right)dW\left( s \right)},
\end{align*}
$\mathbb{P}$-almost surely, where $k\left( t,\tau ,{{\xi }_{0}} \right)=\left( u\left( t,\tau ,{{\xi }_{0}} \right),v\left( t,\tau ,{{\xi }_{0}} \right) \right)$ or $k\left( t \right)=\left( u\left( t \right),v\left( t \right) \right)$, and $\zeta $ in the stochastic terms is identified with the element in ${{\left( {{L}^{2}}\left( {{\mathbb{R}}^{n}} \right) \right)}^{*}}={{L}^{2}}\left( {{\mathbb{R}}^{n}} \right)$ in view of Riesz’s representation theorem.

Note that if $k=\left( u,v \right)$ is a solution of system \eqref{c18}-\eqref{c19} in the sense of Definition 4.1, then $\Delta u\in {{L}^{2}}\left( \Omega ,{{L}^{2}}\left( \tau ,\tau +T;{{\left( {{H}^{1}}\left( {{\mathbb{R}}^{n}} \right) \right)}^{*}} \right) \right)$ and $f\in {{L}^{q}}\left( \Omega ,{{L}^{q}}\left( \tau ,\tau +T;{{L}^{q}}\left( {{\mathbb{R}}^{n}} \right) \right) \right)$ with $\frac{1}{p}+\frac{1}{q}=1$, and hence for all $t\ge \tau $,
\begin{align}\label{m1}
	\begin{split}
	& u\left( t \right)-\int_{\tau }^{t}{\Delta u(s)}ds+\int_{\tau }^{t}{\left( \lambda u(s)+\alpha v\left( s \right) \right)}ds+\int_{\tau }^{t}{f\left( s,x,u\left( s \right),{{\mathcal{L}}_{u\left( s \right)}} \right)}ds \\
	 =&{{\xi }_{1}}+\int_{\tau }^{t}{{{G}_{1}}\left( s,x,u,{{\mathcal{L}}_{u\left( s \right)}} \right)}ds+\int_{\tau }^{t}{\sigma \left( s,u,{{\mathcal{L}}_{u\left( s \right)}} \right)}dW\left( s \right)\ \text{in}\ {{\left( {{H}^{1}}\left( {{\mathbb{R}}^{n}} \right)\cap {{L}^{p}}\left( {{\mathbb{R}}^{n}} \right) \right)}^{*}},
	\end{split}
\end{align}
and
\begin{align}\label{m2}
	\begin{split}
v\left( t \right)+\int_{\tau }^{t}{\left( \gamma v\left( s \right)-\beta u\left( s \right) \right)ds}={{\xi }_{2}}+\int_{\tau }^{t}{{{G}_{2}}\left( s,x \right)ds}+\int_{\tau }^{t}{\delta \left( s,v\left( s \right) \right)dW\left( s \right)}\ \text{in}\  {{L}^{2}}\left( {{\mathbb{R}}^{n}} \right),
	\end{split}
\end{align}
$\mathbb{P}$-almost surely.
\end{defn}
Next we will prove the existence and uniqueness of solutions of the system \eqref{c18}-\eqref{c19} by Banach fixed-point theorem.
\begin{thm}\label{lemm1}
	Suppose assumption  $\mathbf{(H_{1})-({H_{4}})}$ hold, then for every $\tau\in \mathbb{R}$ and $\xi _{0}\in L_{{{\mathcal{F}}_{\tau }}}^{2}\left( \Omega ,\mathbb L^2(\R^n) \right)$, the system  \eqref{c18}-\eqref{c19} has a unique solution $k=(u,v)$ in the sense of Definition 4.1
	such that
		\begin{align}\label{m3}
		\begin{split}
			& \beta  {{\left\| u\left( t \right) \right\|}^{2}}+\alpha {{\left\| v\left( t \right) \right\|}^{2}}+2\beta  \int_{\tau }^{t}{\left\| \nabla u\left( s \right) \right\|^{2}}ds  +2\beta \lambda  \int_{\tau }^{t}{{{\left\| u\left( s \right) \right\|}^{2}}}ds+2\alpha \gamma  \int_{\tau }^{t}{{{\left\| v\left( s \right) \right\|}^{2}}}ds \\
			& +2\beta  \int_{\tau }^{t}{\left( f\left( s,x,u\left( s \right),{{\mathcal{L}}_{u\left( s \right)}} \right),u\left( s \right) \right)}ds \\
			=&\beta  {{\left\| {{\xi }_{1}} \right\|}^{2}}+\alpha {{\left\| {{\xi }_{2}} \right\|}^{2}}+2\alpha  \int_{\tau }^{t}{\left( {{G}_{2 }}\left( s\right),v\left( s \right) \right)}ds  +2\beta  \int_{\tau }^{t}{\left( {{G}_{1}}\left( s,x,u\left( s \right),{{\mathcal{L}}_{u\left( s \right)}} \right),u\left( s \right) \right)}ds \\
			& +\alpha  \int_{\tau }^{t}{\left\| {\delta}\left( s,v\left( s \right) \right) \right\|_{{{L}_{2}}\left( {{l}^{2}},{{L}^{2}}\left( {{\mathbb{R}}^{n}} \right) \right)}^{2}}ds+\beta  \int_{\tau }^{t}{\left\| {{\sigma }}\left( s,u\left( s \right),{{\mathcal{L}}_{u\left( s \right)}} \right) \right\|_{{{L}_{2}}\left( {{l}^{2}},{{L}^{2}}\left( {{\mathbb{R}}^{n}} \right) \right)}^{2}}ds \\
			&+2\beta  \int_{\tau }^{t}{\left( {{\sigma }}\left( s,u\left( s,x \right),{{\mathcal{L}}_{u\left( s \right)}} \right),u\left( s \right) \right)}dW\left( s \right)+2\alpha  \int_{\tau }^{t}{\left( {\delta}\left( s,v\left( s,x \right) \right),v\left( s \right) \right)}dW\left( s \right),
		\end{split}
	\end{align}
for all $t\in \left[ \tau ,\tau +T \right]$, $\mathbb{P}$-almost surely.
\end{thm}
\begin{proof}
 Given $\xi _{0}\in L_{{{\mathcal{F}}_{\tau }}}^{2}\left( \Omega , \mathbb L^2(\R^n) \right)$ and $\mu=(\mu_1,\mu_2) \in C\left( \left[ \tau ,\tau +T \right],{{\mathcal{P}}_{2}}\left( {\mathbb{L}^{2}}\left( {{\mathbb{R}}^{n}} \right) \right) \right)$, consider the stochastic equation
   \begin{equation}\label{m4}
  	\left\{ \begin{array}{l}
  	d{{u}_{\mu }}(t)-\Delta {{u}_{\mu }}(t)dt+\lambda {{u}_{\mu }}(t)dt+\alpha {{v}_{\mu }}(t)dt+{{f}_{\mu }}\left( t,x,{{u}_{\mu }}(t) \right)dt={{G}_{1}}_{\mu }\left( t,x,{{u}_{\mu }}(t) \right)dt  \\
  	+{{\sigma }_{\mu }}\left( t,{{u}_{\mu }}(t) \right)dW\left( t \right),\quad t>\tau ,  \\
  	d{{v}_{\mu }}(t)+\gamma {{v}_{\mu }}(t)dt-\beta {{u}_{\mu }}\left( t \right)dt={{G}_{2}}\left( t,x \right)dt+{{\delta }_{\mu }}\left( t,{{v}_{\mu }}(t) \right)dW\left( t \right),\quad t>\tau,	\end{array} \right.
  \end{equation}
with initial condition
\begin{align}\label{m5}
{{u}_{\mu }}\left( \tau  \right)={{\xi }_{1}},\ {{v}_{\mu }}\left( \tau  \right)={{\xi }_{2}},
\end{align}
 where ${{f}_{\mu }}\left( t,x,{{u}_{\mu }}(t) \right)=f\left( t,x,{{u}_{\mu }}(t),\mu_1 \left( t \right) \right)$, ${{G}_{1}}_{\mu }\left( t,x,{{u}_{\mu }}(t) \right)={{G}_{1}}\left( t,x,{{u}_{\mu }}(t),\mu_1 \left( t \right) \right)$, ${{\sigma }_{\mu }}\left( t,{{u}_{\mu }}(t) \right)=\sigma \left( t,{{u}_{\mu }}(t),\mu_1 \left( t \right) \right)$ and ${{\delta }_{\mu }}\left( t,{{v}_{\mu }}(t) \right)=\delta \left( t,{{v}_{\mu }}(t) \right)$.
One can verify that if $\mathbf{(H_{1})-({H_{4}})}$ hold, then ${{f}_{\mu }},{{G}_{1}}_{\mu },{{\sigma }_{\mu }}$ and ${{\delta }_{\mu }}$ satisfy all conditions in [\cite{WBD}, Theorem 6.3], further we can get that system \eqref{m4}-\eqref{m5} has a unique solution
${{k}_{\mu }}=\left( {{u}_{\mu }},{{v}_{\mu }} \right)$ satisfy
\[
{{k}_{\mu }}\in {{L}^{2}}\left( \Omega ,C[\tau ,\tau +T],{{\mathbb{L}}^{2}}\left( {{\mathbb{R}}^{n}} \right) \right)\quad \text{with}
\]
\[
{{u}_{\mu }}\in {{L}^{2}}\left( \Omega ,{{L}^{2}}\left( \tau ,\tau +T;{{H}^{1}}\left( {{\mathbb{R}}^{n}} \right) \right) \right)\cap {{L}^{p}}\left( \Omega ,{{L}^{p}}\left( \tau ,\tau +T;{{L}^{p}}\left( {{\mathbb{R}}^{n}} \right) \right) \right).
\]	
By ${{k}_{\mu }}\in {{L}^{2}}\left( \Omega ,C[\tau ,\tau +T],{{\mathbb{L}}^{2}}\left( {{\mathbb{R}}^{n}} \right) \right)$ and the Lebesgue dominated convergence theorem, we see that ${{k}_{\mu }}\in C\left( [\tau ,\tau +T],{{L}^{2}}\left( \Omega ,{{\mathbb L}^{2}}\left( {{\mathbb{R}}^{n}} \right) \right) \right)$. Since
\[
{{\mathbb{W}}_{2}}\left( {{\mathcal{L}}_{{{k}_{\mu }}\left( t \right)}},{{\mathcal{L}}_{{{k}_{\mu }}\left( s \right)}} \right)\le {{\left( \mathbb{E}\left[ {{\left\| {{k}_{\mu }}\left( t \right)-{{k}_{\mu }}\left( s \right) \right\|}^{2}} \right] \right)}^{\frac{1}{2}}},
\quad \forall s,t\in \left[ \tau ,\tau +T \right],
\]
and hence ${{\mathcal{L}}_{{{k}_{\mu }}\left( \cdot  \right)}}\in C\left( [\tau ,\tau +T],{{\mathcal{P}}_{2}}\left( {\mathbb{L}^{2}}\left( {{\mathbb{R}}^{n}} \right) \right) \right)$. Define a map ${{\Phi }^{{{\xi }_{0}}}}:C\left( [\tau ,\tau +T],{{\mathcal{P}}_{2}}\left( {\mathbb{L}^{2}}\left( {{\mathbb{R}}^{n}} \right) \right) \right)\to C\left( [\tau ,\tau +T],{{\mathcal{P}}_{2}}\left( {\mathbb{L}^{2}}\left( {{\mathbb{R}}^{n}} \right) \right) \right)$ by
\[
{{\Phi }^{{{\xi }_{0}}}}\left( \mu  \right)\left( t \right)={{\mathcal{L}}_{{{k}_{\mu }}\left( t \right)}},\quad t\in [\tau ,\tau +T],\,\mu \in C\left( [\tau ,\tau +T],{{\mathcal{P}}_{2}}\left( {{L}^{2}}\left( {{\mathbb{R}}^{n}} \right) \right) \right),
\]
where $k_{\mu}$ is the solution of \eqref{m4}-\eqref{m5}.
Next, we will prove the existence of solutions of
\eqref{c18}-\eqref{c19} by finding a fixed point of ${{\Phi }^{{{\xi }_{0}}}}$. To that end, we need to show ${{\Phi }^{{{\xi }_{0}}}}$ is a contractive map in the complete metric space $\left( C\left( [\tau ,\tau +T],{{\mathcal{P}}_{2}}\left( {\mathbb{L}^{2}}\left( {{\mathbb{R}}^{n}} \right) \right) \right),{{d}_{C\left( [\tau ,\tau +T],{{\mathcal{P}}_{2}}\left( {\mathbb{L}^{2}}\left( {{\mathbb{R}}^{n}} \right) \right) \right)}} \right)$ where the metric ${{d}_{C\left( [\tau ,\tau +T],{{\mathcal{P}}_{2}}\left( {\mathbb{L}^{2}}\left( {{\mathbb{R}}^{n}} \right) \right) \right)}}$ is defined by
\[
{{d}_{C\left( [\tau ,\tau +T],{{\mathcal{P}}_{2}}\left( {\mathbb{L}^{2}}\left( {{\mathbb{R}}^{n}} \right) \right) \right)}}\left( \mu ,\nu  \right)=\underset{t\in [\tau ,\tau +T]}{\mathop{\sup }}\,{{e}^{-\lambda t}}{{\mathbb{W}}_{2}}\left( \mu \left( t \right),\nu \left( t \right) \right),\quad\forall \mu ,\nu \in C\left( [\tau ,\tau +T],{{\mathcal{P}}_{2}}\left( {\mathbb{L}^{2}}\left( {{\mathbb{R}}^{n}} \right) \right) \right).
\]
Let $\mu ,\nu \in C\left( [\tau ,\tau +T],{{\mathcal{P}}_{2}}\left( {{L}^{2}}\left( {{\mathbb{R}}^{n}} \right) \right) \right)$, and $k_{\mu}$, $k_{\nu}$ be the solution of \eqref{m4}-\eqref{m5}. By Ito's formula we can get
\begin{align}
	\begin{split}\label{m6}
& \beta {{\left\| {{u}_{\mu }}\left( t \right)-{{u}_{\nu }}\left( t \right) \right\|}^{2}}+\alpha {{\left\| {{v}_{\mu }}\left( t \right)-{{v}_{\nu }}\left( t \right) \right\|}^{2}}+2\beta \int_{\tau }^{t}{{{\left\| \nabla \left( {{u}_{\mu }}\left( s \right)-{{u}_{\nu }}\left( s \right) \right) \right\|}^{2}}}ds \\
& +2\beta \lambda \int_{\tau }^{t}{{{\left\| {{u}_{\mu }}\left( s \right)-{{u}_{\nu }}\left( s \right) \right\|}^{2}}}ds+2\alpha \gamma \int_{\tau }^{t}{{{\left\| {{v}_{\mu }}\left( s \right)-{{v}_{\nu }}\left( s \right) \right\|}^{2}}}ds \\
& +2\beta \int_{\tau }^{t}{\int_{{{\mathbb{R}}^{n}}}{\left( f\left( s,x,{{u}_{\mu }}\left( s \right),\mu_1 \left( s \right) \right)-f\left( s,x,{{u}_{\nu }}\left( s \right),\nu_1 \left( s \right) \right) \right)\left( {{u}_{\mu }}\left( s \right)-{{u}_{\nu }}\left( s \right) \right)dx}}ds \\
 =&2\alpha \int_{\tau }^{t}{\left( {{G}_{2}}\left( s \right),{{v}_{\mu }}\left( s \right)-{{v}_{\nu }}\left( s \right) \right)}ds+\alpha \int_{\tau }^{t}{\left\| \delta \left( s,{{v}_{\mu }}\left( s \right) \right)-\delta \left( s,{{v}_{\nu }}\left( s \right) \right) \right\|_{{{L}_{2}}\left( {{l}^{2}},{{L}^{2}}\left( {{\mathbb{R}}^{n}} \right) \right)}^{2}}ds \\
& +2\beta \int_{\tau }^{t}{\int_{{{\mathbb{R}}^{n}}}{\left( {{G}_{1}}\left( s,x,{{u}_{\mu }}\left( s \right),\mu_1 \left( s \right) \right)-{{G}_{1}}\left( s,x,{{u}_{\nu }}\left( s \right),\nu_1 \left( s \right) \right) \right)\left( {{u}_{\mu }}\left( s \right)-{{u}_{\nu }}\left( s \right) \right)dx}}ds \\
& +\beta \int_{\tau }^{t}{\left\| \sigma \left( s,{{u}_{\mu }}\left( s \right),\mu_1 \left( s \right) \right)-\sigma \left( s,{{u}_{\nu }}\left( s \right),\nu_1 \left( s \right) \right) \right\|_{{{L}_{2}}\left( {{l}^{2}},{{L}^{2}}\left( {{\mathbb{R}}^{n}} \right) \right)}^{2}}ds \\
& +2\beta \int_{\tau }^{t}{\left( \sigma \left( s,{{u}_{\mu }}\left( s \right),\mu_1 \left( s \right) \right)-\sigma \left( s,{{u}_{\nu }}\left( s \right),\nu_1 \left( s \right) \right),{{u}_{\mu }}\left( s \right)-{{u}_{\nu }}\left( s \right) \right)}dW\left( s \right) \\
& +2\alpha \int_{\tau }^{t}{\left( \delta \left( s,{{v}_{\mu }}\left( s \right) \right)-\delta \left( s,{{v}_{\nu }}\left( s \right) \right),{{v}_{\mu }}\left( s \right)-{{v}_{\nu }}\left( s \right) \right)}dW\left( s \right).
\end{split}
\end{align}
Taking the expectation of both side of \eqref{m6}, we can get
\begin{align}
	\begin{split}\label{m7}
& \mathbb{E}\left( \beta {{\left\| {{u}_{\mu }}\left( t \right)-{{u}_{\nu }}\left( t \right) \right\|}^{2}}+\alpha {{\left\| {{v}_{\mu }}\left( t \right)-{{v}_{\nu }}\left( t \right) \right\|}^{2}} \right)+2\beta \mathbb{E}\left( \int_{\tau }^{t}{{{\left\| \nabla \left( {{u}_{\mu }}\left( s \right)-{{u}_{\nu }}\left( s \right) \right) \right\|}^{2}}}ds \right) \\
& +2\lambda \mathbb{E}\left( \int_{\tau }^{t}{\left( \beta {{\left\| {{u}_{\mu }}\left( s \right)-{{u}_{\nu }}\left( s \right) \right\|}^{2}}+\alpha {{\left\| {{v}_{\mu }}\left( s \right)-{{v}_{\nu }}\left( s \right) \right\|}^{2}} \right)}ds \right) \\
 \le&-2\beta \mathbb{E}\left( \int_{\tau }^{t}{\int_{{{\mathbb{R}}^{n}}}{\left( f\left( s,x,{{u}_{\mu }}\left( s \right),\mu_1 \left( s \right) \right)-f\left( s,x,{{u}_{\nu }}\left( s \right),\nu_1 \left( s \right) \right) \right)\left( {{u}_{\mu }}\left( s \right)-{{u}_{\nu }}\left( s \right) \right)dx}}ds \right) \\
& +2\beta \mathbb{E}\left( \int_{\tau }^{t}{\int_{{{\mathbb{R}}^{n}}}{\left( {{G}_{1}}\left( s,x,{{u}_{\mu }}\left( s \right),\mu_1 \left( s \right) \right)-{{G}_{1}}\left( s,x,{{u}_{\nu }}\left( s \right),\nu_1 \left( s \right) \right) \right)\left( {{u}_{\mu }}\left( s \right)-{{u}_{\nu }}\left( s \right) \right)dx}}ds \right) \\
& +\beta \mathbb{E}\left( \int_{\tau }^{t}{\left\| \sigma \left( s,{{u}_{\mu }}\left( s \right),\mu_1 \left( s \right) \right)-\sigma \left( s,{{u}_{\nu }}\left( s \right),\nu_1 \left( s \right) \right) \right\|_{{{L}_{2}}\left( {{l}^{2}},{{L}^{2}}\left( {{\mathbb{R}}^{n}} \right) \right)}^{2}}ds \right) \\
& +2\alpha \mathbb{E}\left( \int_{\tau }^{t}{\left( {{G}_{2}}\left( s \right),{{v}_{\mu }}\left( s \right)-{{v}_{\nu }}\left( s \right) \right)}ds \right)+\alpha \mathbb{E}\left( \int_{\tau }^{t}{\left\| \delta \left( s,{{v}_{\mu }}\left( s \right) \right)-\delta \left( s,{{v}_{\nu }}\left( s \right) \right) \right\|_{{{L}_{2}}\left( {{l}^{2}},{{L}^{2}}\left( {{\mathbb{R}}^{n}} \right) \right)}^{2}}ds \right).
\end{split}
\end{align}
For the first term on the right side of \eqref{m7}, by \eqref{c3}, \eqref{c4} and H\"{o}lder inequality, we can get
\begin{align}
	\begin{split}\label{m8}
& -2\beta \mathbb{E}\left( \int_{\tau }^{t}{\int_{{{\mathbb{R}}^{n}}}{\left( f\left( s,x,{{u}_{\mu }}\left( s \right),\mu_1 \left( s \right) \right)-f\left( s,x,{{u}_{\nu }}\left( s \right),\nu_1 \left( s \right) \right) \right)\left( {{u}_{\mu }}\left( s \right)-{{u}_{\nu }}\left( s \right) \right)dx}}ds \right) \\
 =&-2\beta \mathbb{E}\left( \int_{\tau }^{t}{\int_{{{\mathbb{R}}^{n}}}{\left( f\left( s,x,{{u}_{\mu }}\left( s \right),\mu_1 \left( s \right) \right)-f\left( s,x,{{u}_{\nu }}\left( s \right),\mu_1 \left( s \right) \right) \right)\left( {{u}_{\mu }}\left( s \right)-{{u}_{\nu }}\left( s \right) \right)dx}}ds \right) \\
& -2\beta \mathbb{E}\left( \int_{\tau }^{t}{\int_{{{\mathbb{R}}^{n}}}{\left( f\left( s,x,{{u}_{\nu }}\left( s \right),\mu_1 \left( s \right) \right)-f\left( s,x,{{u}_{\nu }}\left( s \right),\nu_1 \left( s \right) \right) \right)\left( {{u}_{\mu }}\left( s \right)-{{u}_{\nu }}\left( s \right) \right)dx}}ds \right) \\
 \le& 2\beta{{\left\| {{\phi }_{4}} \right\|}_{{{L}^{\infty }}\left( \mathbb{R},{{L}^{\infty }}\left( {{\mathbb{R}}^{n}} \right) \right)}}\int_{\tau }^{t}{\mathbb{E}\left[ {{\left\| {{u}_{\mu }}\left( s \right)-{{u}_{\nu }}\left( s \right) \right\|}^{2}} \right]}ds \\
& +\beta\int_{\tau }^{t}{{{\left\| {{\phi }_{2}}\left( s \right) \right\|}^{2}}\mathbb{E}\left[ {{\left\| {{u}_{\mu }}\left( s \right)-{{u}_{\nu }}\left( s \right) \right\|}^{2}} \right]}ds+\beta\int_{\tau }^{t}{{{\left( {{\mathbb{W}}_{2}}\left( \mu \left( s \right),\nu \left( s \right) \right) \right)}^{2}}}ds.
\end{split}
\end{align}
For the second term on the right side of \eqref{m7}, by \eqref{c8} we can get
\begin{align}
	\begin{split}\label{m9}
& 2\beta \mathbb{E}\left( \int_{\tau }^{t}{\int_{{{\mathbb{R}}^{n}}}{\left( {{G}_{1}}\left( s,x,{{u}_{\mu }}\left( s \right),\mu_1 \left( s \right) \right)-{{G}_{1}}\left( s,x,{{u}_{\nu }}\left( s \right),\nu_1 \left( s \right) \right) \right)\left( {{u}_{\mu }}\left( s \right)-{{u}_{\nu }}\left( s \right) \right)dx}}ds \right) \\
 \le& 2\beta \mathbb{E}\left( \int_{\tau }^{t}{\int_{{{\mathbb{R}}^{n}}}{{{\phi }_{7}}\left( s,x \right)\left( \left( {{u}_{\mu }}\left( s \right)-{{u}_{\nu }}\left( s \right) \right)+{{\mathbb{W}}_{2}}\left( \mu \left( s \right),\nu \left( s \right) \right) \right)\left| {{u}_{\mu }}\left( s \right)-{{u}_{\nu }}\left( s \right) \right|dx}}ds \right) \\
 \le& 2\beta{{\left\| {{\phi }_{7}} \right\|}_{{{L}^{\infty }}\left( \mathbb{R},{{L}^{\infty }}\left( {{\mathbb{R}}^{n}} \right) \right)}}\int_{\tau }^{t}{\mathbb{E}\left[ {{\left\| {{u}_{\mu }}\left( s \right)-{{u}_{\nu }}\left( s \right) \right\|}^{2}} \right]}ds \\
& +\beta\int_{\tau }^{t}{{{\left\| {{\phi }_{7}}\left( s \right) \right\|}^{2}}\mathbb{E}\left[ {{\left\| {{u}_{\mu }}\left( s \right)-{{u}_{\nu }}\left( s \right) \right\|}^{2}} \right]}ds+\beta\int_{\tau }^{t}{{{\left( {{\mathbb{W}}_{2}}\left( \mu \left( s \right),\nu \left( s \right) \right) \right)}^{2}}}ds.	
	\end{split}
\end{align}
For the last three terms on the right side of \eqref{m7}, by \eqref{c00} and \eqref{c01} we can get
\begin{align}
\begin{split}\label{m10}
& \beta \mathbb{E}\left( \int_{\tau }^{t}{\left\| \sigma \left( s,{{u}_{\mu }}\left( s \right),\mu_1 \left( s \right) \right)-\sigma \left( s,{{u}_{\nu }}\left( s \right),\nu_1 \left( s \right) \right) \right\|_{{{L}_{2}}\left( {{l}^{2}},{{L}^{2}}\left( {{\mathbb{R}}^{n}} \right) \right)}^{2}}ds \right) \\
& +2\alpha \mathbb{E}\left( \int_{\tau }^{t}{\left( {{G}_{2}}\left( s \right),{{v}_{\mu }}\left( s \right)-{{v}_{\nu }}\left( s \right) \right)}ds \right)+\alpha \mathbb{E}\left( \int_{\tau }^{t}{\left\| \delta \left( s,{{v}_{\mu }}\left( s \right) \right)-\delta \left( s,{{v}_{\nu }}\left( s \right) \right) \right\|_{{{L}_{2}}\left( {{l}^{2}},{{L}^{2}}\left( {{\mathbb{R}}^{n}} \right) \right)}^{2}}ds \right) \\
 \le& \left( {{L}_{\sigma }}+2\left\| {{G}_{2}} \right\|_{{{L}^{\infty }}\left( \mathbb{R},{{L}^{\infty}}\left( {{\mathbb{R}}^{n}} \right) \right)} \right)\int_{\tau }^{t}{\mathbb{E}\left( \beta {{\left\| {{u}_{\mu }}\left( s \right)-{{u}_{\nu }}\left( s \right) \right\|}^{2}}+\alpha {{\left\| {{v}_{\mu }}\left( s \right)-{{v}_{\nu }}\left( s \right) \right\|}^{2}} \right)}ds \\
 & +\beta {{L}_{\sigma }}\int_{\tau }^{t}{{{\left( {{\mathbb{W}}_{2}}\left( \mu \left( s \right),\nu \left( s \right) \right) \right)}^{2}}}ds.		
\end{split}
                                                        \end{align}
By \eqref{m7}-\eqref{m10} we can get
\begin{align}
	\begin{split}\label{m11}
& \mathbb{E}\left( \beta {{\left\| {{u}_{\mu }}\left( t \right)-{{u}_{\nu }}\left( t \right) \right\|}^{2}}+\alpha {{\left\| {{v}_{\mu }}\left( t \right)-{{v}_{\nu }}\left( t \right) \right\|}^{2}} \right)+2\beta \mathbb{E}\left( \int_{\tau }^{t}{{{\left\| \nabla \left( {{u}_{\mu }}\left( s \right)-{{u}_{\nu }}\left( s \right) \right) \right\|}^{2}}}ds \right) \\
 \le& \beta \int_{\tau }^{t}{\left( {{\left\| {{\phi }_{2}}\left( s \right) \right\|}^{2}}+{{\left\| {{\phi }_{7}}\left( s \right) \right\|}^{2}} \right)\mathbb{E}\left[ {{\left\| {{u}_{\mu }}\left( s \right)-{{u}_{\nu }}\left( s \right) \right\|}^{2}} \right]}ds+\left( 2+{{L}_{\sigma }} \right)\beta \int_{\tau }^{t}{{{\left( {{\mathbb{W}}_{2}}\left( \mu \left( s \right),\nu \left( s \right) \right) \right)}^{2}}}ds \\
& +\left( -2\lambda +{{L}_{\sigma }}+2{{\left\| {{G}_{2}} \right\|}_{{{L}^{\infty }}\left( \mathbb{R},{{L}^{\infty }}\left( {{\mathbb{R}}^{n}} \right) \right)}} \right)\int_{\tau }^{t}{\mathbb{E}\left( \beta {{\left\| {{u}_{\mu }}\left( s \right)-{{u}_{\nu }}\left( s \right) \right\|}^{2}}+\alpha {{\left\| {{v}_{\mu }}\left( s \right)-{{v}_{\nu }}\left( s \right) \right\|}^{2}} \right)}ds \\
& +\left( 2{{\left\| {{\phi }_{7}} \right\|}_{{{L}^{\infty }}\left( \mathbb{R},{{L}^{\infty }}\left( {{\mathbb{R}}^{n}} \right) \right)}}+2{{\left\| {{\phi }_{4}} \right\|}_{{{L}^{\infty }}\left( \mathbb{R},{{L}^{\infty }}\left( {{\mathbb{R}}^{n}} \right) \right)}} \right)\int_{\tau }^{t}{\mathbb{E}\left( \beta {{\left\| {{u}_{\mu }}\left( s \right)-{{u}_{\nu }}\left( s \right) \right\|}^{2}}+\alpha {{\left\| {{v}_{\mu }}\left( s \right)-{{v}_{\nu }}\left( s \right) \right\|}^{2}} \right)}ds.	
	\end{split}
\end{align}
By \eqref{m11} and Gronwall inequality  we can get for all $t\in[\tau,\tau+T]$,
\begin{align}
	\begin{split}\label{m0}
\mathbb{E}\left( \beta {{\left\| {{u}_{\mu }}\left( t \right)-{{u}_{\nu }}\left( t \right) \right\|}^{2}}+\alpha {{\left\| {{v}_{\mu }}\left( t \right)-{{v}_{\nu }}\left( t \right) \right\|}^{2}} \right)\le {{c}_{T}}\int_{\tau }^{t}{{{\left( {{\mathbb{W}}_{2}}\left( \mu \left( s \right),\nu \left( s \right) \right) \right)}^{2}}}ds,
	\end{split}
\end{align}
where ${{c}_{T}}>0$ is a positive constant. By \eqref{m0} we can get
\begin{align}
	\begin{split}\label{m12}
& {{e}^{-2\eta t}}\mathbb{E}\left( \beta {{\left\| {{u}_{\mu }}\left( t \right)-{{u}_{\nu }}\left( t \right) \right\|}^{2}}+\alpha {{\left\| {{v}_{\mu }}\left( t \right)-{{v}_{\nu }}\left( t \right) \right\|}^{2}} \right) \\
 \le &{{c}_{T}}\int_{\tau }^{t}{{{e}^{-2\eta \left( t-s \right)}}{{e}^{-2\eta s}}}{{\left( {{\mathbb{W}}_{2}}\left( \mu \left( s \right),\nu \left( s \right) \right) \right)}^{2}}ds\le \frac{{{c}_{T}}}{2\eta }\underset{s\in \left[ \tau ,t \right]}{\mathop{\sup }}\,{{e}^{-2\eta s}}{{\left( {{\mathbb{W}}_{2}}\left( \mu \left( s \right),\nu \left( s \right) \right) \right)}^{2}}.
	\end{split}
\end{align}
By \eqref{m12} and the definition of ${{\mathbb{W}}_{2}}\left( \cdot ,\cdot  \right)$,
we can get
\begin{align}
	\begin{split}\label{m13}
{{e}^{-2\eta t}}{{\left( {{\mathbb{W}}_{2}}\left( {{\mathcal{L}}_{{{u}_{\mu }}\left( t \right)}},{{\mathcal{L}}_{{{u}_{\nu }}\left( t \right)}} \right) \right)}^{2}}\le \frac{{{c}_{T}}}{2\eta }\underset{s\in \left[ \tau ,t \right]}{\mathop{\sup }}\,{{e}^{-2\eta s}}{{\left( {{\mathbb{W}}_{2}}\left( \mu \left( s \right),\nu \left( s \right) \right) \right)}^{2}},\quad \forall t\in \left[ \tau ,\tau +T \right].
	\end{split}
\end{align}
By \eqref{m13} and the definition of ${{d}_{C\left([\tau ,\tau +T],{{\mathcal{P}}_{2}}\left( {\mathbb{L}^{2}}\left( {{\mathbb{R}}^{n}} \right) \right) \right)}}$,
we can get
\begin{align}
	\begin{split}\label{m14}
		{{d}_{C\left( [\tau ,\tau +T],{{\mathcal{P}}_{2}}\left( {\mathbb{L}^{2}}\left( {{\mathbb{R}}^{n}} \right) \right) \right)}}\left( {{\mathcal{L}}_{{{u}_{\mu }}}},{{\mathcal{L}}_{{{u}_{\nu }}}} \right)\le {{\left( \frac{{{c}_{T}}}{2\eta } \right)}^{\frac{1}{2}}}{{d}_{C\left( [\tau ,\tau +T],{{\mathcal{P}}_{2}}\left( {\mathbb{L}^{2}}\left( {{\mathbb{R}}^{n}} \right) \right) \right)}}\left( \mu ,\nu  \right).
\end{split}
\end{align}
Here we assume that $\eta$ is a enough large positive constant such that $\frac{{{c}_{T}}}{2\eta }=\frac{1}{4}$, by \eqref{m14} and the definition of ${{\Phi }^{{{\xi }_{0}}}}$, we can get that for all  $\mu ,\nu \in C\left( [\tau ,\tau +T],{{\mathcal{P}}_{2}}\left( {{\mathbb L}^{2}}\left( {{\mathbb{R}}^{n}} \right) \right) \right)$,
\begin{align}
	\begin{split}\label{m15}
{{d}_{C\left( [\tau ,\tau +T],{{\mathcal{P}}_{2}}\left( {\mathbb{L}^{2}}\left( {{\mathbb{R}}^{n}} \right) \right) \right)}}\left( {{\Phi }^{{{\xi }_{0}}}}\left( \mu  \right),{{\Phi }^{{{\xi }_{0}}}}\left( \nu  \right) \right)\le \frac{1}{2}{{d}_{C\left( [\tau ,\tau +T],{{\mathcal{P}}_{2}}\left( {\mathbb{L}^{2}}\left( {{\mathbb{R}}^{n}} \right) \right) \right)}}\left( \mu ,\nu  \right),
	\end{split}
\end{align}
which shows ${{\Phi }^{{{\xi }_{0}}}}$ is a contractive map in $C\left( [\tau ,\tau +T],{{\mathcal{P}}_{2}}\left( {\mathbb{L}^{2}}\left( {{\mathbb{R}}^{n}} \right) \right) \right)$ with metric ${{d}_{C\left( [\tau ,\tau +T],{{\mathcal{P}}_{2}}\left( {\mathbb{L}^{2}}\left( {{\mathbb{R}}^{n}} \right) \right) \right)}}$. Therefore, ${{\Phi }^{{{\xi }_{0}}}}$ has a unique fixed point $\bar{\mu }\in C\left([\tau ,\tau +T],{{\mathcal{P}}_{2}}\left( {\mathbb{L}^{2}}\left( {{\mathbb{R}}^{n}} \right) \right) \right)$. Then ${{k}_{{\bar{\mu }}}}$ is a solution of \eqref{c18}-\eqref{c19}. On the other hand, for every solution $k$ of \eqref{c18}-\eqref{c19}, if we set $\mu \left( t \right)={{\mathcal{L}}_{k\left( t \right)}}$ for $t\in[\tau,\tau+T]$, then $\mu$ is a fixed point of ${{\Phi }^{{{\xi }_{0}}}}$ in $C\left( [\tau ,\tau +T],{{\mathcal{P}}_{2}}\left( {\mathbb{L}^{2}}\left( {{\mathbb{R}}^{n}} \right) \right) \right)$, which along with the uniqueness of fixed points of ${{\Phi }^{{{\xi }_{0}}}}$ and \eqref{m0} implies the uniqueness of solutions of
\eqref{c18}-\eqref{c19}.

Since $T>0$ was taken arbitrarily and the uniqueness is ensured, this implies that \eqref{c18}-\eqref{c19} has a unique solution up to every time $T>0$. Hence \eqref{c18}-\eqref{c19} has solution in $\mathbb{R}$.

\end{proof}

\section{Uniform Estimates}
In this section, we give the uniform estimates of the solutions of \eqref{c18}-\eqref{c19} which is devoted  to the existence and uniqueness of $\mathcal{D}$-pullback measure attractors.
\begin{lem}\label{lem4.1}
Suppose assumption  $\mathbf{(H_{1})-({H_{4}})}$, \eqref{c21} and \eqref{c24} hold, then for every $\tau\in \mathbb{R}$ and ${{D}_{1}}=\left\{ {{D}_{1}}\left( t \right):t\in \mathbb{R} \right\}\in {{\mathcal{D}}_{0}}$, there exists $T=T(\tau,D_{1})>0$ such that for all $t\geq T$, the solution  $k=(u,v)$ of  \eqref{c18}-\eqref{c19} satisfies
\begin{align}
	\begin{split}\label{d1}
& \mathbb{E}\left( {{\left\| \left( {{u}}\left( \tau ,\tau -t,\xi _{0} \right),{{v}}\left( \tau ,\tau -t,\xi _{0} \right) \right) \right\|_{{{\mathbb{L}}^{2}}\left( {{\mathbb{R}}^{n}} \right)}^{2}}} \right) \\
 \le& {{M}_{1}}+{{M}_{1}}\int_{-\infty }^{\tau }{{{e}^{\eta \left( s-\tau  \right)}}\left( {{\left\| {{\phi }_{g}}\left( s \right) \right\|}^{2}}+\left\| {{\theta }_{1}}\left( s \right) \right\|_{{{L}^{2}}\left( {{\mathbb{R}}^{n}},{{l}^{2}} \right)}^{2}+\left\| {{\theta }_{2}}\left( s \right) \right\|_{{{L}^{2}}\left( {{\mathbb{R}}^{n}},{{l}^{2}} \right)}^{2} \right)}ds,
	\end{split}
\end{align}
and
\begin{align}\label{d2}
	\begin{split}
	  & \int_{\tau -t}^{\tau }{{{e}^{\eta \left( s-\tau  \right)}}\mathbb{E}\left( \left\| {{u}}\left( \tau ,\tau -s,\xi_0  \right) \right\|_{{{H}^{1}}\left( {{\mathbb{R}}^{n}} \right)}^{2}+\left\| {{v}}\left( \tau ,\tau -s,\xi_0 \right) \right\|_{{{L}^{2}}\left( {{\mathbb{R}}^{n}} \right)}^{2} \right)}ds \\
	  \le& {{M}_{1}}+{{M}_{1}}\int_{-\infty }^{\tau }{{{e}^{\eta \left( s-\tau \right)}}\left( \left\| {{\phi }_{g}}\left( s \right) \right\|_{{{L}^{2}}\left( {{\mathbb{R}}^{n}} \right)}^{2}+\left\| {{\theta }_{1}}\left( s \right) \right\|_{{{L}^{2}}\left( {{\mathbb{R}}^{n}},{{l}^{2}} \right)}^{2}+\left\| {{\theta }_{2}}\left( s \right) \right\|_{{{L}^{2}}\left( {{\mathbb{R}}^{n}},{{l}^{2}} \right)}^{2} \right)}ds,
     \end{split}
     \end{align}
where $\xi _{0}\in L_{{{\mathcal{F}}_{\tau -t}}}^{2}\left( \Omega ,\mathbb L^2(\R^n) \right)$ with ${{\mathcal{L}}_{{{\xi }_{0}}}}\in {{D}_{1}}\left( \tau -t \right)$,  $\eta>0$ is the same number as in  \eqref{c21} and $M_{1}$ is a positive constant independent of $\tau$  and $D_{1}$.
\end{lem}
\begin{proof}
	By \eqref{m3} and Ito's formula  we can get that for all $t\geq\tau$,
\begin{align}\label{d5}
	\begin{split}
	  & \frac{d}{dt}\mathbb{E}\left( \beta  {{\left\| u\left( t \right) \right\|}^{2}}+\alpha {{\left\| v\left( t \right) \right\|}^{2}} \right)+2\beta  \mathbb{E}\left( {{\left\| \nabla u\left( t \right) \right\|}^{2}} \right)+2\beta \lambda  \mathbb{E}\left( {{\left\| u\left( t \right) \right\|}^{2}} \right)+2\alpha \gamma  \mathbb{E}\left( {{\left\| v\left( t \right) \right\|}^{2}} \right) \\
	   =&-2\beta  \mathbb{E}\left( \left( f\left( t,x,u\left( t \right),{{\mathcal{L}}_{u\left( t \right)}} \right),u\left( t \right) \right) \right)+2\alpha  \mathbb{E}\left( {{G}_{2}}\left( t \right),v\left( t \right) \right)+\alpha  \mathbb{E}\left( \left\| {\delta}\left( t,v\left( t \right) \right) \right\|_{{{L}_{2}}\left( {{l}^{2}},{{L}^{2}}\left( {{\mathbb{R}}^{n}} \right) \right)}^{2} \right) \\
	  & +2\beta  \mathbb{E}\left( {{G}_{1}}\left( t,x,u\left( t,x \right),{{\mathcal{L}}_{u\left( t \right)}} \right),u\left( t \right) \right)+\beta  \mathbb{E}\left( \left\| {{\sigma }}\left( t,u\left( t \right),{{\mathcal{L}}_{u\left( s \right)}} \right) \right\|_{{{L}_{2}}\left( {{l}^{2}},{{L}^{2}}\left( {{\mathbb{R}}^{n}} \right) \right)}^{2} \right).
\end{split}
\end{align}
For the first term on the right-hand side of \eqref{d5}, by \eqref{c2}  we can get
\begin{align}\label{d6}
	\begin{split}
& -2\beta \mathbb{E}\left( f\left( t,x,u\left( t,x \right),{{\mathcal{L}}_{u\left( t \right)}} \right),u\left( t \right) \right) \\
 \le& 2\beta \mathbb{E}\left( \int_{{{\mathbb{R}}^{n}}}{{{\phi }_{1}}\left( t,x \right)\left( 1+{{\left\vert u \right\vert}^{2}} \right)dx} \right)+2\beta \mathbb{E}\left( \int_{{{\mathbb{R}}^{n}}}{{{\psi }_{1}}\left( x \right)\mathbb{E}{{\left\| u\left( t \right) \right\|}^{2}}dx} \right) \\
& -2\beta {{\alpha }_{1}}\mathbb{E}\left( \int_{{{\mathbb{R}}^{n}}}{{{\left\vert {{u}}\left( t,x \right) \right\vert}^{p}}dx} \right) \\
 \le& 2\beta {{\left\| {{\phi }_{1}}\left( t \right) \right\|}_{_{{{L}^{1}}\left( {{\mathbb{R}}^{n}} \right)}}}-2\beta {{\alpha }_{1}}\mathbb{E}\left( \left\| {{u}}\left( t \right) \right\|_{{{L}^{p}}\left( {{\mathbb{R}}^{n}} \right)}^{p} \right)+2\beta \left( {{\left\| {{\phi }_{1}}\left( t \right) \right\|}_{_{{{L}^{\infty }}\left( {{\mathbb{R}}^{n}} \right)}}}+{{\left\| {{\psi }_{1}} \right\|}_{_{{{L}^{1}}\left( {{\mathbb{R}}^{n}} \right)}}} \right)\mathbb{E}\left( {{\left\| u\left( t \right) \right\|}^{2}} \right).
\end{split}
\end{align}
For the second and fourth term on the right-hand side of \eqref{d5}, by \eqref{c7}  we can get
\begin{align}\label{d7}
	\begin{split}
   & 2\beta \mathbb{E}\left( {{G}_{1}}\left( t,x,u\left( t,x \right),{{\mathcal{L}}_{u\left( s \right)}} \right),u\left( t \right) \right)+2\alpha  \mathbb{E}\left( {{G}_{2}}\left( t,x \right),v\left( t \right) \right) \\
   \le &2\beta \mathbb{E}\left( \int_{{{\mathbb{R}}^{n}}}{\left\vert u\left( t,x \right) \right\vert{{\phi }_{g}}\left( t,x \right)dx} \right)+\lambda \alpha \mathbb{E}\left( {{\left\| v\left( t \right) \right\|}^{2}} \right)+{{\lambda }^{-1}}\alpha \left\| {{G}_{2}}\left( t \right) \right\|_{{{L}^{2}}\left( {{\mathbb{R}}^{n}} \right)}^{2} \\
  & +2\beta \mathbb{E}\left( \int_{{{\mathbb{R}}^{n}}}{\left\vert {{u}}\left( t,x \right) \right\vert\left[ {{\phi }_{7}}\left( t,x \right)\left\vert {{u}}\left( t,x \right) \right\vert+{{\psi }_{g}}\left( x \right)\mu \sqrt{{{\left\| \cdot  \right\|}^{2}}} \right]dx} \right) \\
   \le& \beta\mathbb{E}\left( \int_{{{\mathbb{R}}^{n}}}{\lambda {{\left\vert {{u}}\left( t\right) \right\vert}^{2}}+{{\lambda }^{-1}}{{\left\vert {{\phi }_{g}}\left( t,x \right) \right\vert}^{2}}dx} \right)  +2\beta \mathbb{E}\left( \int_{{{\mathbb{R}}^{n}}}{\left\vert {{\psi }_{g}}\left( x \right) \right\vert\left[ {{\left\vert {{u}}\left( t \right) \right\vert}^{2}}+{{\left( \mu \sqrt{{{\left\| \cdot  \right\|}^{2}}} \right)}^{2}} \right]dx} \right) \\
  & +2\beta \mathbb{E}\left( \int_{{{\mathbb{R}}^{n}}}{{{\phi }_{7}}\left( t,x \right){{\left\vert u\left( t \right) \right\vert}^{2}}dx} \right)+\lambda \alpha \mathbb{E}\left( {{\left\| v\left( t \right) \right\|}^{2}} \right)+{{\lambda }^{-1}}\alpha \left\| {{G}_{2}}\left( t \right) \right\|_{{{L}^{2}}\left( {{\mathbb{R}}^{n}} \right)}^{2} \\
   \le& {{\lambda }^{-1}}\beta \mathbb{E}\left( \left\| {{\phi }_{g}}\left( t \right) \right\|_{{{L}^{2}}\left( {{\mathbb{R}}^{n}} \right)}^{2} \right)+\lambda \alpha \mathbb{E}\left( {{\left\| v\left( t \right) \right\|}^{2}} \right)+{{\lambda }^{-1}}\alpha \left\| {{G}_{2}}\left( t \right) \right\|_{{{L}^{2}}\left( {{\mathbb{R}}^{n}} \right)}^{2} \\
  & +\beta \left( \lambda +2{{\left\| {{\phi }_{7}}\left( t \right) \right\|}_{{{L}^{\infty }}\left( {{\mathbb{R}}^{n}} \right)}}+{{\left\| {{\psi }_{g}} \right\|}_{{{L}^{1}}\left( {{\mathbb{R}}^{n}} \right)}}+{{\left\| {{\psi }_{g}} \right\|}_{{{L}^{\infty }}\left( {{\mathbb{R}}^{n}} \right)}} \right)\mathbb{E}\left( {{\left\| u\left( t \right) \right\|}^{2}} \right).
\end{split}
\end{align}
For the third and fifth term on the right-hand side of \eqref{d5}, by \eqref{c16} and \eqref{c17} we can get
	\begin{align}\label{d8}
		\begin{split}
	 & \beta \mathbb{E}\left( \left\| {{\sigma }}\left( t,u\left( t \right),{{\mathcal{L}}_{u\left( t \right)}} \right) \right\|_{{{L}_{2}}\left( {{l}^{2}},{{L}^{2}}\left( {{\mathbb{R}}^{n}} \right) \right)}^{2} \right)+\alpha  \mathbb{E}\left( \left\| {\delta}\left( t,v\left( t \right) \right) \right\|_{{{L}_{2}}\left( {{l}^{2}},{{L}^{2}}\left( {{\mathbb{R}}^{n}} \right) \right)}^{2} \right) \\
	 \le&  8\beta {{\left\| w \right\|}^{2}}{{{\left\| {{\beta }_{1}} \right\|}_{l^{2}}^{2}}}+2\left( \beta \left\| {{\theta }_{1}}\left( t \right) \right\|_{{{L}^{2}}\left( {{\mathbb{R}}^{n}},{{l}^{2}} \right)}^{2}+\alpha \left\| {{\theta }_{2}}\left( t \right) \right\|_{{{L}^{2}}\left( {{\mathbb{R}}^{n}},{{l}^{2}} \right)}^{2} \right) \\
	 & +\left( 2 \|\delta\|^2_{l^2}+8 {{\left\| w \right\|}^{2}}\left\| {{\beta }_{1}} \right\|_{{{l}^{2}}}^{2}+4 \left\| w \right\|_{{{L}^{\infty }}\left( {{\mathbb{R}}^{n}} \right)}^{2}\left\| {{\gamma }_{1}} \right\|_{{{l}^{2}}}^{2} \right)\mathbb{E}\left( \beta {{\left\| u\left( t \right) \right\|}^{2}}+\alpha {{\left\| v\left( t \right) \right\|}^{2}} \right).
	\end{split}
	\end{align}
 It follows from \eqref{d5}-\eqref{d8} that for all  $t\geq\tau $,
	\begin{align}\label{d9}
		\begin{split}
		 & \frac{d}{dt}\mathbb{E}\left( \beta  {{\left\| u\left( t \right) \right\|}^{2}}+\alpha {{\left\| v\left( t \right) \right\|}^{2}} \right)+2\beta  \mathbb{E}\left( {{\left\| \nabla u\left( s \right) \right\|}^{2}} \right) +2\beta {{\alpha }_{1}}\mathbb{E}\left( \left\| u\left( t \right) \right\|_{{{L}^{p}}\left( {{\mathbb{R}}^{n}} \right)}^{p} \right) \\
		 \le &8\beta {{\left\| w \right\|}^{2}}{{{\left\| {{\beta }_{1}} \right\|}_{l^{2}}^{2}}}+2\left( \beta \left\| {{\theta }_{1}}\left( t \right) \right\|_{{{L}^{2}}\left( {{\mathbb{R}}^{n}},{{l}^{2}} \right)}^{2}+\alpha \left\| {{\theta }_{2}}\left( t \right) \right\|_{{{L}^{2}}\left( {{\mathbb{R}}^{n}},{{l}^{2}} \right)}^{2} \right) \\
		& +2\beta {{\left\| {{\phi }_{1}}\left( t \right) \right\|}_{_{{{L}^{1}}\left( {{\mathbb{R}}^{n}} \right)}}}+{{\lambda }^{-1}}\alpha \left\| {{G}_{2}}\left( t \right) \right\|_{{{L}^{2}}\left( {{\mathbb{R}}^{n}} \right)}^{2}+{{\lambda }^{-1}}\beta \mathbb{E}\left( \left\| {{\phi }_{g}}\left( t \right) \right\|_{{{L}^{2}}\left( {{\mathbb{R}}^{n}} \right)}^{2} \right) \\
		& +2\left( \|\delta\|^2_{l^2}+{{\left\| {{\phi }_{1}}\left( t \right) \right\|}_{_{{{L}^{\infty }}\left( {{\mathbb{R}}^{n}} \right)}}}+{{\left\| {{\psi }_{1}} \right\|}_{_{{{L}^{1}}\left( {{\mathbb{R}}^{n}} \right)}}} \right)\mathbb{E}\left( \beta {{\left\| u\left( t \right) \right\|}^{2}}+\alpha {{\left\| v\left( t \right) \right\|}^{2}} \right) \\
		& +\left( -\lambda +2{{\left\| {{\phi }_{7}}\left( t \right) \right\|}_{{{L}^{\infty }}\left( {{\mathbb{R}}^{n}} \right)}}+{{\left\| {{\psi }_{g}} \right\|}_{{{L}^{1}}\left( {{\mathbb{R}}^{n}} \right)}}+{{\left\| {{\psi }_{g}} \right\|}_{{{L}^{\infty }}\left( {{\mathbb{R}}^{n}} \right)}} \right)\mathbb{E}\left( \beta {{\left\| u\left( t \right) \right\|}^{2}}+\alpha {{\left\| v\left( t \right) \right\|}^{2}} \right) \\
		& +\left( 8 {{\left\| w \right\|}^{2}}\left\| {{\beta }_{1}} \right\|_{{{l}^{2}}}^{2}+4 \left\| w \right\|_{{{L}^{\infty }}\left( {{\mathbb{R}}^{n}} \right)}^{2}\left\| {{\gamma }_{1}} \right\|_{{{l}^{2}}}^{2} \right)\mathbb{E}\left( \beta {{\left\| u\left( t \right) \right\|}^{2}}+\alpha {{\left\| v\left( t \right) \right\|}^{2}} \right) .
	\end{split}
	\end{align}
	Multiplying \eqref{d9} by ${{e}^{\eta t}}$ and integrating the resulting inequality on $\left( \tau -t,\tau  \right)$ with $t\in {{\mathbb{R}}^{+}}$, we get
	\begin{align}\label{d10}
		\begin{split}
		& \mathbb{E}\left( \beta  {{\left\| u\left( t \right) \right\|}^{2}}+\alpha {{\left\| v\left( t \right) \right\|}^{2}} \right)+2\beta \int_{\tau -t}^{\tau }{{{e}^{\eta \left( s-\tau  \right)}}\mathbb{E}\left( {{\left\| \nabla u\left( s \right) \right\|}^{2}} \right)ds} \\
		& +2\beta {{\alpha }_{1}}\int_{\tau -t}^{\tau }{{{e}^{\eta \left( s-\tau  \right)}}\mathbb{E}\left( \left\| u\left( \tau ,\tau -s,{{\xi }_{1}} \right) \right\|_{{{L}^{p}}\left( {{\mathbb{R}}^{n}} \right)}^{p} \right)ds} \\
		 \le& \mathbb{E}\left( \beta  {{\left\| {{\xi }_{1}} \right\|}^{2}}+\alpha {{\left\| {{\xi }_{2}} \right\|}^{2}} \right){{e}^{-\eta t}}+{{\lambda }^{-1}}\beta \int_{\tau -t}^{\tau }{{{e}^{\eta \left( s-\tau  \right)}}\left\| {{\phi }_{g}}\left( s \right) \right\|_{_{{{L}^{2}}\left( {{\mathbb{R}}^{n}} \right)}}^{2}ds} \\
		& +2\beta \int_{\tau -t}^{\tau }{{{e}^{\eta \left( s-\tau  \right)}}{{\left\| {{\phi }_{1}}\left( s \right) \right\|}_{_{{{L}^{1}}\left( {{\mathbb{R}}^{n}} \right)}}}ds}+{{\lambda }^{-1}}\alpha \int_{\tau -t}^{\tau }{{{e}^{\eta \left( s-\tau  \right)}}\left\| {{G}_{2}}\left( s \right) \right\|_{_{{{L}^{2}}\left( {{\mathbb{R}}^{n}} \right)}}^{2}ds} \\
		& +2\int_{\tau -t}^{\tau }{{{e}^{\eta \left( s-\tau  \right)}}\left( \beta \left\| {{\theta }_{1}}\left( s \right) \right\|_{{{L}^{2}}\left( {{\mathbb{R}}^{n}},{{l}^{2}} \right)}^{2}+\alpha \left\| {{\theta }_{2}}\left( s \right) \right\|_{{{L}^{2}}\left( {{\mathbb{R}}^{n}},{{l}^{2}} \right)}^{2} \right)ds}+8\beta {{\left\| w \right\|}^{2}}\left\| {{\beta }_{1}} \right\|_{{{l}^{2}}}^{2}{{\eta }^{-1}} \\
		& +\int_{\tau -t}^{\tau }{{{e}^{\eta \left( s-\tau  \right)}}2\left( \|\delta\|^2_{l^2}+{{\left\| {{\phi }_{7}}\left( s \right) \right\|}_{_{{{L}^{\infty }}\left( {{\mathbb{R}}^{n}} \right)}}}+{{\left\| {{\psi }_{1}} \right\|}_{_{{{L}^{1}}\left( {{\mathbb{R}}^{n}} \right)}}} \right)\mathbb{E}\left( \beta {{\left\| u\left( s \right) \right\|}^{2}}+\alpha {{\left\| v\left( s \right) \right\|}^{2}} \right)ds} \\
		& +\int_{\tau -t}^{\tau }{{{e}^{\eta \left( s-\tau  \right)}}\left( 2{{\left\| {{\phi }_{1}}\left( s \right) \right\|}_{{{L}^{\infty }}\left( {{\mathbb{R}}^{n}} \right)}}+{{\left\| {{\psi }_{g}} \right\|}_{{{L}^{1}}\left( {{\mathbb{R}}^{n}} \right)}}+{{\left\| {{\psi }_{g}} \right\|}_{{{L}^{\infty }}\left( {{\mathbb{R}}^{n}} \right)}} \right)\mathbb{E}\left( \beta {{\left\| u\left( s \right) \right\|}^{2}}+\alpha {{\left\| v\left( s \right) \right\|}^{2}} \right)ds} \\
		& +\int_{\tau -t}^{\tau }{{{e}^{\eta \left( s-\tau  \right)}}\left( -\lambda +8{{\left\| w \right\|}^{2}}\left\| {{\beta }_{1}} \right\|_{{{l}^{2}}}^{2}+4\left\| w \right\|_{{{L}^{\infty }}\left( {{\mathbb{R}}^{n}} \right)}^{2}\left\| {{\gamma }_{1}} \right\|_{{{l}^{2}}}^{2} \right)\mathbb{E}\left( \beta {{\left\| u\left( s \right) \right\|}^{2}}+\alpha {{\left\| v\left( s \right) \right\|}^{2}} \right)ds}.
	\end{split}
	\end{align}
	By  \eqref{d5}-\eqref{d10} we get for all   $t\ge \tau $,
	\begin{align}\label{d11}
		\begin{split}
		& \mathbb{E}\left( \beta  {{\left\| u\left( t \right) \right\|}^{2}}+\alpha {{\left\| v\left( t \right) \right\|}^{2}} \right)+2\beta \int_{\tau -t}^{\tau }{{{e}^{\eta \left( s-\tau  \right)}}\mathbb{E}\left( {{\left\| \nabla u\left( s \right) \right\|}^{2}} \right)ds} \\
		& +2\beta {{\alpha }_{1}}\int_{\tau -t}^{\tau }{{{e}^{\eta \left( s-\tau  \right)}}\mathbb{E}\left( \left\| u\left( \tau ,\tau -s,{{\xi }_{1}} \right) \right\|_{{{L}^{p}}\left( {{\mathbb{R}}^{n}} \right)}^{p} \right)ds} \\
		 \le& \mathbb{E}\left( \beta  {{\left\| {{\xi }_{1}} \right\|}^{2}}+\alpha {{\left\| {{\xi }_{2}} \right\|}^{2}} \right){{e}^{-\eta t}}+{{c}_{1}}\\
		   &+{{c}_{1}}\int_{-\infty }^{\tau }{{{e}^{\eta \left( s-\tau  \right)}}\left( \left\| {{\theta }_{1}}\left( s \right) \right\|_{{{L}^{2}}\left( {{\mathbb{R}}^{n}},{{l}^{2}} \right)}^{2}+\left\| {{\theta }_{2}}\left( s \right) \right\|_{{{L}^{2}}\left( {{\mathbb{R}}^{n}},{{l}^{2}} \right)}^{2}+\left\| {{\phi }_{g}}\left( s \right) \right\|_{_{{{L}^{2}}\left( {{\mathbb{R}}^{n}} \right)}}^{2} \right)ds}.
			\end{split}
	\end{align}
	Note that $\mathcal{L}_{{{\xi }_{0}}}  \in D_{1}\left( \tau -t \right)$  then we have
\[\underset{t\to \infty }{\mathop{\lim }}\,{{e}^{-\eta t}}\mathbb{E}\left( \beta  {{\left\| {{\xi }_{1}} \right\|}^{2}}+\alpha {{\left\| {{\xi }_{2}} \right\|}^{2}} \right)\le \underset{t\to \infty }{\mathop{\lim }}\,{{e}^{-\eta t}}\left\| {{D}_{1}}\left( \tau -t \right) \right\|_{{{\mathcal{P}}_{2}}\left( {\mathbb{L}^{2}}\left( {{\mathbb{R}}^{n}} \right) \right)}^{2}=0.\]
	And hence there exists $T=T\left( \tau ,D_{1} \right)$ such that for all $t\ge T$,
		\begin{equation*}
	{{e}^{-\eta t}}\mathbb{E}\left( \beta  {{\left\| {{\xi }_{1}} \right\|}^{2}}+\alpha {{\left\| {{\xi }_{2}} \right\|}^{2}} \right)\le \frac{{{c}_{1}}}{\lambda },
	\end{equation*}
	which along with \eqref{d11}  concludes the proof.
\end{proof}
\begin{lem}\label{lem4.2}
	Suppose assumption  $\mathbf{(H_{1})-({H_{4}})}$, \eqref{c21} and \eqref{c24} hold, then for every $\tau\in \mathbb{R}$ and ${{D}_{1}}=\left\{ {{D}_{1}}\left( t \right):t\in \mathbb{R} \right\}\in {{\mathcal{D}}_{0}}$, there exists $T=T(\tau,D_{1})>0$ such that for all $t\geq T$, the solution  $k=(u,v)$ of  \eqref{c18}-\eqref{c19} satisfies
\begin{align*}
& \int_{\tau -1}^{\tau }{\mathbb{E}\left( \left\| u\left( s ,\tau -s,{{\xi }_{0}} \right)   \right\|_{{{H}^{1}}\left( {{\mathbb{R}}^{n}} \right)}^{2}+\left\| v\left( s ,\tau -s,{{\xi }_{0}} \right)   \right\|_{{{L}^{2}}\left( {{\mathbb{R}}^{n}} \right)}^{2} \right)}ds \\
 \le& {{M}_{2}}+{{M}_{2}}\int_{-\infty }^{\tau }{{{e}^{\eta \left( s-\tau  \right)}}\left( \left\| {{\phi }_{g}}\left( s \right) \right\|_{_{{{L}^{2}}\left( {{\mathbb{R}}^{n}} \right)}}^{2}+\left\| {{\theta }_{1}}\left( s \right) \right\|_{{{L}^{2}}\left( {{\mathbb{R}}^{n}},{{l}^{2}} \right)}^{2}+\left\| {{\theta }_{2}}\left( s \right) \right\|_{{{L}^{2}}\left( {{\mathbb{R}}^{n}},{{l}^{2}} \right)}^{2} \right)}ds,
\end{align*}
where $\xi _{0}\in L_{{{\mathcal{F}}_{\tau -t}}}^{2}\left( \Omega ,{\mathbb{L}^{2}}\left( {{\mathbb{R}}^{n}} \right) \right)$ with ${{\mathcal{L}}_{{{\xi }_{0}}}}\in {{D}_{1}}\left( \tau -t \right)$,  $\eta>0$ is the same number as in  \eqref{c21} and $M_{2}$ is a positive constant independent of $\tau$ and $D_{1}$.
\end{lem}
\begin{lem}\label{lem4.3}
Suppose assumption  $\mathbf{(H_{1})-({H_{4}})}$, \eqref{c21} and \eqref{c24} hold, then for every $\tau\in \mathbb{R}$ and ${{D}_{1}}=\left\{ {{D}_{1}}\left( t \right):t\in \mathbb{R} \right\}\in {{\mathcal{D}}_{0}}$, there exists $T=T(\tau,D_{1})>0$ such that for all $t\geq T$, the solution  $k=(u,v)$ of  \eqref{c18}-\eqref{c19} satisfies
	\begin{align*}
		 & \mathbb{E}\left( \left\| u\left( \tau ,\tau -s,{{\xi }_{0}} \right)   \right\|_{{{H}^{1}}\left( {{\mathbb{R}}^{n}} \right)}^{2} \right) \\
	 \le 	&{{M}_{3}}+{{M}_{3}}\int_{-\infty }^{\tau }{{{e}^{\eta \left( s-\tau  \right)}}\left( \left\| {{\phi }_{g}}\left( s \right) \right\|_{_{{{L}^{2}}\left( {{\mathbb{R}}^{n}} \right)}}^{2}+\left\| {{\theta }_{1}}\left( s \right) \right\|_{{{L}^{2}}\left( {{\mathbb{R}}^{n}},{{l}^{2}} \right)}^{2}+\left\| {{\theta }_{2}}\left( s \right) \right\|_{{{L}^{2}}\left( {{\mathbb{R}}^{n}},{{l}^{2}} \right)}^{2} \right)}ds,
	\end{align*}
	where $\xi _{0}\in L_{{{\mathcal{F}}_{\tau -t}}}^{2}\left( \Omega ,{\mathbb{L}^{2}}\left( {{\mathbb{R}}^{n}} \right) \right)$ with ${{\mathcal{L}}_{{{\xi }_{0}}}}\in {{D}_{1}}\left( \tau -t \right)$,  $\eta>0$ is the same number as in  \eqref{c21} and $M_{3}$ is a positive constant independent of $\tau$ but not on $D_{1}$.
\end{lem}
\begin{proof}
By \eqref{c18}  and Ito's formula we can get for all $\tau \in \mathbb{R}$, $t>1$ and $\varsigma \in \left( \tau -1,\tau  \right)$,
\begin{align}
	\label{d12}
	\begin{split}
	& {{\left\| \nabla u\left( \tau \right)  \right\|}^{2}}+2\int_{\varsigma }^{\tau }{{{\left\| \Delta u\left( s \right)   \right\|}^{2}}ds}+2\lambda \int_{\varsigma }^{\tau }{{{\left\| \nabla u\left( s \right)  \right\|}^{2}}ds} \\
	 =&{{\left\| \nabla u\left( \varsigma  \right)   \right\|}^{2}}+2\alpha \int_{\varsigma }^{\tau }{\left( v\left( s \right),\Delta u\left( s \right) \right)ds}-2\int_{\varsigma }^{\tau }{\left( \nabla f\left( s,x,u\left( s \right),{{\mathcal{L}}_{u\left( s \right)}} \right),\nabla u\left( s \right) \right)ds} \\
	& +2\int_{\varsigma }^{\tau }{\left( \nabla {{G}_{1}}\left( s,x,u\left( s \right),{{\mathcal{L}}_{u\left( s \right)}} \right),\nabla u\left( s \right) \right)ds}+\int_{\varsigma }^{\tau}{\left\| {{\sigma }}\left( s,u\left( s \right),{{\mathcal{L}}_{u\left( s \right)}} \right) \right\|_{{{L}_{2}}\left( {{l}^{2}},{{L}^{2}}\left( {{\mathbb{R}}^{n}} \right) \right)}^{2}}ds \\
	& +2\int_{\varsigma }^{\tau}{\left( {{\sigma }}\left( s,u\left( s \right),{{\mathcal{L}}_{u\left( s \right)}} \right),\nabla u\left( s \right) \right)}dW\left( s \right).
	\end{split}
\end{align}
For the second term on the right-hand side of \eqref{d12}, we can get
\begin{align}\label{012}
		2\alpha \int_{\varsigma }^{\tau }{\left( v\left( s \right),\Delta u\left( s \right) \right)ds}\text{ }\le \alpha^2 \int_{\varsigma }^{\tau }{{{\left\| v\left( s \right) \right\|}^{2}}ds}+ \int_{\varsigma }^{\tau }{{{\left\| \Delta u\left( s \right) \right\|}^{2}}ds}.
\end{align}
For the third term on the right-hand side of \eqref{d12}, by \eqref{c4} and \eqref{c5} we can get
\begin{align}
	\label{d13}
	\begin{split}
	 & -2\int_{\varsigma }^{\tau }{\left( \nabla f\left( s,x,u\left( s \right),{{\mathcal{L}}_{u\left( s \right)}} \right),\nabla u\left( s \right) \right)ds} \\
	  =&-2\int_{\varsigma }^{\tau }{\int_{{{\mathbb{R}}^{n}}}{\frac{\partial f}{\partial x}}\left( s,x,u\left( s \right),{{\mathcal{L}}_{u\left( s \right)}} \right)\cdot \nabla u\left( s \right)dxds}-2\int_{\varsigma }^{\tau }{\int_{{{\mathbb{R}}^{n}}}{\frac{\partial f}{\partial u}}\left( s,x,u\left( s \right),{{\mathcal{L}}_{u\left( s \right)}} \right){{\left\vert \nabla u\left( s \right) \right\vert}^{2}}dxds} \\
	  \le& 2\int_{\varsigma }^{\tau }{\int_{{{\mathbb{R}}^{n}}}{{{\phi }_{5}}\left( s,x \right)\left( 1+\left\vert u\left( s \right) \right\vert+\sqrt{\mathbb{E}\left( {{\left\| \cdot  \right\|}^{2}} \right)} \right)}\nabla u\left( s \right)dxds}+2\int_{\varsigma }^{\tau }{\int_{{{\mathbb{R}}^{n}}}{{{\phi }_{4}}\left( s,x \right)}{{\left\vert \nabla u\left( s \right) \right\vert}^{2}}dxds} \\
	  \le& \int_{\varsigma }^{\tau }{\left( {{\left\| {{\phi }_{5}}\left( s \right) \right\|}^{2}}+{{\left\| \nabla u\left( s \right) \right\|}^{2}} \right)ds}+\int_{\varsigma }^{\tau }{{{\left\| {{\phi }_{5}}\left( s \right) \right\|}_{{{L}^{2}}\left( {{\mathbb{R}}^{n}} \right)}}\mathbb{E}\left( {{\left\| u\left( s \right) \right\|}^{2}} \right)ds} \\
	 & +\int_{\varsigma }^{\tau }{{{\left\| {{\phi }_{5}}\left( s \right) \right\|}_{{{L}^{2}}\left( {{\mathbb{R}}^{n}} \right)}}\left( {{\left\| u\left( s \right) \right\|}^{2}}+{{\left\| \nabla u\left( s \right) \right\|}^{2}} \right)ds}+2\int_{\varsigma }^{\tau }{{{\left\| {{\phi }_{4}}\left( s \right) \right\|}_{{{L}^{\infty }}\left( {{\mathbb{R}}^{n}} \right)}}{{\left\| \nabla u\left( s \right) \right\|}^{2}}ds} \\
	  \le& \left( 1+2{{\left\| {{\phi }_{5}} \right\|}_{{{L}^{\infty }}\left( \mathbb{R},{{L}^{\infty }}\left( {{\mathbb{R}}^{n}} \right) \right)}}+2{{\left\| {{\phi }_{4}} \right\|}_{{{L}^{\infty }}\left( \mathbb{R},{{L}^{\infty }}\left( {{\mathbb{R}}^{n}} \right) \right)}} \right)\int_{\tau -1}^{\tau }{\left\| u\left( s \right) \right\|_{{{H}^{1}}\left( {{\mathbb{R}}^{n}} \right)}^{2}}ds \\
	 & +\left\| {{\phi }_{5}} \right\|_{{{L}^{\infty }}\left( \mathbb{R},{{L}^{2}}\left( {{\mathbb{R}}^{n}} \right) \right)}^{2}+{{\left\| {{\phi }_{5}} \right\|}_{{{L}^{\infty }}\left( \mathbb{R},{{L}^{1}}\left( {{\mathbb{R}}^{n}} \right) \right)}}\int_{\tau -1}^{\tau }{\mathbb{E}\left( {{\left\| u\left( s \right) \right\|}^{2}} \right)}ds.
\end{split}
\end{align}
For the fourth term on the right-hand side of \eqref{d12}, similarly by \eqref{c8} and \eqref{c9} we can get
\begin{align}\label{d14}
	\begin{split}
	 & 2\int_{\varsigma }^{\tau }{\left( \nabla {{G}_{1}}\left( s,x,u\left( s \right),{{\mathcal{L}}_{u\left( s \right)}} \right),\nabla u\left( s \right) \right)ds} \\
	  =&2\int_{\varsigma }^{\tau }{\int_{{{\mathbb{R}}^{n}}}{\frac{\partial {{G}_{1}}}{\partial x}}\left( s,u\left( s \right),{{\mathcal{L}}_{u\left( s \right)}} \right)\cdot \nabla u\left( s \right)dxds}+2\int_{\varsigma }^{\tau }{\int_{{{\mathbb{R}}^{n}}}{\frac{\partial {{G}_{1}}}{\partial u}}\left( s,u\left( s \right),{{\mathcal{L}}_{u\left( s \right)}} \right){{\left\vert \nabla u\left( s \right) \right\vert}^{2}}dxds} \\
	  \le& \left( 1+2{{\left\| {{\phi }_{7}} \right\|}_{{{L}^{\infty }}\left( \mathbb{R},{{L}^{\infty }}\left( {{\mathbb{R}}^{n}} \right) \right)}}+2{{\left\| {{\phi }_{7}} \right\|}_{{{L}^{\infty }}\left( \mathbb{R},{{L}^{1}}\left( {{\mathbb{R}}^{n}} \right) \right)}} \right)\int_{\tau -1}^{\tau }{\left\| u\left( s \right) \right\|_{{{H}^{1}}\left( {{\mathbb{R}}^{n}} \right)}^{2}}ds \\
	 & +\left\| {{\phi }_{8}} \right\|_{{{L}^{2}}\left( \tau -1,\tau ,{{L}^{1}}\left( {{\mathbb{R}}^{n}} \right) \right)}^{2}+{{\left\| {{\phi }_{5}} \right\|}_{{{L}^{\infty }}\left( \mathbb{R},{{L}^{1}}\left( {{\mathbb{R}}^{n}} \right) \right)}}\int_{\tau -1}^{\tau }{\mathbb{E}\left( {{\left\| u\left( s \right) \right\|}^{2}} \right)}ds.
\end{split}
\end{align}
For the fifth term on the right-hand side of \eqref{d12}, by \eqref{c11}-\eqref{c15} we can get
\begin{align}\label{d15}
	\begin{split}
	& \int_{\varsigma }^{\tau }{\left\| {{\sigma }}\left( s,u\left( s \right),{{\mathcal{L}}_{u\left( s \right)}} \right) \right\|_{{{L}_{2}}\left( {{l}^{2}},{{L}^{2}}\left( {{\mathbb{R}}^{n}} \right) \right)}^{2}}ds \\
	 \le& 2\sum\limits_{k=1}^{\infty }{\int_{\varsigma }^{\tau }{{{\left\| \nabla {{\theta }_{k}}\left( s \right) \right\|}^{2}}}ds}+4\sum\limits_{k=1}^{\infty }{\int_{\varsigma }^{\tau }{{{\left\| \left( \nabla w \right){{\sigma }_{1,k}}\left( s,u\left( s \right),{{\mathcal{L}}_{u\left( s \right)}} \right) \right\|}^{2}}}ds} \\
	& +4\sum\limits_{k=1}^{\infty }{\int_{\varsigma }^{\tau }{{{\left\| w\nabla u\left( s \right)\frac{\partial {{\sigma }_{1}}}{\partial u}\left( s,u\left( s \right),{{\mathcal{L}}_{u\left( s \right)}} \right) \right\|}^{2}}}ds}
	 \le {{c}_{2}}+{{c}_{2}}\int_{\tau -1}^{\tau }{\left( \mathbb{E}\left( {{\left\| u\left( s \right) \right\|}^{2}} \right)+\left\| u\left( s \right) \right\|_{{{H}^{1}}\left( {{\mathbb{R}}^{n}} \right)}^{2} \right)}ds.
\end{split}
\end{align}
By \eqref{d12}-\eqref{d15} and Lemma 5.1   that for all $\tau \in \mathbb{R}$, $t>1$ and $\varsigma \in \left( \tau -1,\tau  \right)$,
\begin{align}\label{d16}
	\begin{split}
	& \mathbb{E}\left( {{\left\| \nabla u\left( s \right) \right\|}^{2}} \right)\le \mathbb{E}\left( {{\left\| \nabla u\left( \varsigma  \right) \right\|}^{2}} \right) +{{c}_{3}}+{{c}_{3}}\int_{\tau -1}^{\tau }{\mathbb{E}\left( \left\| u\left( s \right) \right\|_{{{H}^{1}}\left( {{\mathbb{R}}^{n}} \right)}^{2} \right)}ds.
\end{split}
\end{align}
Integrating \eqref{d16} with respect to $\zeta$ on $(\tau-1,\tau)$ for all $\tau \in \mathbb{R}$ and $t\geq1$,
\begin{align*}
	\mathbb{E}\left( {{\left\| \nabla u\left( s \right) \right\|}^{2}} \right)\le {{c}_{3}}+\left( 1+{{c}_{3}} \right)\int_{\tau -1}^{\tau }{\mathbb{E}\left( \left\| u\left( s \right) \right\|_{{{H}^{1}}\left( {{\mathbb{R}}^{n}} \right)}^{2} \right)}ds,
\end{align*}
which together with Lemma 5.1 completes the proof.
\end{proof}
Next we will show the asymptotic compactness of solutions of system \eqref{c18}-\eqref{c19}. Since system \eqref{c18}-\eqref{c19}
 is a partly dissipative system and the initial condition $\xi _{0}$ is only in ${{L}^{2}}\left( \Omega ,{{L}^{2}}\left( {{\mathbb{R}}^{n}} \right) \right)$,
then for any $t\ge \tau$ , $v(t,\tau,\xi _{0})$ only belongs to ${{L}^{2}}\left( \Omega ,{{L}^{2}}\left( {{\mathbb{R}}^{n}} \right) \right)$, but not  ${{L}^{2}}\left( \Omega ,{{H}^{1}}\left( {{\mathbb{R}}^{n}} \right) \right)$. This makes it difficult for us to prove the asymptotic compactness of the solution directly by using the  Sobolev embedding theorem in bounded domains. In order to solve this difficulty, we decompose the solution. Let ${{v}}=v_{1}+v_{2},$ where $v_{1}$ and $v_{2}$ are the solutions of the following systems, respectively:
\begin{equation}\label{d17}
	\left\{ \begin{array}{l}
		 dv_{1}+\gamma v_{1}dt=\sum\limits_{k=1}^{\infty }{{{\delta }_{k}}v_{1}\left( t \right)d{{W}_{k}}\left( t \right)}, \\
		 v_{1}\left( \tau  \right)=\xi _{2},
	\end{array} \right.
\end{equation}
and
\begin{equation}\label{d18}
	\left\{ \begin{array}{l}
dv_{2}+\gamma v_{2}dt-\beta {{u}}dt={{G}_{2 }}\left( t,x \right)dt+ {\delta}\left( t,{{v}_{2}}\left( t \right) \right)dW\left( t \right), \\
 v_{2}\left( \tau  \right)=0.
	\end{array} \right.
\end{equation}
By \eqref{d17} and Gronwall inequality we can get for every $\tau\in \mathbb{R}$ and $D_{1}=\{D_{1}(t):t\in \mathbb{R}\}\in \mathcal{D}_0$,
\begin{align}\label{d19}
	\begin{split}
	\underset{t\to \infty }{\mathop{\lim }}\,\mathbb{E}\left( {{\left\| {{v}_{1}}\left( \tau  \right) \right\|}^{2}} \right)\le \underset{t\to \infty }{\mathop{\lim }}\,\mathbb{E}\left( \left\| {{\xi }_{0}} \right\|_{{{\mathbb{L}}^{2}}\left( {{\mathbb{R}}^{n}} \right)}^{2} \right){{e}^{-\left( 2\gamma -\left\| \delta  \right\|_{{{l}^{2}}}^{2} \right)t}}\le \underset{t\to \infty }{\mathop{\lim }}\,\left\| {{D}_{1}} \right\|_{{{\mathcal{P}}_{2}}\left( Z \right)}^{2}{{e}^{-\left( 2\gamma -\left\| \delta  \right\|_{{{l}^{2}}}^{2} \right)t}}=0,
	\end{split}
\end{align}
where ${{\xi }_{0}}=\left( \xi _{1},\xi _{2} \right)\in L_{{{\mathcal{F}}_{\tau -t}}}^{2}\left( \Omega ,{\mathbb{L}^{2}}\left( {{\mathbb{R}}^{n}} \right) \right)$ with $\mathcal{L}_{{{\xi }_{0}}}  \in D_{1}\left( \tau -t \right)$.
\begin{lem}\label{lem4.4}
Suppose assumption  $\mathbf{(H_{1})-({H_{4}})}$, \eqref{c21} and \eqref{c24} hold, then for every $\tau\in \mathbb{R}$ and ${{D}_{1}}=\left\{ {{D}_{1}}\left( t \right):t\in \mathbb{R} \right\}\in {{\mathcal{D}}_{0}}$, there exists $T=T(\tau,D_{1})>0$ such that for all $t\geq T$, the solution  $k=(u,v)$ of  \eqref{c18}-\eqref{c19} satisfies
	\begin{align*}
		 & \mathbb{E}\left( \left\| {{v}_{2}}\left( \tau ,\tau -t,0 \right)   \right\|_{{{H}^{1}}\left( {{\mathbb{R}}^{n}} \right)}^{2} \right) \\
		& \le {{M}_{4}}+{{M}_{4}}\int_{-\infty }^{\tau }{{{e}^{\eta \left( s-\tau  \right)}}\left( \left\| {{\phi }_{g}}\left( s \right) \right\|_{_{{{L}^{2}}\left( {{\mathbb{R}}^{n}} \right)}}^{2}+\left\| {{\theta }_{1}}\left( s \right) \right\|_{{{L}^{2}}\left( {{\mathbb{R}}^{n}},{{l}^{2}} \right)}^{2}+\left\| {{\theta }_{2}}\left( s \right) \right\|_{{{L}^{2}}\left( {{\mathbb{R}}^{n}},{{l}^{2}} \right)}^{2} \right)}ds,
	\end{align*}
where $\xi _{0}\in L_{{{\mathcal{F}}_{\tau -t}}}^{2}\left( \Omega ,{\mathbb{L}^{2}}\left( {{\mathbb{R}}^{n}} \right) \right)$ with ${{\mathcal{L}}_{{{\xi }_{0}}}}\in {{D}_{1}}\left( \tau -t \right)$,  $\eta>0$ is the same number as in  \eqref{c21} and $M_{4}$ is a positive constant independent of $\tau$ and $D_{1}$.
\end{lem}
\begin{proof}
By \eqref{d18} and Ito's formula, we get that for all $\tau \in \mathbb{R}$ and $t\geq0$,
\begin{align}\label{d20}
	\begin{split}
	 & {{\left\| \nabla {{v}_{2}}\left( \tau ,\tau -t,0 \right)   \right\|}^{2}}+2\gamma \int_{\tau -t}^{\tau }{{{\left\| \nabla v_2\left( s ,\tau -t,0 \right)  \right\|}^{2}}ds} \\
	  =&2\beta \int_{\tau -t}^{\tau }{\left( \nabla {{v}_{2}}\left( s,\tau -s,0 \right),\nabla u\left( s,\tau -t,{{\xi }_{0}} \right) \right)ds}  +2\int_{\tau -t}^{\tau }{\left( \nabla {{G}_{2}}\left( s,x \right),\nabla v_2\left( s ,\tau -t,0 \right) \right)ds}\\
	  &+ \int_{\tau -t}^{t}{{{\left\| \nabla \left( {\delta}\left( s,{{v}_{2}}\left( s \right) \right) \right) \right\|}^{2}}}ds+2 \int_{\tau -t}^{t}{\left( \nabla \left( {\delta}\left( s,{{v}_{2}}\left( s \right) \right) \right),\nabla {{v}_{2}}\left( s,\tau -t,0 \right) \right)}dW\left( s \right).
		\end{split}
\end{align}
For the first term on the right-hand side of \eqref{d20}, we can get
\begin{align}\label{d21}
	\begin{split}
	& 2\beta  \int_{\tau -t}^{\tau }{\left( \nabla {{v}_{2}}\left( \tau ,\tau -s,0 \right),\nabla u\left( s,\tau -t,{{\xi }_{1}} \right) \right)ds} \\
	 \le& \frac{\gamma }{2}\int_{\tau -t}^{\tau }{{{\left\| \nabla {{v}_{2}}\left( \tau ,\tau -s,0 \right) \right\|}^{2}}ds}+\frac{2}{\gamma }{{\beta }^{2}}\int_{\tau -t}^{\tau }{{{\left\| \nabla u\left( s,\tau -t,{{\xi }_{1}} \right) \right\|}^{2}}ds} .
		\end{split}
\end{align}
For the second term on the right-hand side of \eqref{d20},  we can get
\begin{align}\label{d22}
	\begin{split}
	& 2\int_{\tau -t}^{\tau }{\left( \nabla {{G}_{2}}\left( s,x \right),\nabla v_2\left( s ,\tau -t,0 \right) \right)ds} \\
	 \le& \frac{\gamma }{2}\int_{\tau -t}^{\tau }{{{\left\| \nabla {{v}_{2}}\left( \tau ,\tau -s,0 \right) \right\|}^{2}}ds}+\frac{2}{\gamma }\int_{\tau -t}^{\tau }{{{\left\| \nabla {{G}_{2}}\left( s \right) \right\|}^{2}}ds}.
	\end{split}
\end{align}
For the third term on the right-hand side of \eqref{d20}, by \eqref{c17} we can get
\begin{align}\label{d23}
	\begin{split}
	  &  \int_{\tau -t}^{t}{{{\left\| \nabla \left( {\delta}\left( s,{{v}_{2}}\left( s \right) \right) \right) \right\|}^{2}}}ds
	  \le 2\int_{\tau -t}^{\tau }{\left\| \nabla {{\theta }_{2}}\left( s \right) \right\|_{{{L}^{2}}\left( {{\mathbb{R}}^{n}},{{l}^{2}}
 \right)}^{2}}ds+2{\|\delta\|^2_{l^2}\int_{\tau -t}^{\tau}{{{\left\| \nabla {{v}_{2}}\left( s \right) \right\|}^{2}}}ds}.
	\end{split}
\end{align}
It follows from \eqref{d20}-\eqref{d23} we can get
\begin{align}\label{d24}
	\begin{split}
	 & \mathbb{E}\left( {{\left\| \nabla {{v}_{2}}\left( \tau ,\tau -t,0 \right)   \right\|}^{2}} \right)+\left( \gamma -2\|\delta\|^2_{l^2} \right)\int_{\tau -t}^{\tau }{\mathbb{E}\left( {{\left\| \nabla {{v}_{2}}\left( \tau ,\tau -t,0 \right)   \right\|}^{2}} \right)ds} \\
	 \le & \frac{2}{\gamma }{{\beta }^{2}}\int_{\tau -t}^{\tau }{\mathbb{E}\left( {{\left\| \nabla u\left( s,\tau -t,{{\xi }_{1}} \right) \right\|}^{2}} \right)ds}
+\frac{2}{\gamma }\int_{\tau -t}^{\tau }{{{\left\| \nabla {{G}_{2}}\left( s \right) \right\|}^{2}}ds}
+2\int_{\tau -t}^{\tau }{\left\| \nabla {{\theta }_{2}}\left( s \right) \right\|_{{{L}^{2}}\left( {{\mathbb{R}}^{n}},{{l}^{2}} \right)}^{2}}ds.
\end{split}
\end{align}
By Gronwall inequality, we can get
\begin{align*}
  & \mathbb{E}\left( {{\left\| \nabla {{v}_{2}}\left( \tau ,\tau -t,0 \right)   \right\|}^{2}} \right)\le \frac{2}{\gamma }{{\beta }^{2}}\int_{\tau -t}^{\tau }{{{e}^{\left( \gamma -2\|\delta\|^2_{l^2} \right)\left( s-\tau  \right)}}\mathbb{E}\left( {{\left\| \nabla u\left( s,\tau -t,{{\xi }_{0}} \right) \right\|}^{2}} \right)ds} \\
 & +\frac{2}{\gamma }\int_{-\infty}^{\tau }{{{e}^{\left( \gamma -2\|\delta\|^2_{l^2}\right)\left( s-\tau  \right)}}{{\left\| \nabla {{G}_{2}}\left( s \right) \right\|}^{2}}ds}+2\int_{-\infty }^{\tau }{{{e}^{\left( \gamma -2\|\delta\|^2_{l^2}\right)\left( s-\tau  \right)}}\left\| \nabla {{\theta }_{2}}\left( s \right) \right\|_{{{L}^{2}}\left( {{\mathbb{R}}^{n}},{{l}^{2}} \right)}^{2}}ds,
\end{align*}
which together with \eqref{c11} and Lemma 5.3 completes the proof.
\end{proof}
\begin{lem}\label{lem4.5}
	Suppose assumption  $\mathbf{(H_{1})-({H_{4}})}$, \eqref{c21} and \eqref{c24} hold, then for every $\tau\in \mathbb{R}$ and ${{D}_{1}}=\left\{ {{D}_{1}}\left( t \right):t\in \mathbb{R} \right\}\in {{\mathcal{D}}_{0}}$, there exists $T=T(\tau,D_{1})>0$ such that for all $t\geq T$, the solution  $k=(u,v)$ of  \eqref{c18}-\eqref{c19} satisfies
	\begin{align*}
		  \mathbb{E}\left( \int_{\left| x \right|\ge \sqrt{2}n}\left( {{{\left\vert u\left( \tau ,\tau -t,{{\xi }_{0}} \right)\left( x \right) \right\vert}^{2}}+{{\left\vert v\left( \tau ,\tau -t,{{\xi }_{0}} \right)\left( x \right) \right\vert}^{2}}}\right) dx \right)<\varepsilon   ,
	\end{align*}
	where $\xi _{0}\in L_{{{\mathcal{F}}_{\tau -t}}}^{2}\left( \Omega ,{\mathbb{L}^{2}}\left( {{\mathbb{R}}^{n}} \right) \right)$, with ${{\mathcal{L}}_{{{\xi }_{0}}}}\in {{D}_{1}}\left( \tau -t \right)$.
\end{lem}
\begin{proof}
Let $\rho $ : ${{\mathbb{R}}^{n}}\to [0,1]$ be a cut-off smooth function such that for any
$s\in {{\mathbb{R}}^{+}}$,
\begin{align}\label{d25}
	\begin{split}
\rho \left( s \right)=\left\{ \begin{array}{*{35}{l}}
	0, & 0\le s\le 1,  \\
	0<\rho \left( s \right)<1, & 1<s<2,  \\
	1, & s\ge 2.  \\
\end{array} \right.
\end{split}
\end{align}
Let ${{c}_{4}}=\underset{s\in {{\mathbb{R}}^{+}}}{\mathop{\sup }}\,\left\vert \rho '\left( s \right) \right\vert$, $n$ be a fixed integer and ${{\rho }_{n}}=\rho \left( \frac{{{\left\vert x \right\vert}^{2}}}{{{n}^{2}}} \right)$.
By \eqref{m3} we have for all $t\geq\tau$,
\begin{align}\label{d26}
	\begin{split}
	 & {{e}^{\eta t}}\left( \beta  {{\left\| {{\rho }_{n}}u\left( t \right) \right\|}^{2}}+\alpha {{\left\| {{\rho }_{n}}v\left( t \right) \right\|}^{2}} \right)+2\beta  \int_{\tau }^{t}{{{e}^{\eta s}}\left( \nabla \left( \rho _{n}^{2}u\left( s \right) \right),\nabla u\left( s \right) \right)}ds \\
	 & +2\beta \lambda \int_{\tau }^{t}{{{e}^{\eta s}}{{\left\| {{\rho }_{n}}u\left( s \right) \right\|}^{2}}}ds-\eta \int_{\tau }^{t}{{{e}^{\eta s}}\left( \beta  {{\left\| {{\rho }_{n}}u\left( t \right) \right\|}^{2}}+\alpha {{\left\| {{\rho }_{n}}v\left( t \right) \right\|}^{2}} \right)}ds \\
	 & +2\alpha \gamma  \int_{\tau }^{t}{{{e}^{\eta s}}{{\left\| {{\rho }_{n}}v\left( s \right) \right\|}^{2}}}ds+2\beta  \int_{\tau }^{t}{{{e}^{\eta s}}\left( f\left( s,x,u\left( s \right),{{\mathcal{L}}_{u\left( s \right)}} \right),\rho _{n}^{2}u\left( s \right) \right)}ds \\
	  =&{{e}^{\eta t}}\left( \beta {{\left\| {{\rho }_{n}}{{\xi }_{1}} \right\|}^{2}}+\alpha {{\left\| {{\rho }_{n}}{{\xi }_{2}} \right\|}^{2}} \right)+2\beta  \int_{\tau }^{t}{{{e}^{\eta s}}\left( {{G}_{1}}\left( s,x,u\left( s \right),{{\mathcal{L}}_{u\left( s \right)}} \right),\rho _{n}^{2}u\left( s \right) \right)}ds \\
	 & +2\alpha \int_{\tau }^{t}{{{e}^{\eta s}}\left( {{G}_{2}}\left( s,x \right),\rho _{n}^{2}v\left( s \right) \right)}ds+\alpha  \int_{\tau }^{t}{{{e}^{\eta s}}\left\| {{\rho }_{n}} {\delta}\left( s,v\left( s \right) \right)  \right\|_{{{L}_{2}}\left( {{l}^{2}},{{L}^{2}}\left( {{\mathbb{R}}^{n}} \right) \right)}^{2}}ds \\
	 & +\beta  \int_{\tau }^{t}{{{e}^{\eta s}}\left\| {{\rho }_{n}}{{\sigma }}\left( s,u\left( s \right),{{\mathcal{L}}_{u\left( s \right)}} \right) \right\|_{{{L}_{2}}\left( {{l}^{2}},{{L}^{2}}\left( {{\mathbb{R}}^{n}} \right) \right)}^{2}}ds+2\alpha  \int_{\tau }^{t}{{{e}^{\eta s}}\left( {\delta}\left( s,v\left( s \right) \right),\rho _{n}^{2}v\left( s \right) \right)}dW\left( s \right) \\
	 & +2\beta  \int_{\tau }^{t}{{{e}^{\eta s}}\left( {{\sigma }}\left( s,u\left( s \right),{{\mathcal{L}}_{u\left( s \right)}} \right),\rho _{n}^{2}u\left( s \right) \right)}dW\left( s \right),
\end{split}
\end{align}
$\mathbb{P}$-almost surely. Given $m\in \mathbb{N}$, denote by
\[{{\tau }_{m}}=\inf \{t\ge \tau :\left\| k\left( t \right) \right\|\ge m\}.\]
By \eqref{d26} we get for all $t\ge \tau $,
\begin{align}\label{d27}
	\begin{split}
   & \mathbb{E}\left( {{e}^{\eta \left( t\wedge {{\tau }_{m}} \right)}}\left( \beta  {{\left\| {{\rho }_{n}}u\left( {t\wedge {{\tau }_{m}}} \right) \right\|}^{2}}+\alpha {{\left\| {{\rho }_{n}}v\left( {t\wedge {{\tau }_{m}}} \right) \right\|}^{2}} \right) \right) +2\beta  \mathbb{E}\left( \int_{\tau }^{t\wedge {{\tau }_{m}}}{{{e}^{\eta s}}\left( \nabla \left( \rho _{n}^{2}u\left( s \right) \right),\nabla u\left( s \right) \right)} \right)ds \\
  \le&  \mathbb{E}\left( {{e}^{\eta \tau}}\left( \beta {{\left\| {{\rho }_{n}}{{\xi }_{1}} \right\|}^{2}}+\alpha {{\left\| {{\rho }_{n}}{{\xi }_{2}} \right\|}^{2}} \right) \right)+\left( \eta -2\lambda  \right)\int_{\tau }^{t}{{{e}^{\eta s}}\left( \beta {{\left\| {{\rho }_{n}}u\left( s \right) \right\|}^{2}}+\alpha {{\left\| {{\rho }_{n}}v\left( s \right) \right\|}^{2}} \right)}ds \\
  & -2\beta  \mathbb{E}\left( \int_{\tau }^{t\wedge {{\tau }_{m}}}{{{e}^{\eta s}}\left( f\left( s,u\left( s \right),{{\mathcal{L}}_{u\left( s \right)}} \right),\rho _{n}^{2}u\left( s \right) \right)}ds \right)  \\
  & +2\beta  \mathbb{E}\left( \int_{\tau }^{t\wedge {{\tau }_{m}}}{{{e}^{\eta s}}\left( {{G}_{1}}\left( s,u\left( s \right),{{\mathcal{L}}_{u\left( s \right)}} \right),\rho _{n}^{2}u\left( s \right) \right)}ds \right) \\
  & +2\alpha  \mathbb{E}\left( \int_{\tau }^{t\wedge {{\tau }_{m}}}{{{e}^{\eta s}}\left( {{G}_{2}}\left( s \right),\rho _{n}^{2}v\left( s \right) \right)}ds \right) +\alpha  \mathbb{E}\left( \int_{\tau }^{t\wedge {{\tau }_{m}}}{{{e}^{\eta s}}\left\| {{\rho }_{n}}{\delta}\left( s,v\left( s \right) \right) \right\|_{{{L}_{2}}\left( {{l}^{2}},{{L}^{2}}\left( {{\mathbb{R}}^{n}} \right) \right)}^{2}}ds \right) \\
  & +\beta  \mathbb{E}\left( \int_{\tau }^{t\wedge {{\tau }_{m}}}{{{e}^{\eta s}}\left\| {{\rho }_{n}}{{\sigma }}\left( s,u\left( s \right),{{\mathcal{L}}_{u\left( s \right)}} \right) \right\|_{{{L}_{2}}\left( {{l}^{2}},{{L}^{2}}\left( {{\mathbb{R}}^{n}} \right) \right)}^{2}}ds \right).
\end{split}
\end{align}
Note that
\begin{align}\label{d28}
	\begin{split}
& -2\beta \mathbb{E}\left( \int_{\tau }^{t\wedge {{\tau }_{m}}}{{{e}^{\eta s}}\left( \nabla \left( \rho _{n}^{2}u\left( s \right) \right),\nabla u\left( s \right) \right)ds} \right)
 =-2\beta \mathbb{E}\left( \int_{\tau }^{t\wedge {{\tau }_{m}}}{{{e}^{\eta s}}\int_{{{\mathbb{R}}^{n}}}{\rho _{n}^{2}{{\left\vert \nabla u\left( s \right) \right\vert}^{2}}}dxds} \right) \\
 & -2\beta \mathbb{E}\left( \int_{\tau }^{t\wedge {{\tau }_{m}}}{{{e}^{\eta s}}\int_{{{\mathbb{R}}^{n}}}{2{{n}^{-1}}u\left( s,x \right){{\rho }_{n}}\left( x \right)\nabla \rho \left( \frac{x}{n} \right)}\nabla u\left( s,x \right)dxds} \right) \\
 \le& 4{{n}^{-1}}{{\left\| \nabla \rho  \right\|}_{{{L}^{\infty }}\left( {{\mathbb{R}}^{n}} \right)}}\mathbb{E}\left( \int_{\tau }^{t\wedge {{\tau }_{m}}}{{{e}^{\eta s}}\left\| u\left( s \right) \right\|\left\| \nabla u\left( s \right) \right\|ds} \right) \\
 \le&2{{n}^{-1}}{{\left\| \nabla \rho  \right\|}_{{{L}^{\infty }}\left( {{\mathbb{R}}^{n}} \right)}}\mathbb{E}\left( \int_{\tau }^{t}{{{e}^{\eta s}}\left\| u\left( s \right) \right\|_{{{H}^{1}}\left( {{\mathbb{R}}^{n}} \right)}^{2}ds} \right).
\end{split}
\end{align}
For the third term on the left-hand side of \eqref{d27}, by \eqref{c2} we can get
\begin{align}\label{d29}
	\begin{split}
	& -2\beta \mathbb{E}\left( \int_{\tau }^{t\wedge {{\tau }_{m}}}{{{e}^{\eta s}}\left( f\left( s,x,u\left( s,x \right),{{\mathcal{L}}_{u\left( s \right)}} \right),\rho _{n}^{2}u\left( s \right) \right)}ds \right) \\
	 \le& -2\beta {{\alpha }_{1}}\mathbb{E}\left( \int_{\tau }^{t\wedge {{\tau }_{m}}}{{{e}^{\eta s}}\int_{{{\mathbb{R}}^{n}}}{\rho _{n}^{2}}{{\left\vert u\left( s \right) \right\vert}^{p}}dx}ds \right)+2\beta \mathbb{E}\left( \int_{\tau }^{t\wedge {{\tau }_{m}}}{{{e}^{\eta s}}\mathbb{E}\left( {{\left\| u\left( s \right) \right\|}^{2}} \right)\int_{{{\mathbb{R}}^{n}}}{\rho _{n}^{2}}{{\psi }_{1}}\left( x \right)dx}ds \right) \\
	& +2\beta \mathbb{E}\left( {{\int_{\tau }^{t\wedge {{\tau }_{m}}}{{{e}^{\eta s}}\int_{{{\mathbb{R}}^{n}}}{{{\left\| {{\phi }_{1}}\left( s \right) \right\|}_{{{L}^{\infty }}\left( {{\mathbb{R}}^{n}} \right)}}}\left\| {{\rho }_{n}}u\left( s \right) \right\|^{2}}}}ds \right)+2\beta \mathbb{E}\left( \int_{\tau }^{t\wedge {{\tau }_{m}}}{{{e}^{\eta s}}\int_{{{\mathbb{R}}^{n}}}{\rho _{n}^{2}\left( x \right)}\left\vert {{\phi }_{1}}\left( s,x \right) \right\vert dx}ds \right) \\
	 \le& 2\beta \int_{\tau }^{t}{{{e}^{\eta s}}\int_{{{\mathbb{R}}^{n}}}{{{\left\| {{\phi }_{1}}\left( s \right) \right\|}_{{{L}^{\infty }}\left( {{\mathbb{R}}^{n}} \right)}}}\mathbb{E}\left( {{\left\| {{\rho }_{n}}u\left( s \right) \right\|}^{2}} \right)}ds+2\beta \mathbb{E}\left( \int_{\tau }^{t}{{{e}^{\eta s}}\int_{{{\mathbb{R}}^{n}}}{\rho _{n}^{2}\left( x \right)}\left\vert {{\phi }_{1}}\left( s,x \right) \right\vert dx}ds \right) \\
	& +2\beta \mathbb{E}\left( \int_{{{\mathbb{R}}^{n}}}{\rho _{n}^{2}\left( x \right)}{{\psi }_{1}}\left( x \right)dx\int_{\tau }^{t}{{{e}^{\eta s}}\mathbb{E}\left( {{\left\| u\left( s \right) \right\|}^{2}} \right)}ds \right)-2\beta {{\alpha }_{1}}\mathbb{E}\left( \int_{\tau }^{t\wedge {{\tau }_{m}}}{{{e}^{\eta s}}\int_{{{\mathbb{R}}^{n}}}{\rho _{n}^{2}}{{\left\vert u\left( s \right) \right\vert }^{p}}dx}ds \right).
\end{split}
\end{align}
For the fourth  term on the right-hand side of \eqref{d27}, by \eqref{c7} and Young's inequality we can get
\begin{align}\label{d30}
	\begin{split}
	 & 2\beta \mathbb{E}\left( \int_{\tau }^{t\wedge {{\tau }_{m}}}{{{e}^{\eta s}}\left( {{G}_{1}}\left( s,x,u\left( s,x \right),{{\mathcal{L}}_{u\left( s \right)}} \right),\rho _{n}^{2}u\left( s \right) \right)}ds \right) \\
	  \le& 2\beta \mathbb{E}\left( \int_{\tau }^{t\wedge {{\tau }_{m}}}{{{e}^{\eta s}}\int_{{{\mathbb{R}}^{n}}}{\rho _{n}^{2}\left( x \right){{\phi }_{g}}\left( s,x \right)\left\vert u\left( s,x \right) \right\vert dxds}} \right) \\
	 & +2\beta \mathbb{E}\left( \int_{\tau }^{t\wedge {{\tau }_{m}}}{{{e}^{\eta s}}\int_{{{\mathbb{R}}^{n}}}{\rho _{n}^{2}\left( x \right)\left( {{\phi }_{7}}\left( s,x \right){{\left\vert u\left( s \right) \right\vert}^{2}}+{{\psi }_{g}}\left( x \right)\left\vert u\left( s \right) \right\vert\sqrt{\mathbb{E}\left( {{\left\| u\left( s \right) \right\|}^{2}} \right)} \right)dxds}} \right) \\
	  \le& \beta \mathbb{E}\left( \int_{\tau }^{t\wedge {{\tau }_{m}}}{{{e}^{\eta s}}\int_{{{\mathbb{R}}^{n}}}{\rho _{n}^{2}\left( x \right)\left( \eta {{\left\vert u\left( s,x \right) \right\vert}^{2}}+{{\eta }^{-1}}{{\left\vert {{\phi }_{g}}\left( s,x \right) \right\vert}^{2}} \right)dxds}} \right) \\
	 & +\beta \mathbb{E}\left( \int_{\tau }^{t\wedge {{\tau }_{m}}}{{{e}^{\eta s}}\int_{{{\mathbb{R}}^{n}}}{\rho _{n}^{2}\left( x \right)\left( \left( 2{{\phi }_{7}}\left( s,x \right)+{{\psi }_{g}}\left( x \right) \right){{\left\vert u\left( s \right) \right\vert}^{2}}+\left\vert {{\psi }_{g}}\left( x \right) \right\vert\mathbb{E}\left( {{\left\| u\left( s \right) \right\|}^{2}} \right) \right)dxds}} \right) \\
	  \le& \beta {{\eta }^{-1}}\int_{\tau }^{t}{{{e}^{\eta s}}{{\left\| {{\rho }_{n}}{{\phi }_{g}}\left( s \right) \right\|}^{2}}}ds+\beta {{\left\| \rho _{n}^{2}{{\psi }_{g}} \right\|}_{_{{{L}^{1}}\left( {{\mathbb{R}}^{n}} \right)}}}\int_{\tau }^{t}{{{e}^{\eta s}}}\mathbb{E}\left( {{\left\| u\left( s \right) \right\|}^{2}} \right)ds \\
	 & +\beta \int_{\tau }^{t}{{{e}^{\eta s}}\left( \eta +2{{\left\| {{\phi }_{7}}\left( s \right) \right\|}_{{{L}^{\infty }}\left( {{\mathbb{R}}^{n}} \right)}}+{{\left\| {{\psi }_{g}} \right\|}_{{{L}^{\infty }}\left( {{\mathbb{R}}^{n}} \right)}} \right)}\mathbb{E}\left( {{\left\| {{\rho }_{n}}u\left( s \right) \right\|}^{2}} \right)ds.
\end{split}
\end{align}
For the fifth  term on the right-hand side of \eqref{d27},  we can get
\begin{align}\label{d31}
	\begin{split}
		 & 2\alpha  \mathbb{E}\left( \int_{\tau }^{t\wedge {{\tau }_{m}}}{{{e}^{\eta s}}\left( {{G}_{2}}\left( s \right),\rho _{n}^{2}v\left( s \right) \right)}ds \right) \\
		  \le& \frac{1}{2}\alpha {{\eta }}\mathbb{E}\left( \int_{\tau }^{t}{{{e}^{\eta s}}{{\left\| {{\rho }_{n}}v\left( s \right) \right\|}^{2}}}ds \right)+2\alpha {{\eta }^{-1}}\mathbb{E}\left( \int_{\tau }^{t}{{{e}^{\eta s}}{{\left\| {{G}_{2}}\left( s \right) \right\|}^{2}}}ds \right).
	\end{split}
\end{align}
For the last two terms on the right-hand side of \eqref{d27}, by  \eqref{c15} and \eqref{c+} we can get
\begin{align}\label{d32}
	\begin{split}
	 & \beta \mathbb{E}\left( \int_{\tau }^{t\wedge {{\tau }_{m}}}{{{e}^{\eta s}}\left\| {{\rho }_{n}}{{\sigma }}\left( s,u\left( s \right),{{\mathcal{L}}_{u\left( s \right)}} \right) \right\|_{{{L}_{2}}\left( {{l}^{2}},{{L}^{2}}\left( {{\mathbb{R}}^{n}} \right) \right)}^{2}}ds \right) +\alpha  \mathbb{E}\left( \int_{\tau }^{t\wedge {{\tau }_{m}}}{{{e}^{\eta s}}\left\| {{\rho }_{n}}{\delta}\left( s,v\left( s \right) \right) \right\|_{{{L}_{2}}\left( {{l}^{2}},{{L}^{2}}\left( {{\mathbb{R}}^{n}} \right) \right)}^{2}}ds \right) \\ 	
	 \le &2\beta \mathbb{E}\left( \int_{\tau }^{t\wedge {{\tau }_{m}}}{{{e}^{\eta s}}\left\| {{\rho }_{n}}{{\theta }_{1}}\left( s \right) \right\|_{{{L}^{2}}\left( {{\mathbb{R}}^{n}},{{l}^{2}} \right)}^{2}}ds \right)+4\beta \left\| w \right\|_{{{L}^{\infty }}\left( {{\mathbb{R}}^{n}} \right)}^{2}\left\| {{\gamma }_{1}} \right\|_{{{l}^{2}}}^{2}\mathbb{E}\left( \int_{\tau }^{t\wedge {{\tau }_{m}}}{{{e}^{\eta s}}{{\left\| {{\rho }_{n}}u\left( s \right) \right\|}^{2}}}ds \right) \\
	& +8\beta {{\left\| {{\rho }_{n}}w \right\|}^{2}}\left\| {{\beta }_{1}} \right\|_{{{l}^{2}}}^{2}\mathbb{E}\left( \int_{\tau }^{t\wedge {{\tau }_{m}}}{{{e}^{\eta s}}\left( 1+\mathbb{E}\left( {{\left\| u\left( s \right) \right\|}^{2}} \right) \right)}ds \right) \\
	& +2\alpha \int_{\tau }^{t}{{{e}^{\eta s}}\left\| {{\rho }_{n}}{{\theta }_{2}}\left( s \right) \right\|}_{{{L}^{2}}\left( {{\mathbb{R}}^{n}},{{l}^{2}} \right)}^{2}ds+2\alpha \|\delta\|^2_{l^2}\int_{\tau }^{t}{{{e}^{\eta s}}{{\left\| {{\rho }_{n}}v\left( s \right) \right\|}^{2}}ds} \\
	 \le& 2\beta \int_{\tau }^{t}{{{e}^{\eta s}}\left\| {{\rho }_{n}}{{\theta }_{1}}\left( s \right) \right\|_{{{L}^{2}}\left( {{\mathbb{R}}^{n}},{{l}^{2}} \right)}^{2}}ds+4\beta \left\| w \right\|_{{{L}^{\infty }}\left( {{\mathbb{R}}^{n}} \right)}^{2}\left\| {{\gamma }_{1}} \right\|_{{{l}^{2}}}^{2}\int_{\tau }^{t\wedge {{\tau }_{m}}}{{{e}^{\eta s}}\mathbb{E}\left( {{\left\| {{\rho }_{n}}u\left( s \right) \right\|}^{2}} \right)}ds \\
	& +8\beta {{\left\| {{\rho }_{n}}w \right\|}^{2}}\left\| {{\beta }_{1}} \right\|_{{{l}^{2}}}^{2}{{\eta }^{-1}}{{e}^{\eta t}}+8\beta {{\left\| {{\rho }_{n}}w \right\|}^{2}}\left\| {{\beta }_{1}} \right\|_{{{l}^{2}}}^{2}\int_{\tau }^{t}{{{e}^{\eta s}}\mathbb{E}\left( {{\left\| u\left( s \right) \right\|}^{2}} \right)}ds \\
	& +2\alpha \int_{\tau }^{t}{{{e}^{\eta s}}\left\| {{\rho }_{n}}{{\theta }_{2}}\left( s \right) \right\|}_{{{L}^{2}}\left( {{\mathbb{R}}^{n}},{{l}^{2}} \right)}^{2}ds+2\alpha \|\delta\|^2_{l^2}\int_{\tau }^{t}{{{e}^{\eta s}}{{\left\| {{\rho }_{n}}v\left( s \right) \right\|}^{2}}ds}.
\end{split}
\end{align}
It follows from \eqref{d27}-\eqref{d32} that for all $t\ge \tau $,
\begin{align}\label{d33}
	\begin{split}
	& \mathbb{E}\left( {{e}^{\eta \left( t\wedge {{\tau }_{m}} \right)}}\left( \beta {{\left\| {{\rho }_{n}}u\left( t \right) \right\|}^{2}}+\alpha {{\left\| {{\rho }_{n}}v\left( t \right) \right\|}^{2}} \right) \right) \\
	 \le& \mathbb{E}\left( {{e}^{\eta \left( t\wedge {{\tau }_{m}} \right)}}\left( \beta  {{\left\| {{\rho }_{n}}{{\xi }_{1}} \right\|}^{2}}+\alpha {{\left\| {{\rho }_{n}}{{\xi }_{2}} \right\|}^{2}} \right) \right)+2\alpha {{\eta }^{-1}}\mathbb{E}\left( \int_{\tau }^{t}{{{e}^{\eta s}}{{\left\| {{G}_{2}}\left( s \right) \right\|}^{2}}}ds \right) \\
	& +2{{n}^{-1}}{{\left\| \nabla \rho  \right\|}_{{{L}^{\infty }}\left( {{\mathbb{R}}^{n}} \right)}}\mathbb{E}\left( \int_{\tau }^{t}{{{e}^{\eta s}}\left\| u\left( s \right) \right\|_{{{H}^{1}}\left( {{\mathbb{R}}^{n}} \right)}^{2}ds} \right)+2\beta \int_{\tau }^{t}{{{e}^{\eta s}}\left\| {{\rho }_{n}}{{\phi }_{1}}\left( s \right) \right\|_{{{L}^{1}}\left( {{\mathbb{R}}^{n}} \right)}^{2}}ds \\
	& +2\beta \int_{\tau }^{t}{{{e}^{\eta s}}\left\| {{\rho }_{n}}{{\theta }_{1}}\left( s \right) \right\|_{{{L}^{2}}\left( {{\mathbb{R}}^{n}},{{l}^{2}} \right)}^{2}}ds+2\alpha \int_{\tau }^{t}{{{e}^{\eta s}}\left\| {{\rho }_{n}}{{\theta }_{2}}\left( s \right) \right\|}_{{{L}^{2}}\left( {{\mathbb{R}}^{n}},{{l}^{2}} \right)}^{2}ds \\
	& +8\beta {{\left\| {{\rho }_{n}}w \right\|}^{2}}\left\| {{\beta }_{1}} \right\|_{{{l}^{2}}}^{2}{{\eta }^{-1}}{{e}^{\eta t}}+\beta {{\eta }^{-1}}\int_{\tau }^{t}{{{e}^{\eta s}}{{\left\| {{\rho }_{n}}{{\phi }_{g}}\left( s \right) \right\|}^{2}}}ds \\
	& +\int_{\tau }^{t}{{{e}^{\eta s}}\left( 2\eta -2\lambda +2{{\left\| {{\phi }_{7}}\left( s \right) \right\|}_{{{L}^{\infty }}\left( {{\mathbb{R}}^{n}} \right)}}+{{\left\| {{\psi }_{g}} \right\|}_{{{L}^{\infty }}\left( {{\mathbb{R}}^{n}} \right)}} \right)}\mathbb{E}\left( \beta {{\left\| {{\rho }_{n}}u\left( s \right) \right\|}^{2}}+\alpha {{\left\| {{\rho }_{n}}v\left( s \right) \right\|}^{2}} \right)ds \\
	& +2\int_{\tau }^{t}{{{e}^{\eta s}}\left( 2\left\| w \right\|_{{{L}^{\infty }}\left( {{\mathbb{R}}^{n}} \right)}^{2}\left\| {{\gamma }_{1}} \right\|_{{{l}^{2}}}^{2}+{{\left\| {{\phi }_{1}}\left( s \right) \right\|}_{{{L}^{\infty }}\left( {{\mathbb{R}}^{n}} \right)}}+\|\delta\|^2_{l^2} \right)}\mathbb{E}\left( \beta {{\left\| {{\rho }_{n}}u\left( s \right) \right\|}^{2}}+\alpha {{\left\| {{\rho }_{n}}v\left( s \right) \right\|}^{2}} \right)ds \\
	& +\left( 8{{\left\| {{\rho }_{n}}w \right\|}^{2}}\left\| {{\beta }_{1}} \right\|_{{{l}^{2}}}^{2}+2{{\left\| \rho _{n}^{2}{{\psi }_{1}} \right\|}_{_{{{L}^{1}}\left( {{\mathbb{R}}^{n}} \right)}}}+{{\left\| \rho _{n}^{2}{{\psi }_{g}} \right\|}_{_{{{L}^{1}}\left( {{\mathbb{R}}^{n}} \right)}}} \right)\int_{\tau }^{t}{{{e}^{\eta s}}\mathbb{E}\left( \beta {{\left\| u\left( s \right) \right\|}^{2}}+\alpha {{\left\| v\left( s \right) \right\|}^{2}} \right)}ds.
\end{split}
\end{align}
Taking the limit of \eqref{d33} as $m\to \infty $, then replacing $\tau $ and $t$ in \eqref{d33} by $\tau -t$ and $\tau $, respectively, by Fatou's lemma, \eqref{c21} and Lemma 5.1 we can get that there exist ${{c}_{5}}={{c}_{5}}\left( \tau  \right)>0$ and ${{T}_{1}}={{T}_{1}}\left( \tau ,D_{1} \right)\ge 1$ such that $t\ge {{T}_{1}}$,
\begin{align}\label{d35}
	\begin{split}
	& \mathbb{E}\left( \beta  {{\left\| {{\rho }_{n}}u\left( \tau ,\tau -t,{{\xi }_{1}} \right) \right\|}^{2}}+\alpha {{\left\| {{\rho }_{n}}v\left( \tau ,\tau -t,{{\xi }_{2}} \right) \right\|}^{2}} \right) \\
	 \le& {{e}^{-\eta t}} \mathbb{E}\left( \left( \beta  {{\left\| {{\rho }_{n}}{{\xi }_{1}} \right\|}^{2}}+\alpha {{\left\| {{\rho }_{n}}{{\xi }_{2}} \right\|}^{2}} \right) \right)+2\alpha {{\eta }^{-1}}\mathbb{E}\left( \int_{\tau }^{t}{{{e}^{\eta s}}{{\left\| {{G}_{2}}\left( s \right) \right\|}^{2}}}ds \right) \\
	& +2\beta \int_{-\infty }^{t}{{{e}^{\eta s}}\left\| {{\rho }_{n}}{{\phi }_{1}}\left( s \right) \right\|_{{{L}^{1}}\left( {{\mathbb{R}}^{n}} \right)}^{2}}ds+\beta {{\eta }^{-1}}\int_{-\infty }^{\tau }{{{e}^{\eta \left( s-\tau  \right)}}\left\| {{\rho }_{n}}{{\phi }_{g}}\left( s \right) \right\|_{{{L}^{2}}\left( {{\mathbb{R}}^{n}} \right)}^{2}}ds \\
	& +2\beta \int_{-\infty }^{\tau }{{{e}^{\eta \left( s-\tau  \right)}}\left\| {{\rho }_{n}}{{\theta }_{1}}\left( s \right) \right\|_{{{L}^{2}}\left( {{\mathbb{R}}^{n}},{{l}^{2}} \right)}^{2}}ds+2\alpha \int_{-\infty }^{\tau }{{{e}^{\eta \left( s-\tau  \right)}}\left\| {{\rho }_{n}}{{\theta }_{2}}\left( s \right) \right\|}_{{{L}^{2}}\left( {{\mathbb{R}}^{n}},{{l}^{2}} \right)}^{2}ds \\
	& +2{{n}^{-1}}{{\left\| \nabla \rho  \right\|}_{{{L}^{\infty }}\left( {{\mathbb{R}}^{n}} \right)}}{{c}_{5}}+8\beta {{\left\| {{\rho }_{n}}w \right\|}^{2}}\left\| {{\beta }_{1}} \right\|_{{{l}^{2}}}^{2}{{\eta }^{-1}} \\
	& +\left( 8{{\left\| {{\rho }_{n}}w \right\|}^{2}}\left\| {{\beta }_{1}} \right\|_{{{l}^{2}}}^{2}+2{{\left\| \rho _{n}^{2}{{\psi }_{1}} \right\|}_{_{{{L}^{1}}\left( {{\mathbb{R}}^{n}} \right)}}}+{{\left\| \rho _{n}^{2}{{\psi }_{g}} \right\|}_{_{{{L}^{1}}\left( {{\mathbb{R}}^{n}} \right)}}} \right){{c}_{5}}.
\end{split}
\end{align}
Note that $\mathcal{L}_{{{\xi }_{0}}} \in D_{1}\left( \tau -t \right)$ and $D_{1}\in {{\mathcal{D}}_{0}}$, then we have
\begin{align*}
	\underset{t\to \infty }{\mathop{\lim }}\,{{e}^{-\eta t}} \mathbb{E}\left( \left( \beta  {{\left\| {{\rho }_{n}}{{\xi }_{1}} \right\|}^{2}}+\alpha {{\left\| {{\rho }_{n}}{{\xi }_{2}} \right\|}^{2}} \right) \right)\le \underset{t\to \infty }{\mathop{\lim }}\,b{{e}^{-\eta t}}\left\| {{D}_{1}}\left( \tau -t \right) \right\|_{{{\mathcal{P}}_{2}}\left( {{L}^{2}}\left( {{\mathbb{R}}^{n}} \right) \right)}^{2}=0,
\end{align*}
and hence for every $\varepsilon >0$, there exists ${{T}_{2}}={{T}_{2}}\left( \varepsilon ,\tau ,D_{1} \right)\ge {{T}_{1}}$ such that for all $t\ge {{T}_{2}}$,
\begin{align}\label{d36}
	\begin{split}
{{e}^{-\eta t}} \mathbb{E}\left( \left( \beta  {{\left\| {{\rho }_{n}}{{\xi }_{1}} \right\|}^{2}}+\alpha {{\left\| {{\rho }_{n}}{{\xi }_{2}} \right\|}^{2}} \right) \right)<\frac{\varepsilon}{4}.
\end{split}
\end{align}
Since ${{G}_{2}}\in {{L}^{\infty }}\left( \mathbb{R},{{L}^{2}}\left( {{\mathbb{R}}^{n}} \right) \right)$ and ${{\phi }_{1}}\in {{L}^{\infty }}\left( \mathbb{R},{{L}^{1}}\left( {{\mathbb{R}}^{n}} \right) \right)$ we can get that there exists
${{N}_{1}}={{N}_{1}}\left( \varepsilon ,\tau  \right)\in \mathbb{N}$ such that for all $n\ge {{N}_{1}}$,
\begin{align}\label{d37}
	\begin{split}
	  & 2\alpha {{\eta }^{-1}}\mathbb{E}\left( \int_{\tau }^{t}{{{e}^{\eta s}}{{\left\| {{G}_{2}}\left( s \right) \right\|}^{2}}}ds \right)+2\beta \int_{-\infty }^{\tau}{{{e}^{\eta \left( s-\tau  \right)}}\left\| {{\rho }_{n}}{{\phi }_{1}}\left( s \right) \right\|_{{{L}^{1}}\left( {{\mathbb{R}}^{n}} \right)}^{2}}ds \\
	  \le & 2\alpha {{\eta }^{-1}}\int_{-\infty }^{\tau }{\int_{\left\vert x \right\vert\ge n}{{{e}^{\eta \left( s-\tau  \right)}}\left\vert {{G}_{2}}\left( s,x \right) \right\vert}dx}ds+2\beta \int_{-\infty }^{\tau }{\int_{\left\vert x \right\vert\ge n}{{{e}^{\eta \left( s-\tau  \right)}}{{\left\vert {{\rho }_{n}}{{\phi }_{1}}\left( s \right) \right\vert}^{2}}}dx}ds<\frac{\varepsilon }{4}.
\end{split}
\end{align}
By \eqref{c24} we can get there exists ${{N}_{2}}={{N}_{2}}\left( \varepsilon ,\tau  \right)\ge {{N}_{1}}$ such that for all $n\ge {{N}_{2}}$,
\begin{align}\label{d38}
	\begin{split}
	& 2\beta \int_{-\infty }^{\tau }{{{e}^{\eta \left( s-\tau  \right)}}\left\| {{\rho }_{n}}{{\theta }_{1}}\left( s \right) \right\|_{{{L}^{2}}\left( {{\mathbb{R}}^{n}},{{l}^{2}} \right)}^{2}}ds+2\alpha \int_{-\infty }^{\tau }{{{e}^{\eta \left( s-\tau  \right)}}\left\| {{\rho }_{n}}{{\theta }_{2}}\left( s \right) \right\|}_{{{L}^{2}}\left( {{\mathbb{R}}^{n}},{{l}^{2}} \right)}^{2}ds \\
	& +\beta {{\eta }^{-1}}\int_{-\infty }^{\tau }{{{e}^{\eta \left( s-\tau  \right)}}\left\| {{\rho }_{n}}{{\phi }_{g}}\left( s \right) \right\|_{{{L}^{2}}\left( {{\mathbb{R}}^{n}} \right)}^{2}}ds \\
	 \le &2\beta \int_{-\infty }^{\tau }{\int_{\left\vert x \right\vert\ge n}{{{e}^{\eta \left( s-\tau  \right)}}{{\left\vert {{\rho }_{n}}{{\theta }_{1}}\left( s \right) \right\vert}^{2}}}dx}ds+2\alpha \int_{-\infty }^{\tau }{\int_{\left\vert x \right\vert\ge n}{{{e}^{\eta \left( s-\tau  \right)}}{{\left\vert {{\rho }_{n}}{{\theta }_{2}}\left( s \right) \right\vert}^{2}}}dx}ds \\
	 & +\beta {{\eta }^{-1}}\int_{-\infty }^{\tau }{\int_{\left\vert x \right\vert\ge n}{{{e}^{\eta \left( s-\tau  \right)}}{{\left\vert {{\rho }_{n}}{{\phi }_{g}}\left( s \right) \right\vert}^{2}}}dx}ds<\frac{\varepsilon }{4}.
	\end{split}
\end{align}
For the last three term on the left-hand side of \eqref{d35},  we can get there exists ${{N}_{3}}={{N}_{3}}\left( \varepsilon ,\tau  \right)\ge {{N}_{2}}$ such that for all $n\ge {{N}_{3}}$,
\begin{align}\label{d39}
	\begin{split}
	 & 2{{n}^{-1}}{{\left\| \nabla \rho  \right\|}_{{{L}^{\infty }}\left( {{\mathbb{R}}^{n}} \right)}}{{c}_{5}}+8{{\left\| {{\rho }_{n}}w \right\|}^{2}}\left\| {{\beta }_{1}} \right\|_{{{l}^{2}}}^{2}{{\eta }^{-1}}  +\left( 8{{\left\| {{\rho }_{n}}w \right\|}^{2}}\left\| {{\beta }_{1}} \right\|_{{{l}^{2}}}^{2}+2{{\left\| \rho _{n}^{2}{{\psi }_{1}} \right\|}_{_{{{L}^{1}}\left( {{\mathbb{R}}^{n}} \right)}}}+{{\left\| \rho _{n}^{2}{{\psi }_{g}} \right\|}_{_{{{L}^{1}}\left( {{\mathbb{R}}^{n}} \right)}}} \right){{c}_{5}} \\
	 \le&  2{{n}^{-1}}{{\left\| \nabla \rho  \right\|}_{{{L}^{\infty }}\left( {{\mathbb{R}}^{n}} \right)}}{{c}_{5}}+8\left\| {{\beta }_{1}} \right\|_{{{l}^{2}}}^{2}{{\eta }^{-1}}\int_{\left\vert x \right\vert\ge n}{{{w}^{2}}\left( x \right)}dx \\
	 & +\left( 8\left\| {{\beta }_{1}} \right\|_{{{l}^{2}}}^{2}\int_{\left\vert x \right\vert\ge n}{{{w}^{2}}\left( x \right)}dx+2\int_{\left\vert x \right\vert\ge n}{\left\vert {{\psi }_{1}}\left( x \right) \right\vert}dx+\int_{\left\vert x \right\vert\ge n}{\left\vert {{\psi }_{g}}\left( x \right) \right\vert}dx \right){{c}_{5}}<\frac{\varepsilon }{4}.
	\end{split}
\end{align}
It follows from \eqref{d35}-\eqref{d39} that for all  $t\ge {{T}_{2}}$ and $n\ge {{N}_{3}}$,
\begin{align*}
	 & \mathbb{E}\left( \int_{\left\vert x \right\vert\ge \sqrt{2}n}{{{\left\vert u\left( \tau ,\tau -t,{{\xi }_{0}} \right)\left( x \right) \right\vert}^{2}}+{{\left\vert v\left( \tau ,\tau -t,{{\xi }_{0}} \right)\left( x \right) \right\vert}^{2}}}dx \right)
	 \le  \mathbb{E}\left( \beta {{\left\| {{\rho }_{n}}u\left( t \right) \right\|}^{2}}+\alpha {{\left\| {{\rho }_{n}}v\left( t \right) \right\|}^{2}} \right)<\varepsilon,
\end{align*}
as desired.
\end{proof}
\begin{lem}\label{lem4.6}
	Suppose assumption  $\mathbf{(H_{1})-({H_{4}})}$, \eqref{c21} and \eqref{c24} hold, then for every $\tau\in \mathbb{R}$ and ${{D}_{1}}=\left\{ {{D}_{1}}\left( t \right):t\in \mathbb{R} \right\}\in {{\mathcal{D}}_{0}}$, there exists $T=T(\tau,D_{1})>0$ such that for all $t\geq T$, the solution  $k=(u,v)$ of  \eqref{c18}-\eqref{c19} satisfies
	\begin{align*}
		& \mathbb{E}\left( \left\| \left( u\left( \tau ,\tau -s,{{\xi }_{0}} \right),v\left( \tau ,\tau -s,{{\xi }_{0}} \right) \right) \right\|_{{\mathbb{L}^{2}}\left( {{\mathbb{R}}^{n}} \right)}^{4} \right) \\
		 \le& {{M}_{5}}+{{M}_{5}}\int_{-\infty }^{\tau }{{{e}^{\eta \left( s-\tau  \right)}}\left( \left\| {{\phi }_{g}}\left( s \right) \right\|_{_{{{L}^{2}}\left( {{\mathbb{R}}^{n}} \right)}}^{4}+\left\| {{\theta }_{1}}\left( s \right) \right\|_{{{L}^{2}}\left( {{\mathbb{R}}^{n}},{{l}^{2}} \right)}^{4}+\left\| {{\theta }_{2}}\left( s \right) \right\|_{{{L}^{2}}\left( {{\mathbb{R}}^{n}},{{l}^{2}} \right)}^{4} \right)}ds ,
	\end{align*}
	where $\xi _{0}\in L_{{{\mathcal{F}}_{\tau -t}}}^{2}\left( \Omega ,{\mathbb{L}^{2}}\left( {{\mathbb{R}}^{n}} \right) \right)$ with ${{\mathcal{L}}_{{{\xi }_{0}}}}\in {{D}_{1}}\left( \tau -t \right)$,  $\eta>0$ is the same number as in  \eqref{c21} and $M_{5}$ is a positive constant independent of $\tau$ and $D_{2}$.
\end{lem}
\begin{proof}
By \eqref{m3} and Ito's formula, we can get  for all $t\ge \tau $,
\begin{align}\label{d40}
	\begin{split}
	& {{e}^{2\eta t}}\left( \beta  {{\left\| u\left( t \right) \right\|}^{4}}+\alpha {{\left\| v\left( t \right) \right\|}^{4}} \right)+4\beta  \int_{\tau }^{t}{{{e}^{2\eta s}}{{\left\| u\left( s \right) \right\|}^{2}}{{\left\| \nabla u\left( s \right) \right\|}^{2}}}ds \\
	& -2\eta \int_{\tau }^{t}{{{e}^{2\eta s}}\left( \beta  {{\left\| u\left( s \right) \right\|}^{4}}+\alpha {{\left\| v\left( s \right) \right\|}^{4}} \right)}ds+4\beta \lambda  \int_{\tau }^{t}{{{e}^{2\eta s}}{{\left\| u\left( s \right) \right\|}^{4}}}ds \\
	& +4\beta  \int_{\tau }^{t}{{{e}^{2\eta s}}{{\left\| u\left( s \right) \right\|}^{2}}\left( f\left( s,x,u\left( s \right),{{\mathcal{L}}_{u\left( s \right)}} \right),u\left( s \right) \right)}ds+4\alpha \gamma  \int_{\tau }^{t}{{{e}^{\eta s}}{{\left\| v\left( s \right) \right\|}^{4}}}ds \\
	 =&{{e}^{2\eta \tau }}\left( \beta  {{\left\| {{\xi }_{1}} \right\|}^{4}}+\alpha {{\left\| {{\xi }_{2}} \right\|}^{4}} \right)+4\alpha  \int_{\tau }^{t}{{{e}^{\eta s}}{{\left\| v\left( s \right) \right\|}^{2}}\left( {{G}_{2}}\left( s \right),v\left( s \right) \right)}ds \\
	& +4\beta\int_{\tau }^{t}{{{e}^{2\eta s}}{{\left\| u\left( s \right) \right\|}^{2}}\left( {{G}_{1}}\left( s,x,u\left( s \right),{{\mathcal{L}}_{u\left( s \right)}} \right),u\left( s \right) \right)}ds \\
	& +2\beta  \int_{\tau }^{t}{{{e}^{2\eta s}}{{\left\| u\left( s \right) \right\|}^{2}}\left\| {{\sigma }}\left( s,u,{{\mathcal{L}}_{u}} \right) \right\|_{{{L}_{2}}\left( {{l}^{2}},{{L}^{2}}\left( {{\mathbb{R}}^{n}} \right) \right)}^{2}}ds+2\alpha  \int_{\tau }^{t}{{{e}^{2\eta s}}{{\left\| v\left( s \right) \right\|}^{2}}\left\| {\delta}\left( s,v\left( s \right) \right) \right\|_{{{L}_{2}}\left( {{l}^{2}},{{L}^{2}}\left( {{\mathbb{R}}^{n}} \right) \right)}^{2}}ds \\
	& +4\beta  \int_{\tau }^{t}{{{e}^{2\eta s}}\left( {{\left\| u\left( s \right) \right\|}^{2}},\left\| {{\sigma }}\left( s,u,{{\mathcal{L}}_{u}} \right) \right\|_{{{L}_{2}}\left( {{l}^{2}},{{L}^{2}}\left( {{\mathbb{R}}^{n}} \right) \right)}^{2} \right)}ds+4\alpha  \int_{\tau }^{t}{{{e}^{2\eta s}}\left( {{\left\| v\left( s \right) \right\|}^{2}},\left\| {\delta}\left( s,v \right) \right\|_{{{L}_{2}}\left( {{l}^{2}},{{L}^{2}}\left( {{\mathbb{R}}^{n}} \right) \right)}^{2} \right)}ds \\
	& +4\beta  \int_{\tau }^{t}{{{e}^{2\eta s}}{{\left\| u\left( s \right) \right\|}^{2}}\left( {{\sigma }}\left( s,u,{{\mathcal{L}}_{u}} \right),u \right)}dW\left( s \right)+4\alpha \int_{\tau }^{t}{{{e}^{2\eta s}}{{\left\| v\left( s \right) \right\|}^{2}}\left( {\delta}\left( s,v\left( s \right) \right),v \right)}dW\left( s \right).
\end{split}
\end{align}
$\mathbb{P}$-almost surely. Given $m\in \mathbb{N}$, denote by
\[{{\tau }_{m}}=\inf \{t\ge \tau :\left\| k\left( t \right) \right\|\ge m\}.\]
Consequently we can get that for all  $t\ge \tau $,
\begin{align}\label{d41}
	\begin{split}
	  & \mathbb{E}\left( {{e}^{2\eta \left( t\wedge {{\tau }_{m}} \right)}}\left( \beta  {{\left\| u\left( t\wedge {{\tau }_{m}} \right) \right\|}^{4}}+\alpha {{\left\| v\left( t\wedge {{\tau }_{m}} \right) \right\|}^{4}} \right) \right)+4\beta  \mathbb{E}\left( \int_{\tau }^{t\wedge {{\tau }_{m}}}{{{e}^{2\eta s}}{{\left\| u\left( s \right) \right\|}^{2}}{{\left\| \nabla u\left( s \right) \right\|}^{2}}}ds \right) \\
	 & +2\left( 2\lambda -\eta  \right)\mathbb{E}\left( \int_{\tau }^{t\wedge {{\tau }_{m}}}{{{e}^{2\eta s}}\left( \beta  {{\left\| u\left( s \right) \right\|}^{4}}+\alpha {{\left\| v\left( s \right) \right\|}^{4}} \right)}ds \right) \\
	 \le & {{e}^{2\eta \tau }}\mathbb{E}\left(  \beta  {{\left\| {{\xi }_{1}} \right\|}^{4}}+\alpha {{\left\| {{\xi }_{2}} \right\|}^{4}}  \right)-4\beta \mathbb{E}\left( \int_{\tau }^{t\wedge {{\tau }_{m}}}{{{e}^{2\eta s}}{{\left\| u\left( s \right) \right\|}^{2}}\left( f\left( s,x,u\left( s,x \right),{{\mathcal{L}}_{u\left( s \right)}} \right),u\left( s \right) \right)}ds \right) \\
	 & +4\beta  \mathbb{E}\left( \int_{\tau }^{t\wedge {{\tau }_{m}}}{{{e}^{2\eta s}}{{\left\| u\left( s \right) \right\|}^{2}}\left( {{G}_{1}}\left( s,x,u\left( s,x \right),{{\mathcal{L}}_{u\left( s,\tau -t,{{\xi }_{1}} \right)}} \right),u\left( s \right) \right)}ds \right) \\
	 & +4\alpha  \mathbb{E}\left( \int_{\tau }^{t\wedge {{\tau }_{m}}}{{{e}^{\eta s}}{{\left\| v\left( s \right) \right\|}^{2}}\left( {{G}_{2}}\left( s,x \right),v\left( s \right) \right)}ds \right) \\
	 & +6\alpha \mathbb{E}\left( \int_{\tau }^{t\wedge {{\tau }_{m}}}{{{e}^{2\eta s}}{{\left\| v\left( s \right) \right\|}^{2}}\left\| {\delta}\left( s,v\left( s \right) \right) \right\|_{{{L}_{2}}\left( {{l}^{2}},{{L}^{2}}\left( {{\mathbb{R}}^{n}} \right) \right)}^{2}}ds \right)\\
	 & +6\beta  \mathbb{E}\left( \int_{\tau }^{t\wedge {{\tau }_{m}}}{{{e}^{2\eta s}}{{\left\| u\left( s \right) \right\|}^{2}}\left\| {{\sigma }}\left( s,u\left( s,\tau -t,{{\xi }_{1}} \right),{{\mathcal{L}}_{u\left( s,\tau -t,{{\xi }_{1}} \right)}} \right) \right\|_{{{L}_{2}}\left( {{l}^{2}},{{L}^{2}}\left( {{\mathbb{R}}^{n}} \right) \right)}^{2}}ds \right) .
\end{split}
\end{align}
For the second term on the left-hand side of \eqref{d41}, we can get
\begin{align}\label{d42}
	\begin{split}
& -4\beta \mathbb{E}\left( \int_{\tau }^{t\wedge {{\tau }_{m}}}{{{e}^{2\eta s}}{{\left\| u\left( s \right) \right\|}^{2}}\left( f\left( s,x,u\left( s,x \right),{{\mathcal{L}}_{u\left( s \right)}} \right),u\left( s \right) \right)}ds \right) \\
 \le & 4\beta \mathbb{E}\left( \int_{\tau }^{t\wedge {{\tau }_{m}}}{{{e}^{2\eta s}}{{\left\| {{\phi }_{1}}\left( s \right) \right\|}_{{{L}^{\infty }}\left( {{\mathbb{R}}^{n}} \right)}}{{\left\| u\left( s \right) \right\|}^{4}}}ds \right) \\
& +4\beta \mathbb{E}\left( \int_{\tau }^{t\wedge {{\tau }_{m}}}{{{e}^{2\eta s}}{{\left\| u\left( s \right) \right\|}^{2}}\left( {{\left\| {{\phi }_{1}}\left( s \right) \right\|}_{{{L}^{1}}\left( {{\mathbb{R}}^{n}} \right)}}+{{\left\| {{\psi }_{1}} \right\|}_{{{L}^{1}}\left( {{\mathbb{R}}^{n}} \right)}}\mathbb{E}\left( {{\left\| u\left( s \right) \right\|}^{2}} \right) \right)}ds \right) \\
 \le&  4\beta \left( \int_{\tau }^{t}{{{e}^{2\eta s}}{{\left\| {{\phi }_{1}}\left( s \right) \right\|}_{{{L}^{\infty }}\left( {{\mathbb{R}}^{n}} \right)}}\mathbb{E}\left( {{\left\| u\left( s \right) \right\|}^{4}} \right)}ds \right)+4\beta \mathbb{E}\left( \int_{\tau }^{t}{{{e}^{2\eta s}}{{\left\| u\left( s \right) \right\|}^{2}}{{\left\| {{\phi }_{1}}\left( s \right) \right\|}_{{{L}^{1}}\left( {{\mathbb{R}}^{n}} \right)}}}ds \right) \\
& +4\beta \left( \int_{\tau }^{t\wedge {{\tau }_{m}}}{{{e}^{2\eta s}}{{\left\| {{\psi }_{1}} \right\|}_{{{L}^{1}}\left( {{\mathbb{R}}^{n}} \right)}}{{\left( \mathbb{E}\left( {{\left\| u\left( s \right) \right\|}^{2}} \right) \right)}^{2}}}ds \right) \\
\le&  \int_{\tau }^{t}{{{e}^{2\eta s}}\left( 4{{\left\| {{\phi }_{1}}\left( s \right) \right\|}_{{{L}^{\infty }}\left( {{\mathbb{R}}^{n}} \right)}}+4{{\left\| {{\psi }_{1}} \right\|}_{{{L}^{1}}\left( {{\mathbb{R}}^{n}} \right)}}+\frac{1}{3}\eta  \right)}\mathbb{E}\left( \beta {{\left\| u\left( s \right) \right\|}^{4}} \right)ds \\
& +12{{\eta }^{-1}}\int_{\tau }^{t}{{{e}^{2\eta s}}\left\| {{\phi }_{1}}\left( s \right) \right\|_{{{L}^{1}}\left( {{\mathbb{R}}^{n}} \right)}^{2}}ds
 .
\end{split}
\end{align}
For the third  term on the right-hand side of \eqref{d41}, we can get
\begin{align}\label{d43}
	\begin{split}
	& 4\beta  \mathbb{E}\left( \int_{\tau }^{t\wedge {{\tau }_{m}}}{{{e}^{2\eta s}}{{\left\| u\left( s \right) \right\|}^{2}}\left( {{G}_{1}}\left( s,x,u\left( s,x \right),{{\mathcal{L}}_{u\left( s,\tau -t,{{\xi }_{1}} \right)}} \right),u\left( s \right) \right)}ds \right) \\
	 \le	& 2\beta \mathbb{E}\left( \int_{\tau }^{t\wedge {{\tau }_{m}}}{{{e}^{2\eta s}}\left( \eta {{\left\| u\left( s \right) \right\|}^{4}}+{{\eta }^{-1}}{{\left\| {{\phi }_{g}}\left( s \right) \right\|}^{2}}{{\left\| u\left( s \right) \right\|}^{2}} \right)}ds \right) \\
	& +2\beta \mathbb{E}\left( \int_{\tau }^{t\wedge {{\tau }_{m}}}{{{e}^{2\eta s}}{{\left\| u\left( s \right) \right\|}^{2}}{{\left\| {{\psi }_{g}} \right\|}_{{{L}^{1}}\left( {{\mathbb{R}}^{n}} \right)}}\mathbb{E}\left( {{\left\| u\left( s \right) \right\|}^{2}} \right)}ds \right) \\
	& +2\beta \mathbb{E}\left( \int_{\tau }^{t\wedge {{\tau }_{m}}}{{{e}^{2\eta s}}\left( 2{{\left\| {{\phi }_{7}}\left( s \right) \right\|}_{{{L}^{\infty }}\left( {{\mathbb{R}}^{n}} \right)}}+{{\left\| {{\psi }_{g}} \right\|}_{{{L}^{\infty }}\left( {{\mathbb{R}}^{n}} \right)}} \right){{\left\| u\left( s \right) \right\|}^{4}}}ds \right) \\
	 \le&  2\beta \mathbb{E}\left( \int_{\tau }^{t}{{{e}^{2\eta s}}\left( \frac{4}{3}\eta {{\left\| u\left( s \right) \right\|}^{4}}+\frac{3}{4}{{\eta }^{-3}}{{\left\| {{\phi }_{g}}\left( s \right) \right\|}^{4}} \right)}ds \right) \\
	& +2\beta \mathbb{E}\left( \int_{\tau }^{t}{{{e}^{2\eta s}}{{\left\| u\left( s \right) \right\|}^{2}}{{\left\| {{\psi }_{g}} \right\|}_{{{L}^{1}}\left( {{\mathbb{R}}^{n}} \right)}}\mathbb{E}\left( {{\left\| u\left( s \right) \right\|}^{2}} \right)}ds \right) \\
	& +2\beta \mathbb{E}\left( \int_{\tau }^{t}{{{e}^{2\eta s}}\left( 2{{\left\| {{\phi }_{7}}\left( s \right) \right\|}_{{{L}^{\infty }}\left( {{\mathbb{R}}^{n}} \right)}}+{{\left\| {{\psi }_{g}} \right\|}_{{{L}^{\infty }}\left( {{\mathbb{R}}^{n}} \right)}} \right){{\left\| u\left( s \right) \right\|}^{4}}}ds \right) \\
	\le&  \frac{3}{2}{{\eta }^{-3}}\left( \int_{\tau }^{t}{{{e}^{2\eta s}}{{\left\| {{\phi }_{g}}\left( s \right) \right\|}^{4}}}ds \right) \\
	& +2\int_{\tau }^{t}{{{e}^{2\eta s}}\left( 2{{\left\| {{\phi }_{7}}\left( s \right) \right\|}_{{{L}^{\infty }}\left( {{\mathbb{R}}^{n}} \right)}}+{{\left\| {{\psi }_{g}} \right\|}_{{{L}^{\infty }}\left( {{\mathbb{R}}^{n}} \right)}}+{{\left\| {{\psi }_{g}} \right\|}_{{{L}^{1}}\left( {{\mathbb{R}}^{n}} \right)}}+\frac{4}{3}\eta  \right)\mathbb{E}\left( \beta {{\left\| u\left( s \right) \right\|}^{4}} \right)}ds .
\end{split}
\end{align}
For the fourth and fifth term on the left-hand side of \eqref{d41}, we can get
\begin{align}\label{d44}
	\begin{split}
  &  4\alpha  \mathbb{E}\left( \int_{\tau }^{t\wedge {{\tau }_{m}}}{{{e}^{\eta s}}{{\left\| v\left( s \right) \right\|}^{2}}\left( {{G}_{2}}\left( s,x \right),v\left( s \right) \right)}ds \right) \\
  & +6\alpha \mathbb{E}\left( \int_{\tau }^{t\wedge {{\tau }_{m}}}{{{e}^{2\eta s}}{{\left\| v\left( s \right) \right\|}^{2}}\left\| {\delta}\left( s,v\left( s \right) \right) \right\|_{{{L}_{2}}\left( {{l}^{2}},{{L}^{2}}\left( {{\mathbb{R}}^{n}} \right) \right)}^{2}}ds \right) \\
  \le& \frac{2\alpha }{\eta }\mathbb{E}\left( \int_{\tau }^{t\wedge {{\tau }_{m}}}{{{e}^{\eta s}}{{\left\| v\left( s \right) \right\|}^{2}}{{\left\| {{G}_{2}}\left( s \right) \right\|}^{2}}}ds \right)+2\alpha \eta \mathbb{E}\left( \int_{\tau }^{t\wedge {{\tau }_{m}}}{{{e}^{\eta s}}{{\left\| v\left( s \right) \right\|}^{4}}}ds \right) \\
 & +12\alpha \mathbb{E}\left( \int_{\tau }^{t\wedge {{\tau }_{m}}}{{{e}^{2\eta s}}{{\left\| v\left( s \right) \right\|}^{2}}\left\| {{\theta }_{2}}\left( s \right) \right\|_{{{L}^{2}}\left( {{\mathbb{R}}^{n}}, {{l}^{2}} \right)}^{2}ds} \right)+12\alpha \left\| \delta  \right\|_{{{l}^{2}}}^{2}\mathbb{E}\left( \int_{\tau }^{t}{{{e}^{2\eta s}}{{\left\| v\left( s \right) \right\|}^{4}}ds} \right) \\
  \le &\frac{3}{2}{{\eta }^{-3}}\mathbb{E}\left( \int_{\tau }^{t\wedge {{\tau }_{m}}}{{{e}^{\eta s}}{{\left\| {{G}_{2}}\left( s \right) \right\|}^{4}}}ds \right)+3\alpha \eta \mathbb{E}\left( \int_{\tau }^{t\wedge {{\tau }_{m}}}{{{e}^{\eta s}}{{\left\| v\left( s \right) \right\|}^{4}}}ds \right) \\
 & +12\alpha \left\| \delta  \right\|_{{{l}^{2}}}^{2}\mathbb{E}\left( \int_{\tau }^{t}{{{e}^{2\eta s}}{{\left\| v\left( s \right) \right\|}^{4}}ds} \right)+108\alpha {{\eta }^{-1}}\int_{\tau }^{t}{{{e}^{2\eta s}}\left\| {{\theta }_{2}}\left( s \right) \right\|_{{{L}^{2}}\left( {{\mathbb{R}}^{n}},{{l}^{2}} \right)}^{4}ds} .
\end{split}
\end{align}
For the last term on the left-hand side of \eqref{d41}, we can get
\begin{align}\label{d45}
	\begin{split}
	& 6\beta  \mathbb{E}\left( \int_{\tau }^{t\wedge {{\tau }_{m}}}{{{e}^{2\eta s}}{{\left\| u\left( s \right) \right\|}^{2}}\left\| {{\sigma }}\left( s,u\left( s,\tau -t,{{\xi }_{1}} \right),{{\mathcal{L}}_{u\left( s,\tau -t,{{\xi }_{1}} \right)}} \right) \right\|_{{{L}_{2}}\left( {{l}^{2}},{{L}^{2}}\left( {{\mathbb{R}}^{n}} \right) \right)}^{2}}ds \right) \\
	 \le& 12\beta \mathbb{E}\left( \int_{\tau }^{t\wedge {{\tau }_{m}}}{{{e}^{2\eta s}}{{\left\| u\left( s \right) \right\|}^{2}}\left\| {{\theta }_{1}}\left( s \right) \right\|_{{{L}^{2}}\left( {{\mathbb{R}}^{n}},{{l}^{2}} \right)}^{2}ds} \right) \\
	& +48{{\left\| w \right\|}^{2}}\left\| {{\beta }_{1}} \right\|_{{{l}^{2}}}^{2}\mathbb{E}\left( \int_{\tau }^{t\wedge {{\tau }_{m}}}{{{e}^{2\eta s}}{{\left\| u\left( s \right) \right\|}^{2}}\left( 1+\mathbb{E}\left( {{\left\| u\left( s \right) \right\|}^{2}} \right) \right)ds} \right) \\
	& +24\left\| w \right\|_{{{L}^{\infty }}\left( \mathbb{R}^{n} \right)}^{2}\left\| {{\gamma }_{1}} \right\|_{{{l}^{2}}}^{2}\mathbb{E}\left( \int_{\tau }^{t\wedge {{\tau }_{m}}}{{{e}^{2\eta s}}{{\left\| u\left( s \right) \right\|}^{4}}ds} \right) \\
	 \le& \frac{1}{3}\beta \eta \int_{\tau }^{t}{{{e}^{2\eta s}}\mathbb{E}\left( {{\left\| u\left( s \right) \right\|}^{4}} \right)ds}+108\beta {{\eta }^{-1}}\int_{\tau }^{t}{{{e}^{2\eta s}}\left\| {{\theta }_{1}}\left( s \right) \right\|_{{{L}^{2}}\left( {{\mathbb{R}}^{n}},{{l}^{2}} \right)}^{4}ds} \\
	& +48\beta {{\left\| w \right\|}^{2}}\left\| {{\beta }_{1}} \right\|_{{{l}^{2}}}^{2}\int_{\tau }^{t}{{{e}^{2\eta s}}\mathbb{E}\left( {{\left\| u\left( s \right) \right\|}^{4}} \right)ds}+48\beta {{\left\| w \right\|}^{2}}\left\| {{\beta }_{1}} \right\|_{{{l}^{2}}}^{2}\int_{\tau }^{t}{{{e}^{2\eta s}}\mathbb{E}\left( {{\left\| u\left( s \right) \right\|}^{2}} \right)ds} \\
	& +24\beta \left\| w \right\|_{{{L}^{\infty }}\left( \mathbb{R}^{n} \right)}^{2}\left\| {{\gamma }_{1}} \right\|_{{{l}^{2}}}^{2}\int_{\tau }^{t\wedge {{\tau }_{m}}}{{{e}^{2\eta s}}\mathbb{E}\left( {{\left\| u\left( s \right) \right\|}^{4}} \right)ds} \\
	 \le& \left( \frac{1}{2}\eta +48{{\left\| w \right\|}^{2}}\left\| {{\beta }_{1}} \right\|_{{{l}^{2}}}^{2}+24\left\| w \right\|_{{{L}^{\infty }}\left( \mathbb{R}^{n} \right)}^{2}\left\| {{\gamma }_{1}} \right\|_{{{l}^{2}}}^{2} \right)\int_{\tau }^{t}{{{e}^{2\eta s}}\mathbb{E}\left( \beta {{\left\| u\left( s \right) \right\|}^{4}} \right)ds} \\
	& +108{{\eta }^{-1}}\int_{\tau }^{t}{{{e}^{2\eta s}}\left( \beta \left\| {{\theta }_{1}}\left( s \right) \right\|_{{{L}^{2}}\left( {{\mathbb{R}}^{n}},{{l}^{2}} \right)}^{4} \right)ds}+1728{{\eta }^{-2}}{{\left\| w \right\|}^{4}}\left\| {{\beta }_{1}} \right\|_{{{l}^{2}}}^{4}{{e}^{2\eta t}}.
\end{split}
\end{align}
It follows from \eqref{d41}-\eqref{d45} that for all $t\ge \tau $,
\begin{align}\label{d46}
	\begin{split}
	& \mathbb{E}\left( {{e}^{2\eta \left( t\wedge {{\tau }_{m}} \right)}}\left( \beta  {{\left\| u\left( t\wedge {{\tau }_{m}} \right) \right\|}^{4}}+\alpha {{\left\| v\left( t\wedge {{\tau }_{m}} \right) \right\|}^{4}} \right) \right) \\
	& +\left( 4\lambda -5\eta  \right)\mathbb{E}\left( \int_{\tau }^{t\wedge {{\tau }_{m}}}{{{e}^{2\eta s}}\left( \beta  {{\left\| u\left( s \right) \right\|}^{4}}+\alpha {{\left\| v\left( s \right) \right\|}^{4}} \right)}ds \right) \\
	\le&  {{e}^{2\eta \tau }}\mathbb{E}\left(  \beta  {{\left\| {{\xi }_{1}} \right\|}^{4}}+\alpha {{\left\| {{\xi }_{2}} \right\|}^{4}}  \right)+12{{\eta }^{-1}}\int_{\tau }^{t}{{{e}^{2\eta s}}\left\| {{\phi }_{1}}\left( s \right) \right\|_{{{L}^{1}}\left( {{\mathbb{R}}^{n}} \right)}^{2}}ds+\frac{3}{2}{{\eta }^{-3}}\int_{\tau }^{t}{{{e}^{\eta s}}{{\left\| {{G}_{2}}\left( s \right) \right\|}^{4}}}ds \\
	& +\frac{3}{2}{{\eta }^{-3}}\left( \int_{\tau }^{t}{{{e}^{2\eta s}}{{\left\| {{\phi }_{g}}\left( s \right) \right\|}^{4}}}ds \right)+1728{{\eta }^{-2}}{{\left\| w \right\|}^{4}}\left\| {{\beta }_{1}} \right\|_{{{l}^{2}}}^{4}{{e}^{2\eta t}} \\
	& +108{{\eta }^{-1}}\int_{\tau }^{t}{{{e}^{2\eta s}}\left( \beta \left\| {{\theta }_{1}}\left( s \right) \right\|_{{{L}^{2}}\left( {{\mathbb{R}}^{n}},{{l}^{2}} \right)}^{4}+\alpha \left\| {{\theta }_{2}}\left( s \right) \right\|_{{{L}^{2}}\left( {{\mathbb{R}}^{n}},{{l}^{2}} \right)}^{4} \right)ds} \\
	& +4\int_{\tau }^{t}{{{e}^{2\eta s}}\left( {{\left\| {{\phi }_{1}}\left( s \right) \right\|}_{{{L}^{\infty }}\left( {{\mathbb{R}}^{n}} \right)}}+{{\left\| {{\psi }_{1}} \right\|}_{{{L}^{1}}\left( {{\mathbb{R}}^{n}} \right)}}+3\|\delta\|^2_{l^2} \right)}\mathbb{E}\left( \beta {{\left\| u\left( s \right) \right\|}^{4}}+\alpha {{\left\| v\left( s \right) \right\|}^{4}} \right)ds \\
	& +2\int_{\tau }^{t}{{{e}^{2\eta s}}\left( 2{{\left\| {{\phi }_{7}}\left( s \right) \right\|}_{{{L}^{\infty }}\left( {{\mathbb{R}}^{n}} \right)}}+{{\left\| {{\psi }_{g}} \right\|}_{{{L}^{\infty }}\left( {{\mathbb{R}}^{n}} \right)}}+{{\left\| {{\psi }_{g}} \right\|}_{{{L}^{1}}\left( {{\mathbb{R}}^{n}} \right)}} \right)\mathbb{E}\left( \beta {{\left\| u\left( s \right) \right\|}^{4}}+\alpha {{\left\| v\left( s \right) \right\|}^{4}} \right)}ds \\
	& +\left( \frac{1}{2}\eta +48{{\left\| w \right\|}^{2}}\left\| {{\beta }_{1}} \right\|_{{{l}^{2}}}^{2}+24\left\| w \right\|_{{{L}^{\infty }}\left( \mathbb{R}^{n} \right)}^{2}\left\| {{\gamma }_{1}} \right\|_{{{l}^{2}}}^{2} \right)\int_{\tau }^{t}{{{e}^{2\eta s}}\mathbb{E}\left( \beta {{\left\| u\left( s \right) \right\|}^{4}}+\alpha {{\left\| v\left( s \right) \right\|}^{4}} \right)ds}.
\end{split}
\end{align}
Taking the limit of \eqref{d46} as $m\to \infty $, by Fatou's lemma and \eqref{c22} we get that for all $t\ge \tau $,
\begin{align}
	\label{d47}
	\begin{split}
 & \mathbb{E}\left( {{e}^{2\eta t}}\left( \beta  {{\left\| u\left( t \right) \right\|}^{4}}+\alpha {{\left\| v\left( t \right) \right\|}^{4}} \right) \right) \\
 \le& {{e}^{2\eta \tau }}\mathbb{E}\left(  \beta  {{\left\| {{\xi }_{1}} \right\|}^{4}}+\alpha {{\left\| {{\xi }_{2}} \right\|}^{4}}  \right)+12{{\eta }^{-1}}\int_{\tau }^{t}{{{e}^{2\eta s}}\left\| {{\phi }_{1}}\left( s \right) \right\|_{{{L}^{1}}\left( {{\mathbb{R}}^{n}} \right)}^{2}}ds+\frac{3}{2}{{\eta }^{-3}}\int_{\tau }^{t}{{{e}^{\eta s}}{{\left\| {{G}_{2}}\left( s \right) \right\|}^{4}}}ds \\
& +\frac{3}{2}{{\eta }^{-3}}\left( \int_{\tau }^{t}{{{e}^{2\eta s}}{{\left\| {{\phi }_{g}}\left( s \right) \right\|}^{4}}}ds \right)+1728{{\eta }^{-2}}{{\left\| w \right\|}^{4}}\left\| {{\beta }_{1}} \right\|_{{{l}^{2}}}^{4}{{e}^{2\eta t}} \\
& +108{{\eta }^{-1}}\int_{\tau }^{t}{{{e}^{2\eta s}}\left( \beta \left\| {{\theta }_{1}}\left( s \right) \right\|_{{{L}^{2}}\left( {{\mathbb{R}}^{n}},{{l}^{2}} \right)}^{4}+\alpha \left\| {{\theta }_{2}}\left( s \right) \right\|_{{{L}^{2}}\left( {{\mathbb{R}}^{n}},{{l}^{2}} \right)}^{4} \right)ds}
.
\end{split}
\end{align}
Replacing $\tau $ and $t$ in \eqref{d47} by $\tau -t$ and $\tau $, respectively, we have for all $t\ge \tau $,
\begin{align*}
	& \mathbb{E}\left( \beta  {{\left\| u\left( t \right) \right\|}^{4}}+\alpha {{\left\| v\left( t \right) \right\|}^{4}} \right) \\
	 \le& {{e}^{-2\eta t}}\mathbb{E}\left(  \beta  {{\left\| {{\xi }_{1}} \right\|}^{4}}+\alpha {{\left\| {{\xi }_{2}} \right\|}^{4}}  \right)+12{{\eta }^{-1}}\int_{\tau -t}^{\tau }{{{e}^{2\eta (s-\tau)}}\left\| {{\phi }_{1}}\left( s \right) \right\|_{{{L}^{1}}\left( {{\mathbb{R}}^{n}} \right)}^{2}}ds+\frac{3}{2}{{\eta }^{-3}}\int_{\tau -t}^{\tau }{{{e}^{\eta (s-\tau)}}{{\left\| {{G}_{2}}\left( s \right) \right\|}^{4}}}ds \\
	& +\frac{3}{2}{{\eta }^{-3}}\int_{\tau -t}^{\tau }{{{e}^{2\eta (s-\tau)}}{{\left\| {{\phi }_{g}}\left( s \right) \right\|}^{4}}}ds+1728{{\eta }^{-2}}{{\left\| w \right\|}^{4}}\left\| {{\beta }_{1}} \right\|_{{{l}^{2}}}^{4} \\
	& +108{{\eta }^{-1}}\int_{\tau -t}^{\tau }{{{e}^{2\eta (s-\tau)}}\left( \beta \left\| {{\theta }_{1}}\left( s \right) \right\|_{{{L}^{2}}\left( {{\mathbb{R}}^{n}},{{l}^{2}} \right)}^{4}+\alpha \left\| {{\theta }_{2}}\left( s \right) \right\|_{{{L}^{2}}\left( {{\mathbb{R}}^{n}},{{l}^{2}} \right)}^{4} \right)ds}
	.
\end{align*}
Note that $\mathcal{L}_{{{\xi }_{0}}}\in D_{2}\left( \tau -t \right)$ and $D_{2}\in {{\mathcal{D}}}$, then we have
\begin{align*}
\underset{t\to \infty }{\mathop{\lim }}\,{{e}^{-2\eta t}}\mathbb{E}\left(  \beta  {{\left\| {{\xi }_{1}} \right\|}^{4}}+\alpha {{\left\| {{\xi }_{2}} \right\|}^{4}}  \right)\le {{e}^{-2\eta \tau }}\underset{t\to \infty }{\mathop{\lim }}\,{{e}^{2\eta \left( \tau -t \right)}}\left\| {{D}_{2}}\left( \tau -t \right) \right\|_{{{\mathcal{P}}_{4}}\left( {\mathbb{L}^{2}}\left( {{\mathbb{R}}^{n}} \right) \right)}^{4}=0,
\end{align*}
and hence there exists $T=T\left( \tau ,D_{2} \right)>0$ such that for all $t\ge T$,
\[{{e}^{-2\eta t}}\mathbb{E}\left(  \beta  {{\left\| {{\xi }_{1}} \right\|}^{4}}+\alpha {{\left\| {{\xi }_{2}} \right\|}^{4}}  \right)\le 1.\]
Further we can get that for all $t\ge T$,
\begin{align*}
	& \mathbb{E}\left( \beta {{\left\| u\left( t \right) \right\|}^{4}}+\alpha {{\left\| v\left( t \right) \right\|}^{4}} \right) \\
	 \le& 1+6{{\eta }^{-1}}\left\| {{\phi }_{1}}\left( s \right) \right\|_{{{L}^{\infty }}\left( \mathbb{R},{{L}^{1}}\left( {{\mathbb{R}}^{n}} \right) \right)}^{2}+\frac{3}{2}{{\eta }^{-3}}\left\| {{\phi }_{1}}\left( s \right) \right\|_{{{L}^{\infty }}\left( \mathbb{R},{{L}^{2}}\left( {{\mathbb{R}}^{n}} \right) \right)}^{2}+1728{{\eta }^{-2}}{{\left\| w \right\|}^{4}}\left\| {{\beta }_{1}} \right\|_{{{l}^{2}}}^{4} \\
	& +\frac{3}{2}{{\eta }^{-3}}\int_{\tau -t}^{\tau }{{{e}^{2\eta s}}{{\left\| {{\phi }_{g}}\left( s \right) \right\|}^{4}}}ds+108{{\eta }^{-1}}\int_{\tau -t}^{\tau }{{{e}^{2\eta s}}\left( \beta \left\| {{\theta }_{1}}\left( s \right) \right\|_{{{L}^{2}}\left( {{\mathbb{R}}^{n}},{{l}^{2}} \right)}^{4}+\alpha \left\| {{\theta }_{2}}\left( s \right) \right\|_{{{L}^{2}}\left( {{\mathbb{R}}^{n}},{{l}^{2}} \right)}^{4} \right)ds},
\end{align*}
which completes the proof.

\end{proof}

\section{Existence of Pullback Measure Attractors  }
In this section, we  stablish the existence and uniqueness of  $\mathcal{D}$-pullback measure attractors of \eqref{c18}-\eqref{c19} in $\mathcal{P}_{4}({{\mathbb{L}}^{2}}\left( \mathbb{R}^{n} \right))$. For this purpose, we first define a non-autonomous dynamical system in $\mathcal{P}_{4}({{\mathbb{L}}^{2}}\left( \mathbb{R}^{n} \right))$.
Given $\tau ,t\in \mathbb{R}$ with $t\ge \tau$, define $P_{\tau ,t}^{*}:{{\mathcal{P}}_{4}}\left( {{\mathbb{L}}^{2}}\left( {{\mathbb{R}}^{n}} \right) \right)\to {{\mathcal{P}}_{4}}\left( {{\mathbb{L}}^{2}}\left( {{\mathbb{R}}^{n}} \right) \right)$, then for every $\mu \in {{\mathcal{P}}_{4}}\left( {{\mathbb{L}}^{2}}\left( {{\mathbb{R}}^{n}} \right) \right)$,
\begin{align}\label{f1}
	\begin{split}
		P_{\tau ,t}^{*}\mu ={{\mathcal{L}}_{k\left( t,\tau ,{{\xi }_{0}} \right)}},
	\end{split}
\end{align}
where ${{k}}\left( t,\tau ,{{\xi }_{0}} \right)$  is the solution of (3.19) with $\xi _{0}\in L_{{{\mathcal{F}}_{\tau }}}^{4}\left( \Omega ,{{\mathbb{L}}^{2}}\left( \mathbb{R}^{n} \right) \right)$ such that ${{\mathcal{L}}_{{{\xi }_{0}}}}=\mu $. Moreover, for every $t\in {{\mathbb{R}}^{+}}$  and $\tau \in \mathbb{R}$, define ${{S}}\left( t,\tau  \right):{{\mathcal{P}}_{4}}({{\mathbb{L}}^{2}}\left( \mathbb{R}^{n} \right))\to {{\mathcal{P}}_{4}}({{\mathbb{L}}^{2}}\left( \mathbb{R}^{n}\right))$ by, for all $\mu\in\mathcal{P}_{4}({{\mathbb{L}}^{2}}\left( \mathbb{R}^{n} \right))$,
\begin{align}\label{f2}
	\begin{split}
		{{S}}\left( t,\tau  \right)\mu =P_{\tau ,\tau +t}^{*}\mu.
	\end{split}
\end{align}
According to the uniqueness of solutions of \eqref{c18}-\eqref{c19}, we can get that for all $t,s\in {{\mathbb{R}}^{+}}$, $\tau \in \mathbb{R}$
and $\xi _{0}\in L_{{{\mathcal{F}}_{\tau }}}^{4}\left( \Omega ,{{\mathbb{L}}^{2}}\left( \mathbb{R}^{n} \right) \right)$,
\[
{{k}}\left( t+s+\tau ,\tau ,{{\xi }_{0}} \right)={{k}}\left( t+s+\tau ,s+\tau ,{{k}}\left( s+\tau ,\tau ,{{\xi }_{0}} \right) \right).
\]
 Further we can get that for all $t,s\in {{\mathbb{R}}^{+}}$, $\tau \in \mathbb{R}$ and $\mu\in\mathcal{P}_{4}({{\mathbb{L}}^{2}}\left( \mathbb{R}^{n} \right))$,
\begin{align}\label{f3}
	\begin{split}
		{{S}}\left( t+s,\tau  \right)\mu ={{S}}\left( t,s+\tau  \right)\left( {{S}}\left( s,\tau  \right)\mu  \right).
	\end{split}
\end{align}
Next we will prove ${{S}}$ is weakly continuous over bounded subsets of $\mathcal{P}_{4}({{\mathbb{L}}^{2}}\left( \mathbb{R}^{n} \right))$, which is devoted to obtain ${{S}}$ is weakly continuous over bounded subsets of $\left( {{\mathcal{P}}_{4}}({{\mathbb{L}}^{2}}\left( \mathbb{R}^{n} \right)),{{d}_{\mathcal{P}({{\mathbb{L}}^{2}}\left( \mathbb{R}^{n} \right))}} \right)$.
\begin{lem}\label{lem6.1}
	Suppose  $\mathbf{(H_{1})-({H_{4}})}$  hold. Let $\xi _{0},\xi _{0,n}\in L_{{{\mathcal{F}}_{\tau }}}^{4}\left( \Omega ,{{\mathbb{L}}^{2}}\left( \mathbb{R}^{n} \right) \right) $ such that $\mathbb{E}\left( \left\| {{\xi }_{0}} \right\|_{{{\mathbb{L}}^{2}}\left( {{\mathbb{R}}^{n}} \right)}^{4} \right)\le K$
	and $\mathbb{E}\left( \left\| {{\xi }_{0,n}} \right\|_{{{\mathbb{L}}^{2}}\left( {{\mathbb{R}}^{n}} \right)}^{4} \right)\le K$ for some $K>0$. If ${{\mathcal{L}}_{{{\xi }_{0}}}}\to {{\mathcal{L}}_{\xi _{0,n}}}$ weakly, then for every $\tau \in \mathbb{R}$ and $t\ge \tau $, ${{\mathcal{L}}_{k\left( t,\tau ,{{\xi }_{0}} \right)}}\to {{\mathcal{L}}_{k\left( t,\tau ,\xi _{0,n} \right)}}$ weakly.
\end{lem}
\begin{proof}
	  Since ${{\mathcal{L}}_{{{\xi }_{0}}}}\to {{\mathcal{L}}_{\xi _{0,n}}}$ weakly, by the Skorokhov theorem,  there exist a probability space $\left( \tilde{\Omega },\tilde{\mathcal{F}},\tilde{\mathbb{P}} \right)$  and random variables ${{\tilde{\xi }_{0}}}$ and $\tilde{\xi }_{0,n}$ defined in $\left( \tilde{\Omega },\tilde{\mathcal{F}},\tilde{\mathbb{P}} \right)$ such that the distributions of
${{\tilde{\xi }_{0}}}$ and $\tilde{\xi }_{0,n}$ coincide with that of ${{{\xi }_{0}}}$ and ${\xi }_{0,n}^{\varepsilon }$, respectively. Furthermore, $\tilde{\xi }_{0,n}\to {{\tilde{\xi }_{0}}}$ \  $\tilde{\mathbb{P}}$ -almost surely. Note that ${{\tilde{\xi }_{0}}}$ , $\tilde{\xi }_{0,n}$ and W can be considered as random variables defined in the product space $\left( \Omega \times \tilde{\Omega },\mathcal{F}\times \tilde{\mathcal{F}},\mathbb{P}\times \tilde{\mathbb{P}} \right)$. So we may consider the solutions of the stochastic equation in the product space with initial data ${{{\xi }_{0}}}$ and ${\xi }_{0,n}$, instead of the solutions in $\left( \Omega ,\mathcal{F},\mathbb{P} \right)$ with initial data ${{\tilde{\xi }_{0}}}$ and $\tilde{\xi }_{0,n}$. However, for simplicity, we will not distinguish the new random variables from the original ones, and just consider the solutions of the equation in the original space. Since $\tilde{\xi }_{0,n}\to {{\tilde{\xi }_{0}}}$ $\left( \mathbb{P}\times \tilde{\mathbb{P}} \right)$ -almost surely, without loss of generality, we simply assume that $\xi _{0,n}\to {{\xi }_{0}}$ $\mathbb{P}$ -almost surely. Let $y_{n}={{u}}\left( t,\tau ,\xi _{1} \right)-u_{n}\left( t,\tau ,\xi _{1,n} \right)$ and $z_{n}={{v}}\left( t,\tau ,\xi _{2} \right)-v_{n}\left( t,\tau ,\xi _{2,n} \right)$.

For the first equation of \eqref{c18} we can get that for all $t\geq\tau$,
\begin{align*}
& d{{y}_{n}}\left( t \right)-\nabla {{y}_{n}}\left( t \right)+\lambda {{y}_{n}}\left( t \right)dt+\alpha {{z}_{n}}\left( t \right)dt+\left( f\left( t,u\left( t \right),{{\mathcal{L}}_{u\left( t \right)}} \right)-f\left( t,{{u}_{n}}\left( t \right),{{\mathcal{L}}_{{{u}_{n}}\left( t \right)}} \right) \right)dt \\
 =&\left( {{G}_{1}}\left( t,u\left( t \right),{{\mathcal{L}}_{u\left( t \right)}} \right)-{{G}_{1}}\left( t,{{u}_{n}}\left( t \right),{{\mathcal{L}}_{{{u}_{n}}\left( t \right)}} \right) \right)dt  +\left( {{\sigma }}\left( t,u\left( t \right),{{\mathcal{L}}_{u\left( t \right)}} \right)-{{\sigma }}\left( t,{{u}_{n}}\left( t \right),{{\mathcal{L}}_{{{u}_{n}}\left( t \right)}} \right) \right)dW\left( t \right).
\end{align*}
For the second equation of \eqref{c18} we can get that for all $t\geq\tau$,
\[d{{z}_{n}}\left( t \right)+\gamma {{z}_{n}}\left( t \right)dt-\beta {{z}_{n}}dt=\sum\limits_{k=1}^{\infty }{{\delta_{k}}{{z}_{n}}\left( t \right)}d{{W}_{k}}\left( t \right).\]
Then from above equations  and by Ito's formula we can get that for all  $t\ge \tau $,
\begin{align}\label{f4}
	\begin{split}
& \left( \beta {{\left\| {{y}_{n}}\left( t \right) \right\|}^{2}}+\alpha {{\left\| {{z}_{n}}\left( t \right) \right\|}^{2}} \right)+2\beta \int_{\tau }^{t}{{{\left\| \nabla {{y}_{n}}\left( s \right) \right\|}^{2}}}ds\\
& +2\beta \lambda \int_{\tau }^{t}{{{\left\| {{y}_{n}}\left( s \right) \right\|}^{2}}}ds+2\alpha \gamma \int_{\tau }^{t}{{{\left\| {{z}_{n}}\left( s \right) \right\|}^{2}}}ds \\
=& \beta {{\left\| {{\xi }_{1}}-{{\xi }_{1,n}} \right\|}^{2}}+\alpha {{\left\| {{\xi }_{2}}-{{\xi }_{2,n}} \right\|}^{2}}-2\beta \int_{\tau }^{t}{\left( f\left( s,u\left( s \right),{{\mathcal{L}}_{u\left( s \right)}} \right)-f\left( s,{{u}_{n}}\left( s \right),{{\mathcal{L}}_{{{u}_{n}}\left( s \right)}} \right) \right)}ds \\
& +2\beta \int_{\tau }^{t}{\left( {{G}_{1}}\left( s,u\left( s \right),{{\mathcal{L}}_{u\left( s \right)}} \right)-{{G}_{1}}\left( s,{{u}_{n}}\left( s \right),{{\mathcal{L}}_{{{u}_{n}}\left( s \right)}} \right),{{y}_{n}}\left( s \right) \right)}ds \\
& +\beta \int_{\tau }^{t}{\left\| {{\sigma }}\left( s,u\left( s \right),{{\mathcal{L}}_{u\left( s \right)}} \right)-{{\sigma }}\left( s,{{u}_{n}}\left( s \right),{{\mathcal{L}}_{{{u}_{n}}\left( s \right)}} \right) \right\|}_{{{L}_{2}}\left( {{l}^{2}},{{L}^{2}}\left( {{\mathbb{R}}^{n}} \right) \right)}^{2}ds \\
& +\alpha \sum\limits_{k=1}^{\infty }{\int_{\tau }^{t}{{{\left\| {\delta_{k}}{{z}_{n}}\left( s \right) \right\|}^{2}}}}ds+2\alpha \sum\limits_{k=1}^{\infty }{\int_{\tau }^{t}{\left( {{\delta }_{k}}{{z}_{n}}\left( s \right),{{z}_{n}}\left( s \right) \right)}dW\left( s \right)}\\
& +2\beta \int_{\tau }^{t}{\left( {{\sigma }}\left( s,u\left( s \right),{{\mathcal{L}}_{u\left( s \right)}} \right)-{{\sigma }}\left( s,{{u}_{n}}\left( s \right),{{\mathcal{L}}_{{{u}_{n}}\left( s \right)}} \right),{{y}_{n}}\left( s \right) \right)}dW\left( s \right).
	\end{split}
\end{align}
For the second term on the right-hand, by \eqref{c3} and \eqref{c4} we can get
\begin{align}\label{f5}
	\begin{split}
	& -2\beta \int_{\tau }^{t}{\left( f\left( s,u\left( s \right),{{\mathcal{L}}_{u\left( s \right)}} \right)-f\left( s,{{u}_{n}}\left( s \right),{{\mathcal{L}}_{{{u}_{n}}\left( s \right)}} \right),{{y}_{n}}\left( s \right) \right)}ds \\
	 =&-2\beta \int_{\tau }^{t}{\int_{{{\mathbb{R}}^{n}}}{\left( f\left( s,u\left( s \right),{{\mathcal{L}}_{u\left( s \right)}} \right)-f\left( s,u\left( s \right),{{\mathcal{L}}_{{{u}_{n}}\left( s \right)}} \right) \right)}{{y}_{n}}\left( s \right)}ds \\
	& -2\beta \int_{\tau }^{t}{\int_{{{\mathbb{R}}^{n}}}{\left( f\left( s,u\left( s \right),{{\mathcal{L}}_{{{u}_{n}}\left( s \right)}} \right)-f\left( s,{{u}_{n}}\left( s \right),{{\mathcal{L}}_{{{u}_{n}}\left( s \right)}} \right) \right)}{{y}_{n}}\left( s \right)ds} \\
	 \le& \beta \int_{\tau }^{t}{\left( \left( 2{{\left\| {{\phi }_{4}}\left( s \right) \right\|}_{{{L}^{\infty }}\left( {{\mathbb{R}}^{n}} \right)}}+{{\left\| {{\phi }_{3}}\left( s \right) \right\|}_{{{L}^{\infty }}\left( {{\mathbb{R}}^{n}} \right)}} \right){{\left\| {{y}_{n}}\left( s \right) \right\|}^{2}}+\mathbb{E}\left( {{\left\| {{y}_{n}}\left( s \right) \right\|}^{2}} \right){{\left\| {{\phi }_{3}}\left( s \right) \right\|}_{{{L}^{1}}\left( {{\mathbb{R}}^{n}} \right)}} \right)ds} .
	 \end{split}
 \end{align}
For the third term on the right-hand, similarly by \eqref{c8}  we can get
\begin{align}\label{f6}
	\begin{split}
		& 2\beta \int_{\tau }^{t}{\left( {{G}_{1}}\left( s,u\left( s \right),{{\mathcal{L}}_{u\left( s \right)}} \right)-{{G}_{1}}\left( s,{{u}_{n}}\left( s \right),{{\mathcal{L}}_{{{u}_{n}}\left( s \right)}} \right),{{y}_{n}}\left( s \right) \right)}ds \\
		 \le &\beta \int_{\tau }^{t}{\left( 3{{\left\| {{\phi }_{7}}\left( s \right) \right\|}_{{{L}^{\infty }}\left( {{\mathbb{R}}^{n}} \right)}}{{\left\| {{y}_{n}}\left( s \right) \right\|}^{2}}+\mathbb{E}\left( {{\left\| {{y}_{n}}\left( s \right) \right\|}^{2}} \right){{\left\| {{\phi }_{7}}\left( s \right) \right\|}_{{{L}^{1}}\left( {{\mathbb{R}}^{n}} \right)}} \right)ds}.
	\end{split}
\end{align}
For the fourth term on the right-hand, by \eqref{c17} we can get
\begin{align}\label{f8}
	\begin{split}
& \beta \int_{\tau }^{t}{\left\| {{\sigma }}\left( s,u\left( s \right),{{\mathcal{L}}_{u\left( s \right)}} \right)-{{\sigma }}\left( s,{{u}_{n}}\left( s \right),{{\mathcal{L}}_{{{u}_{n}}\left( s \right)}} \right) \right\|}_{{{L}_{2}}\left( {{l}^{2}},{{L}^{2}}\left( {{\mathbb{R}}^{n}} \right) \right)}^{2}ds \\
 \le& 2\beta \left\| {{l}_{{{\sigma }_{1}}}} \right\|_{{{l}^{2}}}^{2}\int_{\tau }^{t}{\left( \left\| w \right\|_{{{L}^{\infty }}\left( {\mathbb{R}}^{n} \right)}^{2}{{\left\| {{y}_{n}}\left( s \right) \right\|}^{2}}+{{\left\| w \right\|}^{2}}\mathbb{E}\left( {{\left\| {{y}_{n}}\left( s \right) \right\|}^{2}} \right) \right)ds}.
\end{split}
\end{align}
For the fifth term  on the right-hand, we can get
\begin{align}\label{f7}
	\begin{split}
		& \alpha \sum\limits_{k=1}^{\infty }{\int_{\tau }^{t}{{{\left\| {{\delta }_{k}}{{z}_{n}}\left( s \right) \right\|}^{2}}}}ds
		\le   \alpha \left( \lambda +\left\| \delta  \right\|_{{{l}^{2}}}^{2} \right)\int_{\tau }^{t}{{{\left\| {{z}_{n}}\left( s \right) \right\|}^{2}}ds} .
	\end{split}
\end{align}
It follows from \eqref{f4}-\eqref{f8} we can get that for all  $t\ge \tau $,
\begin{align}\label{f9}
	\begin{split}
	  & \mathbb{E} \left( \beta {{\left\| {{y}_{n}}\left( t \right) \right\|}^{2}}+\alpha {{\left\| {{z}_{n}}\left( t \right) \right\|}^{2}} \right) \\
	  \le& \mathbb{E}\left(\beta {{\left\| {{\xi }_{1}}-{{\xi }_{1,n}} \right\|}^{2}}+\alpha {{\left\| {{\xi }_{2}}-{{\xi }_{2,n}} \right\|}^{2}} \right) \\
	 & +\left( 2{{\left\| {{\phi }_{4}} \right\|}_{{{L}^{\infty }}\left( \mathbb{R},{{L}^{\infty }}\left( {{\mathbb{R}}^{n}} \right) \right)}}+3{{\left\| {{\phi }_{7}} \right\|}_{{{L}^{\infty }}\left( {\mathbb{R}}^{n} \right)}} \right)\int_{\tau }^{t}{\mathbb{E}\left( \beta {{\left\| {{y}_{n}}\left( s \right) \right\|}^{2}}+\alpha {{\left\| {{z}_{n}}\left( s \right) \right\|}^{2}} \right)}ds \\
	 & +\left( {{\left\| {{\phi }_{3}} \right\|}_{{{L}^{\infty }}\left( \mathbb{R},{{L}^{\infty }}\left( {{\mathbb{R}}^{n}} \right)\cap {{L}^{\infty }}\left( {{\mathbb{R}}^{n}} \right) \right)}}+\left\| \delta  \right\|_{{{l}^{2}}}^{2} \right)\int_{\tau }^{t}{\mathbb{E}\left( \beta {{\left\| {{y}_{n}}\left( s \right) \right\|}^{2}}+\alpha {{\left\| {{z}_{n}}\left( s \right) \right\|}^{2}} \right)}ds \\
	 & +2\left\| {{L}_{\sigma }} \right\|_{{{l}^{2}}}^{2}\left( \left\| w \right\|_{{{L}^{\infty }}\left( {\mathbb{R}}^{n} \right)}^{2}+{{\left\| w \right\|}^{2}} \right)\int_{\tau }^{t}{\mathbb{E}\left( \beta {{\left\| {{y}_{n}}\left( s \right) \right\|}^{2}}+\alpha {{\left\| {{z}_{n}}\left( s \right) \right\|}^{2}} \right)}ds .
\end{split}
\end{align}
By \eqref{f9} and Gronwall's lemma, we can get for all $t\geq\tau$,
\[
\mathbb{E}\left( \beta {{\left\| {{y}_{n}}\left( t \right) \right\|}^{2}}+\alpha {{\left\| {{z}_{n}}\left( t \right) \right\|}^{2}} \right)\le\mathbb{E}\left( \alpha {{\left\| {{\xi }_{2}}-{{\xi }_{2,n}} \right\|}^{2}}+\beta {{\left\| {{\xi }_{1}}-{{\xi }_{1,n}} \right\|}^{2}} \right){{e}^{{{c}_{7}}(t-\tau )}}.
\]
Since $\mathbb{E}\left( \left\| {{\xi }_{0,n}} \right\|_{{{\mathbb{L}}^{2}}\left( {{\mathbb{R}}^{n}} \right)}^{4} \right)\le K$, we see that the sequence $\left\{ \xi _{0,n} \right\}_{n=1}^{\infty }$ is uniformly integrable in $L_{{{\mathcal{F}}_{\tau }}}^{2}\left( \Omega ,{{\mathbb{L}}^{2}}\left( {{\mathbb{R}}^{n}} \right) \right)$. Then using the assumption that $\xi _{0,n}\to {{\xi }_{0}}$ $\mathbb{P}$-almost surely, we obtain from Vitali's theorem that $\xi _{0,n}\to {{\xi }_{0}}$ in $L_{{{\mathcal{F}}_{\tau }}}^{2}\left( \Omega ,{{\mathbb{L}}^{2}}\left( {{\mathbb{R}}^{n}} \right) \right)$, which along with (6.9) shows that ${{k}}\left( t,\tau ,\xi _{0,n} \right)\to {{k}}\left( t,\tau ,{{\xi }_{0}} \right)$ in $L_{{{\mathcal{F}}_{\tau }}}^{2}\left( \Omega ,{{\mathbb{L}}^{2}}\left( {{\mathbb{R}}^{n}} \right) \right)$ and hence also in distribution.
\end{proof}
Next we will show that the system \eqref{c18}-\eqref{c19} has a closed  $\mathcal{D}$-pullback absorbing set.
\begin{lem}\label{lem6.2}
	Suppose assumption  $\mathbf{(H_{1})-({H_{4}})}$, \eqref{c21} and \eqref{c25} hold, then   $S$  has a closed  $\mathcal{D}$-pullback absorbing set $B\in \mathcal{D}$ which is given by, for each $\tau\in \mathbb{R}$,
	\begin{align}\label{f10}
		\begin{split}
		B\left( \tau  \right)=\left\{ \mu\in \mathcal{P}_{4}({{\mathbb{L}}^{2}}\left( {{\mathbb{R}}^{n}} \right)):\int_{{{\mathbb{L}}^{2}}\left( {{\mathbb{R}}^{n}} \right)}{{{\left\| \xi  \right\|}^{4}}\mu \left( d\xi  \right)}\le R\left( \tau  \right) \right\},
	\end{split}
\end{align}
	where $R\left( \tau  \right) $  is given by
	\begin{align}
		\label{f11}
		\begin{split}
			R\left( \tau  \right)={{M}_{5}}+{{M}_{5}}\int_{-\infty }^{\tau }{{{e}^{2\eta \left( s-\tau  \right)}}\left( \left\| {{\phi }_{g}}\left( s \right)   \right\|_{{{L}^{2}}\left( \mathbb{R}^{n},l^{2} \right)}^{4}+\left\| {{\theta }_{1}}\left( s \right)   \right\|_{{{L}^{2}}\left( \mathbb{R}^{n},l^{2} \right)}^{4}+\left\| {{\theta }_{2}}\left( s \right)   \right\|_{{{L}^{2}}\left( \mathbb{R}^{n},l^{2} \right)}^{4} \right)}ds,
		\end{split}
	\end{align}
	where  $\eta>0$ is the same number as in  \eqref{c21} and $M_{5}>0$ is the same constant  in Lemma 5.6 which is independent of   $\tau$.
\end{lem}
\begin{proof}
	Obviously $B\left( \tau  \right)$  is a closed subset of $\mathcal{P}_{4}(\mathbb{L}^{2}(\mathbb{R}^{n}))$. By \eqref{f2} and Lemma 5.6, we can get  that for every $\tau \in \mathbb{R}$ and
	$D_{2}\in \mathcal{D}$ there exists $T=T\left( \tau ,D_{2} \right)>0$  such that for  all $t\ge T$,
	\begin{align}
		\label{f12}
		\begin{split}
			{{S}}\left( t,\tau -t \right)D_{2}\left( \tau -t \right)\subseteq B\left( \tau  \right).
		\end{split}
	\end{align}
	Finally, by \eqref{c25} and \eqref{f10} we have, as $\tau \to -\infty$,
	\begin{align*}
		 &{{e}^{2\eta \tau }}\left\| B\left( \tau  \right) \right\|_{{{\mathcal{P}}_{4}}\left( {{\mathbb{L}}^{2}}\left( {{\mathbb{R}}^{n}} \right) \right)}^{4}\\
		 =&{{e}^{2\eta \tau }}{{M}_{6}} +{{M}_{6}}\int_{-\infty }^{\tau }{{{e}^{2\eta s}}\left( \left\| {{\phi }_{g}}\left( s \right) \right\|_{{{L}^{2}}\left( {{\mathbb{R}}^{n}},{{l}^{2}} \right)}^{4}+\left\| {{\theta }_{1}}\left( s \right) \right\|_{{{L}^{2}}\left( {{\mathbb{R}}^{n}},{{l}^{2}} \right)}^{4}+\left\| {{\theta }_{2}}\left( s \right) \right\|_{{{L}^{2}}\left( {{\mathbb{R}}^{n}},{{l}^{2}} \right)}^{4} \right)}ds\to 0,
	\end{align*}
	which shows that $B=\left\{ B\left( \tau  \right):\tau \in \mathbb{R} \right\}\in \mathcal{D}$. Therefore, by \eqref{f12} we find that $B$  as given by \eqref{f10} is a closed $\mathcal{D}$-pullback absorbing set for ${{S}}$.
\end{proof}

Let $\chi^n=\chi_1^n+\chi_2^n$, for $n\in \N$, where
$\{\chi^n\}_{n=1}^\infty$, $\{\chi^n_1\}_{n=1}^\infty$
and $\{\chi^n_2\}_{n=1}^\infty$ are $X$-valued random sequence.
\begin{lem}\cite{LW24}\label{Lch}
Assume  $\{\mathcal L(\chi_1^n)\}_{n=1}^\infty$ are tight and $\E(\|\chi_2^n\|_X^2)$
convergence to zero. Then $\{\mathcal L(\chi^n)\}_{n=1}^\infty$ are also tight.
\end{lem}

Given $\tau \in \mathbb{R}$, $D_{2}\in {{\mathcal{D}}}$, ${{t}_{n}}\to \infty $ and ${{\mu }_{n}}\in D_{2}\left( \tau -{{t}_{n}} \right)$, hereafter we will prove the sequence $\left\{ {{S}}\left( {{t}_{n}},\tau -{{t}_{n}} \right){{\mu }_{n}} \right\}_{n=1}^{\infty }$ has a convergent subsequence in $(\mathcal{P}_{4}({{\mathbb{L}}^{2}}\left( {{\mathbb{R}}^{n}} \right)),d_{\mathcal{P}({{\mathbb{L}}^{2}}\left( {{\mathbb{R}}^{n}} \right))})$.
\begin{lem}\label{Tem5.3}
	Suppose assumption  $(H_{1})$-$(H_{4})$,  \eqref{c21} and \eqref{c24}-\eqref{c25} hold. Then $S$ is $\mathcal{D}$-pullback asymptotically compact in $(\mathcal{P}_{4}({{\mathbb{L}}^{2}}\left( {{\mathbb{R}}^{n}} \right)),d_{\mathcal{P}({{\mathbb{L}}^{2}}\left( {{\mathbb{R}}^{n}} \right))})$.
\end{lem}
\begin{proof}
Denote ${{\xi }_{0,n}}\in ~L_{{{\mathcal{F}}_{\tau -{{t}_{n}}}}}^{4}\left( \Omega ,{{\mathbb{L}}^{2}}\left( {{\mathbb{R}}^{n}} \right) \right)$ with ${{\mathcal{L}}_{{{\xi }_{0,n}}}}={{\mu }_{n}}$, we consider the solution $k\left( \tau ,\tau -{{t}_{n}},{{\xi }_{0,n}} \right)$  of \eqref{c18}-\eqref{c19} with initial data ${{\xi }_{0,n}}$  at initial time $\tau -{{t}_{n}}$. We first prove the distributions of  $\left\{ k\left( \tau ,\tau -{{t}_{n}},{{\xi }_{0,n}} \right) \right\}_{n=1}^{\infty }$ are tight in ${{\mathbb{L}}^{2}}\left({{\mathbb{R}}^{n}} \right)$.
	
	 Let  $\rho :{{\mathbb{R}}^{n}}\to [0,1]$  be the smooth cut-off function given by \eqref{d25}, and ${{\rho }_{m}}\left( x \right)=\rho \left( \frac{x}{m} \right)$ for every $m\in \mathbb{N}$ and $x\in {{{L}}^{2}}\left( {{\mathbb{R}}^{n}} \right)$. Then the solution $k$  can be decomposed as
	 \[
	 k\left( \tau ,\tau -{{t}_{n}},{{\xi }_{0,n}} \right)={{\rho }_{m}}k\left( \tau ,\tau -{{t}_{n}},{{\xi }_{0,n}} \right)+\left( 1-{{\rho }_{m}} \right)k\left( \tau ,\tau -{{t}_{n}},{{\xi }_{0,n}} \right).
	 \]
Further we depose the $\left( 1-{{\rho }_{m}} \right)k\left( \tau ,\tau -{{t}_{n}},{{\xi }_{0,n}} \right)$ as
 \begin{align} \label{f13}
	\begin{split}
\left( 1-{{\rho }_{m}} \right)k\left( \tau ,\tau -{{t}_{n}},{{\xi }_{0,n}} \right)=\left( 1-{{\rho }_{m}} \right){{k}_{1}}\left( \tau ,\tau -{{t}_{n}},{{\xi }_{0,n}} \right)+\left( 1-{{\rho }_{m}} \right){{k}_{2}}\left( \tau ,\tau -{{t}_{n}},{{\xi }_{0,n}} \right),
\end{split}
\end{align}
where
\begin{align} \label{f14}
	\begin{split}
\left( 1-{{\rho }_{m}} \right){{k}_{1}}\left( \tau ,\tau -{{t}_{n}},{{\xi }_{0,n}} \right)=\left( 1-{{\rho }_{m}} \right)\left( 0,{{v}_{1}}\left( \tau ,\tau -{{t}_{n}},{{\xi }_{0,n}} \right) \right),
\end{split}
\end{align}
and
\begin{align} \label{f15}
	\begin{split}
\left( 1-{{\rho }_{m}} \right){{k}_{2}}\left( \tau ,\tau -{{t}_{n}},{{\xi }_{0,n}} \right)=\left( 1-{{\rho }_{m}} \right)\left( u\left( \tau ,\tau -{{t}_{n}},{{\xi }_{0,n}} \right),{{v}_{2}}\left( \tau ,\tau -{{t}_{n}},0 \right) \right).
\end{split}
\end{align}
By \eqref{d19} we can get
\begin{align} \label{f16}
	\begin{split}
\mathbb{E}\left( {{\left\| \left( 1-{{\rho }_{m}} \right){{k}_{1}}\left( \tau ,\tau -{{t}_{n}},{{\xi }_{0,n}} \right) \right\|}^{2}} \right)\to 0,\quad\text{as}\ n\to \infty.
\end{split}
\end{align}
Then from \eqref{f13}-\eqref{f16} and  Lemma \ref{Lch} we can get
 $\left\{ {{\mathcal{L}}_{\left( 1-{{\rho }_{m}} \right)k\left( \tau ,\tau -{{t}_{n}},{{\xi }_{0,n}} \right)}} \right\}_{n=1}^{\infty }$  is tight as along as  $\left\{ {{\mathcal{L}}_{\left( 1-{{\rho }_{m}} \right){{k}_{2}}\left( \tau ,\tau -{{t}_{n}},{{\xi }_{0,n}} \right)}} \right\}_{n=1}^{\infty }$ is tight. By Lemma 5.3 and Lemma 5.4 we can get that for every $\tau\in\mathbb{R}$ and ${{D}_{2}}=\left\{ {{D}_{2}}\left( t \right):t\in \mathbb{R} \right\}\in \mathcal{D}$ there exist $N_{5}=N_{5}\left( \tau ,{{D}_{2}} \right)>0$ such that for all $n\ge N_{5}$,
 \begin{align} \label{f17}
 	\begin{split}
 \mathbb{E}\left( \left\| u\left( \tau ,\tau -{{t}_{n}},{{\xi }_{1,n}} \right) \right\|_{{{H}^{1}}\left( {{\mathbb{R}}^{n}} \right)}^{2}+\left\| {{v}_{2}}\left( \tau ,\tau -{{t}_{n}},0 \right) \right\|_{{{H}^{1}}\left( {{\mathbb{R}}^{n}} \right)}^{2} \right)\le {{c}_{8}}.
\end{split}
\end{align}
By \eqref{f17} we can get that for all $m\in\mathbb{N}$ and $n\geq N_{5}$,
\begin{align} \label{f18}
	\begin{split}
\mathbb{E}\left( \left\| \left( 1-{{\rho }_{m}} \right)k_{2}\left( \tau ,\tau -{{t}_{n}},{{\xi }_{0,n}} \right) \right\|_{{{\mathbb{H}}^{1}}\left( {{\mathbb{R}}^{n}} \right)}^{2} \right)\le {{c}_{9}}.
\end{split}
\end{align}
By \eqref{f18} and Chebyshev's inequality we can get that there exists $R(\epsilon,\tau)>0$ such that for all $m\in\mathbb{N}$ and $n\geq N_{5}$,
\begin{align} \label{f19}
	\begin{split}
\mathbb{P}\left( \left\{ {{\left\| \left( 1-{{\rho }_{m}} \right)k_{2}\left( \tau ,\tau -{{t}_{n}},{{\xi }_{0,n}} \right) \right\|}_{{{\mathbb{H}}^{1}}\left( {{\mathbb{R}}^{n}} \right)}}>R(\epsilon,\tau) \right\} \right)\le \epsilon.
\end{split}
\end{align}
	Denote
	\[{{Z}_{\epsilon }}=\left\{ k_{2}\in {{\mathbb{H}}^{1}}\left( {{\mathbb{R}}^{n}} \right):{{\left\|  k_{2} \right\|}_{{{\mathbb{H}}^{1}}\left( {{\mathbb{R}}^{n}} \right)}}\le R\left( \epsilon ,\tau  \right); k_{2}\left( x \right)=0\ \text{for a}\text{.e}\text{.}\ \left| x \right|>m \right\}.\]
	 Then ${{Z}_{\epsilon }}$ is a compact subset of ${{\mathbb{L}}^{2}}\left( {{\mathbb{R}}^{n}} \right)$, by \eqref{f19} we can get that for all $m\in \mathbb{N}$ and $n\ge {{N}_{5}}$,
	 \begin{align} \label{f20}
	 	\begin{split}
	 		\mathbb{P}\left( \left\{ {{\left\| \left( 1-{{\rho }_{m}} \right)k_{2}\left( \tau ,\tau -{{t}_{n}},\xi_{0,n} \right) \right\|}_{{\mathbb{H}^{1}}\left( {{\mathbb{R}}^{n}} \right)}}\in {{Z}_{\epsilon }} \right\} \right)>1-\epsilon.
	 	\end{split}
	 \end{align}
	Note that $\delta >0$ is arbitrary, by \eqref{f20} we can get that for all $m\in \mathbb{N}$ and $n\ge {{N}_{5}}$,
	\begin{align} \label{f21}
		\begin{split}
			\left\{ {{\mathcal{L}}_{\left( 1-{{\rho }_{m}} \right)k_{2}\left( \tau ,\tau -{{t}_{n}},\xi_{0,n} \right)}} \right\}_{n=1}^{\infty } \quad\text{is tight in} \quad{{\mathbb{L}}^{2}}\left( {{\mathbb{R}}^{n}} \right).
		\end{split}
	\end{align}
	Further we can get that the sequence $\left\{ {{\mathcal{L}}_{\left( 1-{{\rho }_{m}} \right)k\left( \tau ,\tau -{{t}_{n}},{{\xi }_{0,n}} \right)}} \right\}_{n=1}^{\infty }$  is tight in ${{\mathbb{L}}^{2}}\left( {{\mathbb{R}}^{n}} \right)$. Next we will use the uniform tail-estimates to prove  $\left\{ {{\mathcal{L}}_{k\left( \tau ,\tau -{{t}_{n}},{{\xi }_{0,n}} \right)}} \right\}_{n=1}^{\infty } $ also is tight in ${{\mathbb{L}}^{2}}\left( {{\mathbb{R}}^{n}} \right)$. According to Lemma 5.4 we can get that for every $\varepsilon >0$, there exists ${{N}_{6}}={{N}_{6}}\left( \varepsilon ,\tau ,D_{2} \right)\in \mathbb{N}$ and  ${{m}_{0}}={{m}_{0}}\left( \varepsilon ,\tau  \right)\in \mathbb{N}$ such that for all $n\ge {{N}_{6}}$,
		\begin{align} \label{f22}
		\begin{split}
			\mathbb{E}\left( \int_{\left| x \right|>{{m}_{0}}}{{{\left\vert k\left( \tau ,\tau -{{t}_{n}},{{\xi }_{0,n}} \right)\left( x \right) \right\vert}^{2}}dx} \right)<\frac{1}{9}{{\varepsilon }^{2}}.
		\end{split}
	\end{align}
	By \eqref{f22} we can get that for all $m\in \mathbb{N}$ and $n\ge {{N}_{6}}$,
	\begin{align} \label{f23}
		\begin{split}
			\mathbb{E}\left( {{\left\| {{\rho }_{{{m}_{0}}}}k\left( \tau ,\tau -{{t}_{n}},{{\xi }_{0,n}} \right) \right\|}^{2}} \right)<\frac{1}{9}{{\varepsilon }^{2}}.
		\end{split}
	\end{align}
	By \eqref{f23} we can see that $\left\{ {{\mathcal{L}}_{\left( 1-{{\rho }_{{{m}_{0}}}} \right)k\left( \tau ,\tau -{{t}_{n}},{{\xi }_{0,n}} \right)}} \right\}_{n={{N}_{6}}}^{\infty }$ is tight in ${{\mathbb{L}}^{2}}\left( {\mathbb{R}}^{n} \right)$, and hence there exist ${{n}_{1}},\cdots {{n}_{l}}\ge {{N}_{6}}$ such that
	\begin{align} \label{f24}
		\begin{split}
		 \left\{ {{\mathcal{L}}_{\left( 1-{{\rho }_{{{m}_{0}}}} \right)k\left( \tau ,\tau -{{t}_{n}},{{\xi }_{0,n}} \right)}} \right\}_{n={{N}_{9}}}^{\infty }\subseteq \bigcup\limits_{j=1}^{l}{B\left( {{\mathcal{L}}_{\left( 1-{{\rho }_{{{m}_{0}}}} \right)k\left( \tau ,\tau -{{t}_{{{n}_{j}}}},{{\xi }_{0,n_{j}}} \right)}},\frac{1}{3}\varepsilon  \right)},
		\end{split}
	\end{align}
	where B $\left( {{\mathcal{L}}_{\left( 1-{{\rho }_{{{m}_{0}}}} \right)k\left( \tau ,\tau -{{t}_{{{n}_{j}}}},{{\xi }_{0,n_{j}}} \right)}},\frac{1}{3}\varepsilon  \right)$ is the $\frac{1}{3}\varepsilon$ -neighborhood of ${{\mathcal{L}}_{\left( 1-{{\rho }_{{{m}_{0}}}} \right)k\left( \tau ,\tau -{{t}_{{{n}_{j}}}},{{\xi }_{0,n_{j}}} \right)}}$ in the space $\left( {{\mathcal{P}}_{4}}({{\mathbb{L}}^{2}}\left({\mathbb{R}}^{n} \right)),{{d}_{\mathcal{P}({{\mathbb{L}}^{2}}\left( {\mathbb{R}}^{n} \right))}} \right)$. Given $n\ge {{N}_{6}}$, by \eqref{f24} we can get that there exist $j\in \left\{ 1,2,\cdots ,l \right\}$ such that
	\begin{align} \label{f25}
		\begin{split}
			{{\mathcal{L}}_{\left( 1-{{\rho }_{{{m}_{0}}}} \right)k\left( \tau ,\tau -{{t}_{n}},{{\xi }_{0,n}} \right)}}\subseteq B\left( {{\mathcal{L}}_{\left( 1-{{\rho }_{{{m}_{0}}}} \right)k\left( \tau ,\tau -{{t}_{{{n}_{j}}}},{{\xi }_{0,n_{j}}} \right)}},\frac{1}{3}\varepsilon  \right).
		\end{split}
	\end{align}
	By \eqref{f23} and \eqref{f25} we can get that
	\begin{align}\label{f26}
		\begin{split}
		& {{d}_{\mathcal{P}\left( Z \right)}}\left( {{\mathcal{L}}_{k\left( \tau ,\tau -{{t}_{n}},{{\xi }_{0,n}} \right)}},{{\mathcal{L}}_{k\left( \tau ,\tau -{{t}_{{{n}_{j}}}},{{\xi }_{0,n_{j}}} \right)}} \right)\\
		=&\underset{\psi \in {{L}_{b}}\left( Z \right),{{\left\| \psi  \right\|}_{{{L}_{b}}}}\le 1}{\mathop{\sup }}\,\left\vert \int_{\mathbb{R}^{n}}{\psi d}{{\mathcal{L}}_{k\left( \tau ,\tau -{{t}_{n}},{{\xi }_{0,n}} \right)}}-\int_{\mathbb{R}^{n}}{\psi d}{{\mathcal{L}}_{k\left( \tau ,\tau -{{t}_{{{n}_{j}}}},{{\xi }_{0,n_{j}}} \right)}} \right\vert \\
		 =&\underset{\psi \in {{L}_{b}}\left( Z \right),{{\left\| \psi  \right\|}_{{{L}_{b}}}}\le 1}{\mathop{\sup }}\,\left\vert \mathbb{E}\left( \psi \left( k\left( \tau ,\tau -{{t}_{n}},{{\xi }_{0,n}} \right) \right) \right)-\mathbb{E}\left( \psi \left( k\left( \tau ,\tau -{{t}_{n}},{{\xi }_{0,n_{j}}} \right) \right) \right) \right\vert \\
		 \le& \underset{\psi \in {{L}_{b}}\left( Z \right),{{\left\| \psi  \right\|}_{{{L}_{b}}}}\le 1}{\mathop{\sup }}\,\left\vert \mathbb{E}\left( \psi \left( k\left( \tau ,\tau -{{t}_{n}},{{\xi }_{0,n}} \right) \right) \right)-\mathbb{E}\left( \psi \left( 1-{{\rho }_{{{m}_{0}}}} \right)k\left( \tau ,\tau -{{t}_{n}},{{\xi }_{0,n}} \right) \right) \right\vert \\
		& +\underset{\psi \in {{L}_{b}}\left( Z \right),{{\left\| \psi  \right\|}_{{{L}_{b}}}}\le 1}{\mathop{\sup }}\,\left\vert \mathbb{E}\left( \psi \left( 1-{{\rho }_{{{m}_{0}}}} \right)k\left( \tau ,\tau -{{t}_{n}},{{\xi }_{0,n}} \right) \right)-\mathbb{E}\left( \psi \left( 1-{{\rho }_{{{m}_{0}}}} \right)k\left( \tau ,\tau -{{t}_{{{n}_{j}}}},{{\xi }_{0,n_{j}}} \right) \right) \right\vert \\
		& +\underset{\psi \in {{L}_{b}}\left( Z \right),{{\left\| \psi  \right\|}_{{{L}_{b}}}}\le 1}{\mathop{\sup }}\,\left\vert \mathbb{E}\left( \psi \left( 1-{{\rho }_{{{m}_{0}}}} \right)k\left( \tau ,\tau -{{t}_{{{n}_{j}}}},{{\xi }_{0,n_{j}}} \right) \right)-\mathbb{E}\left( \psi \left( k\left( \tau ,\tau -{{t}_{{{n}_{j}}}},{{\xi }_{0,n_{j}}} \right) \right) \right) \right\vert \\
		 \le& \mathbb{E}\left( \left\| {{\rho }_{{{m}_{0}}}}k\left( \tau ,\tau -{{t}_{n}},{{\xi }_{0,n}} \right) \right\| \right)+\mathbb{E}\left( \left\| {{\rho }_{{{m}_{0}}}}k\left( \tau ,\tau -{{t}_{{{n}_{j}}}},{{\xi }_{0,n_{j}}} \right) \right\| \right) \\
		& +{{d}_{\mathcal{P}\left( Z \right)}}\left( {{\mathcal{L}}_{\left( 1-{{\rho }_{{{m}_{0}}}} \right)k\left( \tau ,\tau -{{t}_{n}},{{\xi }_{0,n}} \right)}},{{\mathcal{L}}_{\left( 1-{{\rho }_{{{m}_{0}}}} \right)k\left( \tau ,\tau -{{t}_{{{n}_{j}}}},{{\xi }_{0,n_{j}}} \right)}} \right) \\
		 <&\frac{1}{3}\varepsilon +\frac{1}{3}\varepsilon +\frac{1}{3}\varepsilon =\varepsilon.
		 \end{split}
	\end{align}
	Consequently we can get
	\begin{align} \label{f27}
		\begin{split}
			\left\{ {{\mathcal{L}}_{\left( 1-{{\rho }_{{{m}_{0}}}} \right)k\left( \tau ,\tau -{{t}_{n}},{{\xi }_{0,n}} \right)}} \right\}_{n={{N}_{6}}}^{\infty }\subseteq \bigcup\limits_{j=1}^{l}{B\left( {{\mathcal{L}}_{\left( 1-{{\rho }_{{{m}_{0}}}} \right)k\left( \tau ,\tau -{{t}_{{{n}_{j}}}},{{\xi }_{0,n_{j}}} \right)}},\varepsilon  \right)}.
		\end{split}
	\end{align}
	 Since $\varepsilon >0$  is arbitrary, by \eqref{f27} we see that the sequence $\left\{ {{\mathcal{L}}_{k\left( \tau ,\tau -{{t}_{n}},{{\xi }_{0,n}} \right)}} \right\}_{n=1}^{\infty }$ is tight in $\mathcal{P}({{\mathbb{L}}^{2}}\left( {\mathbb{R}}^{n} \right))$, which implies that there exists $\varpi \in \mathcal{P}({{\mathbb{L}}^{2}}\left( {\mathbb{R}}^{n} \right))$ such that, up to a subsequence,
		\begin{align} \label{f28}
		\begin{split}
			{{\mathcal{L}}_{k\left( \tau ,\tau -{{t}_{n}},{{\xi }_{0,n}} \right)}}\to \varpi\quad\text{weakly}.
		\end{split}
	\end{align}
	It remains to show $\varpi \in \mathcal{P}({{\mathbb{L}}^{2}}\left( {\mathbb{R}}^{n} \right))$. Let $B=\{B\left( \tau  \right):\tau \in \mathbb{R}\}$ be the closed ${{\mathcal{D}}}$-pullback absorbing set of ${{S}}$  given by \eqref{f10}. Then there exists ${{N}_{7}}={{N}_{7}}\left( \tau ,D_{2} \right)\in \mathbb{N}$ such that for all $n\ge {{N}_{7}}$,
	 \begin{align} \label{f29}
	 	\begin{split}
	 		{{\mathcal{L}}_{k\left( \tau ,\tau -{{t}_{n}},{{\xi }_{0,n}} \right)}}\in B\left( \tau  \right).
	 	\end{split}
	 \end{align}
	Since $B\left( \tau  \right)$  is closed with respect to the weak topology of $\mathcal{P}({{\mathbb{L}}^{2}}\left( \mathbb{R}^{n} \right))$, by \eqref{f28}-\eqref{f29} we obtain $\varpi \in B\left( \tau  \right)$ and thus $\varpi \in {{\mathcal{P}}_{4}}({{\mathbb{L}}^{2}}\left( \mathbb{R}^{n} \right))$. This completes the proof.
\end{proof}
 We now prove the existence and uniqueness of ${{\mathcal{D}}}$-pullback measure attractors of ${{S}}$.
\begin{thm}
	If $\mathbf{(H1)-(H4)}$, \eqref{c21} and \eqref{c24}-\eqref{c25} hold, then ${{S}}$ associated with \eqref{c18}-\eqref{c19} has a unique ${{\mathcal{D}}}$-pullback measure attractor $\mathcal{A}=\{\mathcal{A}\left( \tau  \right):\tau \in \mathbb{R}\}\in {{\mathcal{D}}}$ in $\mathcal{P}_{4}\left( \mathbb{L}^{2}(\mathbb{R}^{n})\right)$ such that for all $\tau\in\mathbb{R}$,
	\begin{align*}
		 \mathcal{A}\left( \tau  \right)=k\left( B,\tau  \right)&=\left\{ \psi \left( 0,\tau  \right):\psi\, \text{is a}\ \mathcal{D}\text{-complete orbit of}\ S \right\} \\
		& =\left\{ \xi \left( \tau  \right):\xi \,\text{is a}\ \mathcal{D}\text{-complete solution of}\ S \right\},
	\end{align*}
	where ${{B}}=\left\{ {{B}}\left( \tau  \right):\tau \in \mathbb{R} \right\}$ is a ${{\mathcal{D}}}$-pullback absorbing of $S$ as given by  Lemma \ref{lem6.2}.
\end{thm}
\begin{proof}
	We can get the existence and uniqueness of the measure attractor $\mathcal{A}$ follows
from Proposition \ref{P1} based on  Lemma \ref{lem6.1}, \ref{lem6.2} and \ref{Tem5.3}.
\end{proof}

\end{document}